\documentclass{amsart} 
\usepackage[latin1]{inputenc}
\usepackage{amsfonts,amsthm,amssymb,latexsym,amsmath}
\usepackage[all,cmtip]{xy}
\usepackage{bbm}
\usepackage{yfonts}
\usepackage{multirow}

\usepackage{xcolor}
\usepackage{stackrel}

\newtheorem{Teo}{Theorem}
\newtheorem{Main}{Main Theorem}
\newtheorem{Lema}[Teo]{Lemma}
\newtheorem{Prop}[Teo]{Proposition}

\newtheorem{Cor}[Teo]{Corollary}

\newtheorem{Def}[Teo]{Definition}

\newcommand{\h}{\mathrm{H}}

\newcommand{\calb}{\mathcal{B}}
\newcommand{\calc}{\mathcal{C}}
\newcommand{\cald}{{\mathcal D}}
\newcommand{\cale}{{\mathcal E}}
\newcommand{\calg}{\mathcal{G}}
\newcommand{\calf}{\mathcal{F}}
\newcommand{\calh}{\mathcal{H}}
\newcommand{\cali}{{\mathcal I}}
\newcommand{\calk}{{\mathcal K}}
\newcommand{\call}{{\mathcal L}}

\newcommand{\calo}{{\mathcal O}}
\newcommand{\calp}{{\mathcal P}}
\newcommand{\calq}{{\mathcal Q}}
\newcommand{\calr}{{\mathcal R}}
\newcommand{\calu}{{\mathcal U}}
\newcommand{\caly}{{\mathcal Y}}
\newcommand{\simto}{\stackrel{\sim}{\rightarrow}}

\newcommand{\lext}{{\mathcal E}{\it xt}}
\newcommand{\lend}{{\mathcal E}{\it nd}}
\newcommand{\lhom}{{\mathcal H}{\it om}}

\DeclareMathOperator{\Hom}{Hom}
\DeclareMathOperator{\Ext}{Ext}
\DeclareMathOperator{\Ph}{\mathrm \Pi}

\DeclareMathOperator{\coker}{coker}

\DeclareMathOperator{\rk}{{rk}}
\DeclareMathOperator{\im}{im}
\newcommand{\Ima}{\im}

\newcommand{\PP}{\mathbb{P}^3}
\newcommand{\op}[1]{{\mathcal O}_{\mathbb{P}^{#1}}}
\newcommand{\p}[1]{{\mathbb{P}^{#1}}}

\usepackage[
	pdfauthor={Charles}
	pdftitle={Report I}
	pdfdisplaydoctitle=true,
	pdfstartview={FitH},
	bookmarks=false,
]{hyperref}

\begin{document}

\begin{flushright}
\textbf{UDK 512.7}
\end{flushright}

\vspace{1cm}

\title[New moduli components of rank 2 bundles]{New moduli 
components of rank 2 bundles on projective space}

\author[C. Almeida]{Charles Almeida}
\address{ICEx - UFMG \\ 
Departamento de Matematica \\ 
Av. Antonio Carlos, 
6627, 30123-970 Belo Horizonte MG, Brazil}
\email{charlesalmeida@mat.ufmg.br}

\author[M. Jardim]{Marcos Jardim}
\address{IMECC - UNICAMP \\
Departamento de Matem\'atica \\ 
Rua S\'ergio  Buarque de Holanda, 651  \\
13083-970 Campinas-SP, Brazil}
\email{jardim@ime.unicamp.br}

\author[A. Tikhomirov]{Alexander Tikhomirov}
\address{Faculty of Mathematics\\
National Research University  
Higher School of Economics\\
6 Usacheva Street\\ 
119048 Moscow, Russia}
\email{astikhomirov@mail.ru}

\author[S. Tikhomirov]{Sergey Tikhomirov}
\address{Department of Mathematics\\
Yaroslavl State Pedagogical University named after K.D.Ushinskii\\
108 Respublikanskaya Street\\ 
150000 Yaroslavl, Russia}

\begin{abstract}
We present a new family of monads whose cohomology is a 
stable rank two vector bundle on $\PP$.  We also study the 
irreducibility and smoothness together with a geometrical 
description of some of these families. 
These facts are used to construct a new infinite series of 
rational moduli components of stable rank two vector 
bundles with trivial determinant and growing second Chern
class. We also prove that the moduli space of stable rank 
two vector bundles with trivial determinant and second Chern 
class equal to 5 has exactly three irreducible rational components. 
 
\noindent{\bf 2010 MSC:} 14D20, 14J60

\noindent{\bf Keywords:} Rank 2 bundles, Monads, Instanton bundles
\end{abstract}

\maketitle

\section{Introduction}\label{Intro}

In \cite{Maruyama} Maruyama proved that the rank $r$ stable 
vector bundles on a projective scheme $X$ with fixed Chern 
classes $c_1,...,c_r $ can be parametrized by an algebraic 
quasi-projective variety, denoted by 
$\mathcal{B}_{X}(r,c_1,...,c_r)$. Although this result has 
been known for almost 40 years, there are just a few 
concrete examples and established facts about such 
varieties, even for cases like $X = \PP$ and $r=2$. For 
instance, $\mathcal{B}_{\PP}(2,0,1)$ was studied by  Barth 
in \cite{Barth1}, $\mathcal{B}_{\PP}(2,0,2)$ was described 
by Harthorne in \cite{Hart1}, $\mathcal{B}_{\PP}(2,-1,2)$ 
studied by Harthorne and Sols in \cite{Hart2} and by 
Manolache in \cite{M}, while $\mathcal{B}_{\PP}(2,-1,4)$ was 
described by Banica and Manolache in \cite{Banica}. This 
probably happened due to the fact that the questions of 
irreducibility (solved in \cite{Tikhomirov1} and 
\cite{Tikhomirov2}), and smoothness (answered in 
\cite{jardim2014trihyperkahler}) of the so-called 
\emph{instanton component} of the moduli space 
$\mathcal{B}_{\PP}(2,0,c_2)$ for all $c_2\in\mathbb{Z}_+$ remained open until 2014.

In this paper, we continue the study of the moduli space 
$\mathcal{B}_{\PP}(2,0,n)$, which we will simply denote by 
$\calb(n)$ from now on, with the goal of providing new 
examples of families of vector bundles, and understanding 
their geometry. It is more or less clear from the table in
\cite[Section 5.3]{hartshorne1991} that $\calb(1)$ and 
$\calb(2)$ should be irreducible, while $\calb(3)$ and 
$\calb(4)$ should have exactly two irreducible components; 
see \cite{ES} and \cite{C}, respectively, for the proof of 
the statements about $\calb(3)$ and $\calb(4)$. As for 
$\calb(5)$, a description of all its irreducible components
had been a challenge since 1980ies. In the paper, we give a 
complete answer to this problem - see Main Theorem \ref{MT2}. 

For $n\ge5$, two families of irreducible components have been 
studied, namely the \emph{instanton components}, whose generic 
point corresponds to an instanton bundle, and the \emph{Ein 
components}, whose generic point corresponds to a bundle 
given as cohomology of a monad of the form
\[
0 \to \mathcal{O}_{\mathbb{P}^3}(-c) \to 
\mathcal{O}_{\mathbb{P}^3}(-b) \oplus 
\mathcal{O}_{\mathbb{P}^3}(-a) \oplus 
\mathcal{O}_{\mathbb{P}^3}(a) \oplus 
\mathcal{O}_{\mathbb{P}^3}(b) 
\to \mathcal{O}_{\mathbb{P}^3}(c) \to 0
\]
where $b\ge a\ge 0$ and $c>a+b$. All of the components of 
$\calb(n)$ for $n\le4$ are of either of these types; here we 
focus on a new family of bundles that appear as soon as 
$n\ge5$.

More precisely, we study the set of vector bundles in 
$\calb(a^2+k)$ for each $a\ge2$ and $k\ge1$ which arise as 
cohomologies of monads of the form: 
\begin{equation}\label{Monad Gak}
0 \to \op3(-a) \oplus V_k\otimes\op3(-1) \to 
V_{2k+4}\otimes \op3 \to 
V'_k\otimes\op3(1) \oplus \op3(a) \to 0
\end{equation}
which will be denoted by $\mathcal{G}(a,k)$. We provide a 
bijection between such monads and monads of the form:
$$ 0 \to \op3(-a) \xrightarrow{\sigma} \tilde{E} 
\xrightarrow{\tau} \op3(a) \to 0 $$
where $\tilde{E}$ is a symplectic rank 4 instanton bundle of 
charge $k$. When $k=1$, these facts are used to prove our 
first main result. (See Theorem \ref{Thm 20} below.)

\begin{Main}\label{MT1}
For each $a\ge2$ not equal to 3, $\mathcal{G}(a,1)$ is a 
nonsingular dense subset of a rational irreducible 
component of $\calb(a^2+1)$ of dimension
$$ 4\cdot {{a+3}\choose 3}-a-1. $$
\end{Main}

Our second main result provides a complete description of all 
the irreducible components of $\calb(5)$.

\begin{Main} \label{MT2}
The moduli space $\calb(5)$ has exactly 3 rational irreducible 
components, namely:
the \emph{instanton component}, of dimension 37, which is
nonsingular and consists of those bundles given as cohomology of 
monads of the form 
\begin{equation}\label{invariant011}
0 \to V_{5}\otimes\op3(-1) \to V_{12}\otimes\op3 \to 
V_{5}\otimes\op3(1) \to 0 , ~~{\rm and}
\end{equation}
\begin{equation}\label{invariant021}
0 \to  V_{2}\otimes\mathcal{O}_{\PP}(-2) \to 
V_{3}\otimes\op3(-1) \oplus V_{3}\otimes\op3(1) \to 
V_{2}\otimes\op3(2) \to 0;
\end{equation}
the \emph{Ein component}, nonsingular of dimension 40, which 
consists of those bundles given as cohomology of monads of 
the form
\begin{equation}\label{invariant111}
0 \to \op3(-3) \to \op3(-2) \oplus V_{2}\otimes\op3 \oplus 
\op3(2) \to \op3(3) \to 0;
\end{equation}
the closure of the set $\calg(2,1)$, of dimension 
$37$, which consists of those bundles given as cohomology of 
monads of the form 
\begin{equation} \label{invariant121}
0\to\op3(-2)\oplus\op3(-1)\to V_{6}\otimes\op3\to \op3(1)\oplus\op3(2)
\to0 
\end{equation}
and
\begin{equation}\label{invariant131}
0\to\op3(-2)\oplus V_{2}\otimes\op3(-1)\to\op3(-1)\oplus V_{6}\otimes
\op3\oplus\op3(1)\to V'_{2}\otimes\op3(1)\oplus \op3(2) \to 0. 
\end{equation}
\end{Main}

Hartshorne and Rao proved in \cite{hartshorne1991} 
that every stable rank 2 bundle\ $\cale$ on $\p3$ with Chern 
classes $c_1(\cale)=0$ and $c_2(\cale)=5$ is the cohomology of 
one of the monads listed above. Rao showed in \cite{Rao1987} 
that bundles given as cohomology of monads of the form 
(\ref{invariant021}) lie in the closure of the family of 
instanton bundles of charge 5, which was shown to be 
irreducible firstly by Coanda, Tikhomirov and Trautmann in 
\cite{CTT}; see also \cite{Tikhomirov1}. The irreducibility 
of the family of bundles which arise as cohomology of monads 
of the form  (\ref{invariant111}) was established by Ein in 
\cite{Ein}.

The fact that the closure of $\calg(2,1)$ is an irreducible 
rational component of $\calb(5)$ is the particular case $a=2$
of Main Theorem \ref{MT1}. Finally, we show that the set of 
bundles given by the monads of the form (\ref{invariant131}) 
lies in the closure of $\calg(2,1)$.
 
We now give a breaf sketch of the contents of the paper. In 
Section \ref{section 2} we recall some general properties of 
monads and of symplectic instanton bundles on $\p3$. We especially
treat the rank 4 symplectic instantons of charge 1. Any such 
bundle $E$ is described as a middle term of an exact triple with
a rank 2 trivial bundle at the left hand and a null correlation 
rank 2 sheaf at the right hand. In Section \ref{modif} we study
the set $\calg(a,k)$ of (the isomorphism classes of) the so-called 
modified instanton bundles which are rank 2 bundles that arise as 
cohomology bundles of monads of the form \eqref{Monad Gak} with 
$a\ge2$ and $k\ge1$. We show that each modified instanton appears as 
cohomology bundle of a monad of the form 
\begin{equation}\label{symplectic E}
0\to\op3(-a)\to E\to\op3(a)\to0
\end{equation}
where $E$ is a rank 4 symplectic instanton of charge $k$. In 
the case $k=1$, this relation will be essential for the further 
constructions.

In Section \ref{G(a,1)} we study the set $\calg(a,1)$. We construct 
three families of symplectic monads of the form \eqref{symplectic E}. 
The first one is the universal family, with the base scheme $S$, of 
monads with $E$ splitting as $E=\op3^{\oplus2}\oplus N$ where $N$ is a
null correlation bundle. The second is a family, with the 
base scheme $\widetilde{S}$ containing $S$ as a dense open subset, of 
monads $E$ a general symplectic rank 4 instanton of charge 1. The 
third is a family of monads with $E$ splitting as in the first one, 
but with a new base $Y$. All the three families inherit universal
cohomology sheaves, and it is shown that the images of their 
corresponding modular morphisms to $\calb(a^2+1)$ have the same 
closure $\overline{\calg(a,1)}$ - see Propositions 
\ref{descrn of G(a,1)} and \ref{Prop 15}. In Section \ref{Rationality}
the three families mentioned above are used to prove the Main Theorem
1 - see Theorem \ref{Thm 20}.

Sections \ref{F related to E} and \ref{properties of F} are devoted
to the study of the monads of the form \eqref{invariant131}. In 
Section \ref{F related to E} we show that the cohomology sheaves
$\cale$ of
those among such monads that are not reduced to the monads of the form
\eqref{invariant121} are closely related (by two subsequent elementary
transformations - see Proposition \ref{two steps}) to rank 2 reflexive
sheaves with Chern classes $(0,2,2k)$, $0\le k\le3$. A complete
classification of the moduli components of these reflexive sheaves
performed in Section \ref{properties of F} - see Propositions 
\ref{Prop F unstable} and \ref{Prop F stable} - leads to the dimension
estimate, given in Theorem \ref{Cal H}, for the subset of the bundles 
$\cale$ specified above. It follows that this subset is not a 
component of $\calb(5)$, and we use this in Section \ref{B(5)} to
prove the Main Theorem 2.

{\bf Acknowledgements.}
CA was supported by the FAPESP grant number 
2014/08306-4, 2016/14376-0 and PNPD post-doctoral grant by 
the IME-USP from the University of S\~ao Paulo. This work 
was partially done during a visit to the 
University of Barcelona, and he is grateful for its 
hospitality. MJ is partially supported by the CNPq grant 
number 302889/2018-3 and the FAPESP Thematic Project number 
2018/21391-1; this work was partially done during a visit to 
the University of Edinburgh with the support of the FAPESP 
grant number 2016/03759-6. 
AST worked on this article within the framework of the 
Academic Fund Program at HSE University in 2020-2021 
(grant number 20-01-011) and within the framework of the 
Russian Academic Excellence Project ``5-100". AST also 
acknowledges the support from the Max Planck Institute for 
Mathematics in Bonn, where part of this work was done during 
the winter of 2020.

{\bf Notation and Conventions.}
\begin{itemize}
\item In this work, $\mathbf{k}$ is an algebraically closed 
field of characteristic zero. 

\item $V_n$, respectively, $U_n$ denotes a $\mathbf{k}$-vector 
space of dimension $n$. 

\item $\langle v\rangle$ the 1-dimensional subspace of $V_n$
spanned by a nonzero vector $v\in V_n$.

\item $\mathbf{P}(F):=\mathrm{Proj}(\mathrm{Sym}^{\bullet}
_{\calo_X}F)$ the projective spectrum of $F$, for a coherent 
$\calo_X$-sheaf $F$ on a given scheme $X$.  

\item $\calo_{\mathbf{P}(F)}(1)$ the Grothendieck sheaf on
$\mathbf{P}(F)$.

\item $\mathbf{V}(F):=\mathrm{Spec}(\mathrm{Sym}^{\bullet}
_{\calo_X}F)$, for $X$ and $F$ as above.  

\item $\p3:=P(U_4)$ the projective 3-space.

\item $\mathbf{Isom}(V_n\otimes\calo_X,F)\to X$ the 
principal $GL(n,\mathbf{k})$-bundle of frames of a rank $n$ 
locally free $\calo_X$-sheaf $F$.

\item $\mathbf{X}:=\p3\times X$, for a given scheme $X$.

\item $p_X:\mathbf{X}\to X$ the projection onto the second 
factor, for $\mathbf{X}$ and $X$ as above.

\item $\mathbf{f}:\mathbf{X}\to\mathbf{Y}$ the morphism 
induced by the morphism of schemes $f:X\to Y$.

\item $F_X:=f^*F$, $\varphi_X:=f^*\varphi:F_X\to G_X$, 
$\mathbf{E}_{\mathbf{X}}:={\mathbf{f}}^*\mathbf{E}$, for a 
given $\calo_Y$-sheaf $F$, a given morphism $\varphi:F\to G$ 
of $\calo_Y$-sheaves, a given $\calo_{\mathbf{Y}}$-sheaf (or, 
a complex of sheaves) $\mathbf{E}$ , and $f:X\to Y$ and 
$\mathbf{f}:\mathbf{X}\to\mathbf{Y}$ as above.

\item $\mathbf{E}(a,0):=\mathbf{E}\otimes\op3(a)\boxtimes
\calo_X$, for $X$ and $\mathbf{E}$ as above, and $a\in
\mathbb{Z}$.

\item $X\xleftarrow{g_X}X\times_ZY\xrightarrow{f_Y}Y$
the projections of the fibre product $X\times_ZY$ induced by
the morphisms $X\xrightarrow{f}Z\xleftarrow{g}Y$.

\item $H^i(F)$ the $i$-th cohomology group of the sheaf $F$ 
on $\p3$.

\item $Gr(n,V_k)$ the grassmannian variety of 
$n$-dimensional subspaces of $V_k$.

\item Variety means an integral (i. e., reduced and irreducible)
scheme.

\item Since we are working with rank 2 vector bundles on 
$\PP$, and Gieseker stability is equivalent to 
$\mu-$stability, we will not make any distinction between 
these two concepts.

\item We will not make any distinction between vector 
bundles and locally free sheaves.

\item $[E]$ the isomorphism class of a given sheaf on $\p3$; 
in case $\cale$ is a rank 2 stable sheaf on $\p3$, $[\cale]$ 
is also considered as a point in the moduli space $M$ of 
stable rank 2 sheaves on $\p3$.

\item $\Phi_X:X\to M,\ x\mapsto[\mathbf{E}|_{\p3\times\{x\}
}]$ the morphism defined by the $\calo_{\mathbf{X}}$-sheaf 
$\mathbf{E}$ which is a family of stable rank 2 vector bundles 
on $\p3$ with base $X$, for $M$ as above. We call $\Phi_X$ 
the \textit{modular morphism defined by the family} 
$\mathbf{E}$.

\item $\calr(e,n,m)$ the set of isomorphism classes of rank 2
reflexive sheaves on $\p3$ with Chern classes $(c_1,c_2,c_3)=
(e,n,m)$.

\item $\ell(Y):=h^0(\calo_Y)$ the length of a 0-dimensional 
scheme $Y$.

\item $\h^{1}_{*}(E)=\bigoplus_{i\in\mathbb{Z}}\h^1(E(i))$ the 
graded cohomology module over the graded ring $\Gamma_*(\op3):=
\bigoplus_{j\ge0}\h^0(\op3(j))$.

\item $(s)_0:=\{x\in X|s(x)=0\}$ the scheme of zeroes of a  
section $s$ of a given vector bundle on a scheme $X$.

\item $\mathrm{Sp}(\cale)$ the spectrum of a vector bundle 
$[\cale]\in\calb(5)$, i. e., the nondecreasing sequence of 
integers $(a_1,a_2,a_3,a_4,a_5)$ uniquely defined by $\cale$ - 
see \cite{BE}, \cite[Section 7]{Hartreflexive}.

\item All the commutative diagrams of sheaves below which do 
not contain monads are assumed to have exact rows and columns. 
In these diagrams, the arrows $F\rightarrowtail G$, resp.,
$F\twoheadrightarrow G$ are shortenings for $0\to F\to G$, 
resp., $F\to G\to0$.
\end{itemize}

\section{Monads and symplectic instanton bundles}\label{section 2}

Recall that a monad is a complex of vector bundles of the 
form:
\begin{equation}\label{example0}
\xymatrix{
0\ar[r] & A\ar[r]^{\alpha} & B \ar[r]^{\beta} & C\ar[r] & 0} \end{equation}
such that $\alpha$ is injective, and $\beta$ is surjective. 
We call the sheaf $E := \ker \beta / \Ima \alpha $ the 
cohomology of the monad (\ref{example0}). When $\alpha$ is 
locally left invertible (i. e., it is a subbundle morphism), then $E$ is a vector bundle.

The notion of monad is important in the study of vector 
bundles on $\PP$ because Horrocks proved in \cite{Ho} that 
every vector bundle on $\PP$ is cohomology of a monad of the 
form (\ref{example0}) with $A$, $B$ and $C$ being sums of 
line bundles.

For completeness, we include in this section some useful 
results about monads that will be required in this work. The 
following lemma gives a relation between isomorphism classes 
of monads and its cohomology vector bundles; a proof can be 
found in \cite[Lemma 4.1.3]{okonek}.

\begin{Lema}\label{okonek}
Let $E$ and $E^{\prime}$ be, respectively, cohomology of the 
following monads:
\begin{equation}\label{vector1}
\xymatrix{
M:\ \ \  0\ar[r] & A \ar[r]^{a} & B \ar[r]^{b} & C \ar[r] & 0},
\end{equation}
\begin{equation}\label{vector2}
\xymatrix{
M':\ \ \ 0\ar[r] & A'\ar[r]^{a'} & B'\ar[r]^{b'} & C'\ar[r] &0}.
\end{equation}
If one has that 
$\mathrm{Hom}(B,A^{\prime}) = \mathrm{Hom}(C,B^{\prime}) = 
\Ext^1(C, A^{\prime})=\Ext^{1}(B,A^{\prime}) = 
\Ext^{1}(C, B^{\prime})= \Ext^{2}(C,A^{\prime})=0,$
then there exists a bijection between the set of all 
morphisms from $E$ to $E^{\prime}$ and the set of all 
morphisms of monads from (\ref{vector1}) to (\ref{vector2}). 
\end{Lema}

The following important corollary will be used several times 
in what follows, and a proof can also be found in 
\cite[Lemma 4.1.3, Corollary 2]{okonek}.

\begin{Cor}\label{Okoneksimpletico}
Consider the monad $M$ and its dual monad $M^{\vee}$, where: 
$$
\xymatrix{M:\ \ \ 0\ar[r]& A \ar[r]^{a} & B \ar[r]^{b} & C \ar[r] & 
0},\ \ \ \ \ \ \ \ \ \ \ \ \ \ \ \ \ \ \ \ \xymatrix{
M^{\vee}:\ \ \ 0\ar[r] & C^{\vee} \ar[r]^{b^{\vee}} & B^{\vee} 
\ar[r]^{a^{\vee}} & A^{\vee} \ar[r] & 0.}
$$
If these monads satisfy the hypothesis of Lemma 
\ref{okonek}, and there exists an isomorphism $f : E \to 
E^{\vee}$ between its cohomology bundles such that $f^{\vee} 
=-f$, then there are isomorphisms $h: C \to A^{\vee}$, and 
$q: B\to B^{\vee}$, such that $q^{\vee} = -q$, and $h\circ b 
=a^{\vee}\circ q.$
\end{Cor}

Recall that every locally free sheaf $E$ on $\p3$ is the 
cohomology of a monad of the form \cite{Ho}:

\begin{equation}\label{monad-horrocks}
0 \to \oplus_{i = 1}^{r} \mathcal{O}_{\mathbb{P}^3}(a_i) \to 
\oplus_{j = 1}^{s} \mathcal{O}_{\mathbb{P}^3}(b_j) \to 
\oplus_{k = 1}^{t} \mathcal{O}_{\mathbb{P}^3}(c_k) \to 0.
\end{equation}

In this work we will be interested in rank 2 locally free 
sheaves with vanishing first Chern class. Under these 
conditions, we have $E^{\vee} \simeq E$, and this implies that $t = 
r$, $s=2r+2$, and $\{a_i\} = \{-c_k\}$. In addition, the 
middle entry of the monad is also self dual, so that 
(\ref{monad-horrocks}) reduces to
$$
0 \to \oplus_{i = 1}^{r} \mathcal{O}_{\mathbb{P}^3}(a_i) \to 
\oplus_{j = 1}^{r+1} \left( 
\mathcal{O}_{\mathbb{P}^3}(b_j)\oplus 
\mathcal{O}_{\mathbb{P}^3}(-b_j) \right) \to 
\oplus_{i = 1}^{r} \mathcal{O}_{\mathbb{P}^3}(-a_i) \to 0.
$$
Finally, recall also that $r$ coincides with the number of 
generators of 
$\h^{1}_{*}(E) = \bigoplus_{p \in \mathbb{Z}} \h^{1}(E(p))$ 
as a graded module over the ring of homogeneous polynomials 
in four variables, while $a_i$ are the degrees of these 
generators, cf. \cite[Thm. 2.3]{JardimSteiner}.


Instanton bundles are a particularly important class of stable 
rank 2 vector bundles due to their many remarkable properties and 
applications in mathematical physics. Besides this, instanton 
bundles form the only known irreducible component of the moduli 
space $\mathcal{B}(c)$ for every $c\in\mathbb{N}$.

In the remaining part of this section we will present the main 
results concerning instanton sheaves that will be used below. We 
start by recalling the definition of instanton sheaves on $\p3$; 
see \cite[Introduction]{Jardim2006} for further information on 
these objects.

\begin{Def}\label{Def instanton}
An \emph{instanton sheaf} on $\mathbb{P}^3$ is a torsion 
free coherent sheaf $E$ with $c_1(E)=0$ satisfying the 
following cohomological conditions:
\begin{equation}\label{inst vanish}
h^0(E(-1))=h^1(E(-2))=h^{2}(E(-2))=h^3(E(-3))=0. 
\end{equation}
The integer $n:=c_2(E)$ is called the \emph{charge} of $E$. 
When $E$ is locally free, we say that $E$ is an 
\emph{instanton bundle}.
\end{Def}

We remark that instanton bundles of rank $r>2$ and non 
locally free instanton sheaves of rank $r\ge2$ on $\p3$ are 
not $\mu$-semistable in general, and also the vanishing of 
$h^1(E(-2))$ does not imply the vanishing of $h^2(E(-2))$. 
The definition above is the right generalization of the 
usual definition of an instanton vector bundle in the sense 
that, applying the Beilinson spectral sequence \cite[Ch. II, 
Thm. 3.1.4]{okonek}
\begin{equation}\label{Beilinson}
\mathrm{E}_1^{pq}=H^q(E(-p-1)\otimes\Omega_{\p3}^{-p})\otimes
\op3(p+1)\Rightarrow\mathrm{E}^{p+q}_{\infty}=
\left\{
\begin{array}{cc}
E, & \ p+q=0, \\
0, & \ p+q\ne0,
\end{array}\right.
\end{equation}
to an arbitrary rank $r$ instanton sheaf $E$ of charge $k$, 
the vanishing \eqref{inst vanish} yields that $E$ is the 
cohomology of a monad of the form
\begin{equation}\label{general monad}
0 \to V_k\otimes\op3(-1) \to V_{r+2k}\otimes\op3 \to 
V'_k\otimes\op3(1) \to 0. 
\end{equation}
Note that, conversely, the cohomology of a monad as above is an 
instanton sheaf as defined in Definition \ref{Def 
instanton}, see \cite[Thm. 3]{Jardim2006}.

The cokernel $N$ of any monomorphism of sheaves 
$\mathcal{O}_{\PP}(-1) \to \Omega^1_{\PP}(1)$  is called a 
\emph{null correlation sheaf}:
\begin{equation}\label{nul corr}
0\to\mathcal{O}_{\p3}(-1)\xrightarrow{s}\Omega_{\p3}^1(1)\to N
\to0.
\end{equation}
Such sheaves are precisely the rank 2 instanton sheaves of 
charge $1$, and are parametrized by the projective space 
$\mathbb{P}H^0(\Omega^1_{\p3}(2))\simeq \p5$. If $N$ is 
locally free, we say that $N$ is a \emph{null correlation 
bundle}. The set of non locally free null correlation 
sheaves are parametrized by the Grassmanian of lines in 
$\p3$: given a line $l\subset\p3$ the corresponding null 
correlation sheaf $N_l$ is defined up to an 
isomorphism by the exact sequence 
\begin{equation}\label{triple for N}
0 \to N_l\to V_{2}\otimes\op3\xrightarrow{\varepsilon}
\mathcal{O}_l(1)\to 0.
\end{equation}

For the purposes of this paper, it is important to study 
rank 4 instanton bundles of charge 1. Some of the following 
facts might be well known, but for lack of a reference we 
include proofs here.

\begin{Lema}\label{directsum}
Every rank 4 instanton bundle $E$ of charge 1 over $\PP$ 
fits into an exact sequence:
\begin{equation}\label{sum}
0\to V_2\otimes\op3\to E\to N\to0,
\end{equation}
where $N$ is a null correlation sheaf. If $N$ is a null 
correlation bundle, then sequence \eqref{sum} splits. In 
addition,
\begin{equation}\label{h0E=2}
h^0(E)=2,\ \ \ \ \ h^i(E)=0,\ \ \ i\ge1.
\end{equation}
\end{Lema}

\begin{proof}
As observed in the paragraph right below Definition \ref{Def 
instanton}, $E$ can be obtained as cohomology of a monad 
\eqref{general monad} for $r=4$ and $k=1$:
\begin{equation}\label{ME}
M_E:\ \ \ \xymatrix{
0 \ar[r] & \mathcal{O}_{\mathbb{P}^3}(-1) \ar[r]^{\alpha} & 
V_6 \otimes \mathcal{O}_{\mathbb{P}^3} \ar[r]^{\beta} & 
\mathcal{O}_{\mathbb{P}^3}(1)  \ar[r]& 0.}
\end{equation}
\noindent Without loss of generality, we can choose 
homogeneous coordinates $[x:y:z:w]$ in $\p3$ and a basis in 
$V_6$, such that the map $\beta$ can be written as 
\begin{equation}\label{beta}
\beta := \left( 
\begin{array}{cccccc} x ~ & ~y~ & ~z~ & ~w~ & ~0~ & ~0 
\end{array} \right).
\end{equation}
Hence using the display of the above 
monad, we have that $E$ fits into the following short exact 
sequence
\begin{equation}\label{quotE}
0 \to  \mathcal{O}_{\mathbb{P}^3}(-1) \to  V_2 \otimes 
\mathcal{O}_{\mathbb{P}^3} \oplus \Omega(1)\to  E  \to  0.
\end{equation}
From the above short exact sequence we can build up the 
following commutative diagramm
$$ 
\xymatrix{
&V_2\otimes\op3\ar@{>->}[d]\ar@{=}[r]&V_2\otimes\op3\ar@{>>}[d]\\
\op3(-1)\ar@{>->}[r] \ar@{=}[d]& V_2 \otimes 
\op3\oplus\Omega^1(1)\ar@{>>}[r]\ar@{>>}[d] & E\ar@{>>}[d] \\
\op3(-1) \ar@{>->}[r] & \Omega^1(1)\ar@{>>}[r] & \ N.} 
$$
The rightmost column is the desired sequence.

If $N$ is locally free, then $ \mathrm{Ext}^1(N,\op3) \simeq 
H^1(N)=0$, so the sequence in display \eqref{sum} splits.
The equality \eqref{h0E=2} follows from \eqref{sum}.
\end{proof}

\textbf{Remark.} Assume that a bundle $E$ is the cohomology 
bundle of the monad \eqref{ME}. Then an easy cohomological 
computation shows that $E$ is a rank 4 instanton bundle of 
charge 1.

\vspace{5mm}

Note that, substituting $N$ instead of $E$ into the Beilinson
spectral sequence \eqref{Beilinson} yields the monad for $N$:
\begin{equation}\label{MN}
M_N:\ \ \ \ 0\to \op3(-1)\xrightarrow{\overline{\alpha}}
V_4\otimes\op3\xrightarrow{\overline{\beta}}\op3(1)\to0,\ \ 
\ \ \ \ N=\ker\overline{\beta}/\mathrm{im}~\overline{\alpha},
\end{equation}
fitting together with the monad \eqref{ME} in the 
commutative diagram 
\begin{equation}\label{two monads}
\xymatrix{
 & V_2 \otimes \op3 \ar@{>->}[d] & \\
\op3(-1)\ar@{>->}[r]^-{\alpha}\ar@{=}[d]& V_6\otimes 
\op3\ar@{>>}[r]^-{\beta} \ar@{>>}[d] & \op3(1)\ar@{=}[d]\\
\op3(-1)\ar@{>->}[r]^-{\overline{\alpha}} & 
V_4\otimes\op3\ar@{>>}[r]^-{\overline{\beta}} & \op3(1).} 
\end{equation} 
In this diagram the exact middle column is obtained from the exact 
triple $0\to V_2\to V_6\to V_4\to0$ arising as the cohomology 
sequence of the exact triple $0\to V_2\otimes\Omega_{\p3}\to
E\otimes\Omega_{\p3}\to N\otimes\Omega_{\p3}\to0$ induced by
the triple \eqref{sum}. In addition, from \eqref{two monads}
and \eqref{beta} we obtain
\begin{equation}\label{bar beta}
\overline{\beta}= \left( 
\begin{array}{cccc} x ~ & ~y~ & ~z~ & ~w
\end{array} \right).
\end{equation}

\begin{Prop}\label{hi S2E}
Let $E$ be a rank 4 instanton bundle $E$ of charge 1 over 
$\PP$, then $h^0(S^2E)=3,\ h^1(S^2E)=5,\ h^2(S^2E)=0.$ 
\end{Prop}

\begin{proof}
Taking the symmetric power of the sequence in display 
\eqref{quotE}, we obtain that $S^2 E$ fits into the 
following short exact sequence:
\begin{equation*}
    \xymatrix{
0 \ar[r] & V_2 \otimes \mathcal{O}_{\mathbb{P}^3}(-1) \oplus 
\Omega \ar[r] & (S^2 V_2 \otimes \mathcal{O}_{\mathbb{P}^3}) 
\oplus (V_2 \otimes \Omega(1)) \oplus S^2 \Omega (2)  \ar[r] 
& S^2 E  \ar[r]& 0.}
\end{equation*}

\noindent From the long exact sequence of cohomology we have

$$0 \to S^2 V_2 \to H^0(S^2 E) \to \mathbf{k} \to \Lambda^2 
W^{\vee} \to H^1(S^2 E) \to 0,$$
\noindent where $W$ is the 4-dimensional $\mathbf{k}-$vector 
space such that $\mathbb{P}^3 = \mathbb{P}(W)$, and 
$$0 \to H^2(S^2 E) \to 0.$$
\noindent  From which we conclude that $H^2(S^2 E) = 0$. The 
map $\mathbf{k} \to \Lambda^2 W^{\vee}$ is given by the 
skew-form corresponding to the morphism 
$\mathcal{O}_{\PP}(-1) \to \Omega(1)$, in the definition of 
$E$, and in particular is non-zero, which implies that  
$\mathbf{k} \to \Lambda^2 W^{\vee}$ is injective, and 
therefore 
$$H^0(S^2 E) \simeq S^2 V_2 ~{\textrm{and}~ H^1(S^2 E) 
\simeq \Lambda^2 W^{\vee}/\mathbf{k} }$$
\noindent from which our result follows. 
\end{proof}

In the remaining part of this section we will discuss the 
existence of a symplectic structure on an arbitrary rank 4 
instanton bundle of charge 1. Recall that a locally free 
sheaf $E$ is said to be \emph{symplectic} if it admits a 
symplectic structure, that is, there exists an isomorphism 
$\varphi:E\rightarrow E^{\vee}$, such that $\varphi^{\vee} = 
-\varphi$. A \emph{symplectic instanton bundle} is a pair 
$(E,\varphi)$ consisting of an instanton bundle $E$ together 
with a symplectic structure $\varphi$ on it; two symplectic 
instanton bundles $(E,\varphi)$ and $(E',\varphi')$ are 
isomorphic if there exists a bundle isomorphism $g:E\simto 
E'$ such that $\varphi=g^\vee\circ\varphi'\circ g$.

\begin{Prop}\label{sympl str charge 1}
Any rank 4 instanton bundle $E$ of charge 1 admits a 
symplectic structure. In particular, if $E$ splits as
$E=V_2\otimes\op3\oplus N$ where $N$ is a null 
correlation bundle, then any symplectic structure $\varphi$
on $E$ splits as $\varphi=\varphi_1\oplus\varphi_2$ where
$\varphi_1$ and $\varphi_2$ are symplectic structures on
$V_2\otimes\op3$ and $N$, respectively.
\end{Prop}
\begin{proof}
Let $E$ be an instanton rank 4 bundle. If $E$ splits as
$E=V_2\otimes\op3\oplus N$, where $N$ is a null 
correlation bundle, then $\det(V_2\otimes\op3)=\det N=\op3$, 
hence both rank 2 bundles $V_2\otimes\op3$ and $N$ admit 
symplectic structures, say, 
\begin{equation}\label{two sympl str}
\varphi_1:V_2\otimes\op3\xrightarrow{\simeq}V_2^{\vee}\otimes
\op3,\ \ \ \ \ \ \varphi_2:N\xrightarrow{\simeq}N^{\vee}. 
\end{equation}
Then
\begin{equation}\label{dir sum}
\varphi=\varphi_1\oplus\varphi_2:\ 
E\xrightarrow{\simeq}E^{\vee}
\end{equation} 
is a symplectic structure on $E$.
Since
\begin{equation}\label{Homs=0}
\Hom(V_2\otimes\op3,N)=\Hom(N,V_2\otimes\op3)=0,
\end{equation} 
it follows immediately that any symplectic stucture on $E$
splits as in \eqref{dir sum}. 

Note also that, in view of \eqref{nul corr} 
\begin{equation}\label{Exts=0}
\Ext^i(V_2\otimes\op3,N)=\Ext^i(N,V_2\otimes\op3)=0,\ \ \ i\ge1.
\end{equation}

Now let $E$ be a non-splitting instanton, i. e. 
$E/V_2\otimes\op3$ is a null correlation sheaf $N_l$ which 
is not locally free at the points of the line $l$ given by 
the equations, say, $\{x=y=0\}$. This means that the 
morphism $\overline{\alpha}$ in the monad \eqref{MN} for 
$N=N_l$ is vanishes at $l$, so that
\begin{equation}\label{bar alpha}
\overline{\alpha}=A\binom{x}{y},\ \ \ \ \ A=(\alpha_{ij}),
\ \ \ 1\le i\le4,\ 1\le j\le2,
\end{equation}
where $A$ is a $(4\times2)$-matrix of rank 2. The condition
that $\overline{\beta}\circ\overline{\alpha}$ in \eqref{MN} 
is the zero morphism together with \eqref{bar alpha} and
\eqref{bar beta} implies that all the coefficients 
$\alpha_{ij}$ of the matrix $A$, except $\alpha_{12}$ and 
$\alpha_{21}$, vanish and $\alpha_{12}+\alpha_{21}=0$. Thus, 
taking without loss of generality $\alpha_{12}=1$, we obtain
\begin{equation}\label{bar alpha2}
\overline{\alpha}=\left( 
\begin{array}{c} y \\-x\\ 0\\0
\end{array} \right).
\end{equation}
Since the cohomology sheaf of the middle monad in 
\eqref{two monads} is locally free, the morphism $\alpha$ in 
that diagram is a subbundle morphism. This together with
\eqref{bar alpha2} implies, again without loss of generality,
that there exists a $(2\times2)$-matrix $C=(c_{ij})$ such 
that  
\begin{equation}\label{alpha}
\alpha=\left( 
\begin{array}{c} y \\-x\\ 0\\ 0\\ c_{11}x+c_{12}y+z\\
c_{21}x+c_{22}y+w
\end{array} \right)_.
\end{equation}
It now follows from \eqref{alpha} and \eqref{beta} that the 
skew-symmetric $(6\times6)$-matrix $J$ of the following 
$(2\times2)$-block form 
$$
J=\left( 
\begin{array}{ccc} Q & \mathbb{O} & -C^t \\
\mathbb{O} & \mathbb{O} & -\mathbbm{1}\\ C & \mathbbm{1} & 
\mathbb{O} \end{array} \right), \ \ \ \ \ \ \ \textrm{where} 
\ \ \ \ Q=\left(\begin{array}{cc} 0 & -1 \\
1 & 0 \end{array} \right)
$$
satisfies the condition $\alpha=J\beta^t$. This means that,
taking $-J$ for the matrix of the symplectic form $q:V_6\to
V_6^{\vee}$ with respect to the above choice of the basis in
$V_6$, we obtain that $\alpha$ and $\beta$ as morphisms 
satisfy the condition $\beta=\alpha^{\vee}\circ q$.
In other words, the monad \eqref{ME} is symplectic. Then by
Corollary \ref{Okoneksimpletico} its cohomology bundle $E$
also admits a symplectic structure.        
\end{proof}

\section{Modified instanton monads}\label{modif}

We will now study monads of the form \eqref{Monad Gak}, with $a\ge2$ and $k\ge1$:
\begin{equation}\label{Monad2}
0\to\op3(-a)\oplus V_{k}\otimes\op3(-1)\xrightarrow{\alpha}V_{2k+4}
\otimes\op3\xrightarrow{\beta}\op3(a)\oplus V'_k\otimes\op3(1)\to0,
\end{equation}
which we call \emph{modified instanton monads}. The set of 
isomorphism classes of bundles arising as cohomology of such monads
will be denoted by $\mathcal{G}(a,k)$. Note that, so far, $\mathcal{G}(a,k)$ could possibly be empty.
\begin{Prop}\label{Prop 7}
For each $a\ge2$ and $k\ge1$, the family $\mathcal{G}(a,k)$ 
is non-empty and contains stable bundles, while every 
$[\mathcal{E}]\in \mathcal{G}(a,k)$ is $\mu$-semistable. In 
addition, every $[\mathcal{E}]\in\mathcal{G}(a,1)$ is stable.
\end{Prop}
\begin{proof}
Let $F$ be an rank $2$ instanton bundle of charge $k$. Let 
$a \geq 2$ and take $\sigma \in \h^0(F(2a))$ such that its 
zero locus $X:=(\sigma)_0$ is a curve; such 
$\sigma$ always exists if $F$ is a 't Hooft instanton 
bundle, for instance. Let $Y$ be a complete intersection 
curve given by the intersection of two surfaces of degree 
$a$ such that $X \cap Y = \emptyset$. According to 
\cite[Lemma 4.8]{hartshorne1991}, there exists a bundle $E$ 
and a section $\tau \in \h^0(E(a))$ such that $(\tau)_{0} = 
Y \cup X$ which is given as cohomology of a monad of the 
form (\ref{Monad2}). In addition, since $F$ is stable, $X$ 
is not contained in any surface of degree $a$, hence neither 
is $Y \cup X$, and $\mathcal{E}$ is also stable.

It is straightforward to check that every $[\mathcal{E}]\in 
\mathcal{G}(a,k)$ satisfies $h^0(\mathcal{E}(-1))=0$, thus 
$\mathcal{E}$ is $\mu$-semistable.

Now fix $k=1$, and assume that there is $[\mathcal{E}] \in 
\mathcal{G}(a,1)$ satisfying $h^0(\mathcal{E})\ne0$. Setting 
$K:=\ker\beta$, it follows that $h^0(K)\ne0$, hence the 
quotient $K':=K/\op3$ fits into the following exact sequence
$$ 
0 \to K' \to V_{5}\otimes\op3 \xrightarrow{\beta'} \op3(1)
\oplus\op3(a) \to 0. 
$$
By \cite[Thm. 2.7]{bohnhorst1992stability} $K'$ is 
$\mu$-stable. However, the monomorphism $\alpha:\op3(-a)
\oplus\op3(-1)\to K$ induces a monomorphism $\op3(-1)\to 
K'$; by the $\mu$-stability of $K'$, we should have
$$ 
-1 < \mu(K') = -\frac{a+1}{3} ~~ \Longrightarrow ~~ a<2, 
$$
providing the desired contradiction.
\end{proof}

\noindent
\textbf{Remark.} Note that the 
space $X$ of monads \eqref{Monad2} is a locally closed subscheme of 
the affine space $A=\Hom(\op3(-a)\oplus V_{k}\otimes\op3(-1),V_{2k+4}
\otimes\op3)\times\Hom(V_{2k+4}\otimes\op3,\op3(a)\oplus V'_k\otimes
\op3(1))$ defined as $X=\{(\alpha,\beta)\in A\ |\ \alpha\ \text{is a 
subbundle morphism},\ \beta\ \text{is an epimorphism and}\ 
\beta\circ\alpha=0\}$, and there is the 
universal cohomology bundle $\boldsymbol{\cale}$ on $\mathbf{X}$. 
In case $k=1$, it follows from Proposition \ref{Prop 7} that
$\mathcal{G}(a,1)$ is the image of $X$ under the modular morphism 
$\Phi_X:\ X\to\calb(a^2+1),\ x\mapsto[\mathbf{E}|_{\p3\times
\{x\}}]$. Thus, $\mathcal{G}(a,1)$ is a constructible set, i. e., a 
disjoint union of locally closed subsets of $\calb(a^2+1)$. 

\vspace{2mm}
Next, we provide a cohomological characterization for 
modified instanton bundles.

\begin{Prop} A vector bundle $\mathcal{E}$ on $\p3$ is the 
cohomology of a monad of the form (\ref{Monad2}) if and only 
if $\h^1_{*}(\mathcal{E})$ has one generator in degree $-a$ 
and $k$ generators in degree $-1$, and its Chern classes are 
$c_1(\mathcal{E}) = 0$, and $c_2(\mathcal{E}) = a^2 + k$.
\end{Prop}

\begin{proof}
The ``only if'' part is straightforward. If $\mathcal{E}$ is 
a self dual vector bundle on $\PP$ with one generator in 
degree $-a$ and $k$ generators in degree $-1$, then by 
\cite[Thm. 2.3]{JardimSteiner}, $\mathcal{E}$ is 
cohomology of a monad of the type: 
$$
0 \to \mathcal{O}_{\mathbb{P}^3}(-a) \oplus V_{k}\otimes 
\mathcal{O}_{\mathbb{P}^3}(-1) \xrightarrow{\alpha} 
\oplus_{i=1}^{2k+4} \mathcal{O}_{\mathbb{P}^3}(k_i) 
\xrightarrow{\beta} \mathcal{O}_{\mathbb{P}^3}(a)\oplus 
V_{k}\otimes \mathcal{O}_{\mathbb{P}^3}(1) \to 0.
$$

Computing the Chern class give us $c_2(\mathcal{E}) = a^2 + 
k -\sum_{i=1}^{6}k_{i}^{2}$, since $c_2(\mathcal{E}) = a^2 + 
k$, we have $k_i = 0 $ for all $i$.
\end{proof}

The modified instanton bundles are also related to usual 
instanton bundles of higher rank in a very important way. 
The precise relationship is outlined in the next couple of 
lemmas, and then summarized in Proposition \ref{bijection} 
below.

\begin{Lema}\label{equivalence1}
(i) Given a vector bundle 
$[\mathcal{E}]\in\mathcal{G}(a,k)$, there exists a rank 4 
instanton bundle $E$ of charge $k$, and 
sections $\sigma\in H^0(E(a)),\ \tau\in 
H^0(E^{\vee}(a))$ such that the complex:
\begin{equation}\label{Monad3}
0 \to \mathcal{O}_{\mathbb{P}^3}(-a) \xrightarrow{\sigma} 
E \xrightarrow{\tau} \mathcal{O}_{\mathbb{P}^3}(a) \to0
\end{equation}
is a monad whose cohomology coincides with $\mathcal{E}$.\\
(ii) The construction of the monad \eqref{Monad3} is 
functorial in the sense that, if $\cale\xrightarrow{\sim}\cale'
$, then the induced isomorphism $E\xrightarrow{\sim}E'$ extends 
to an isomorphism of monads
\begin{equation}\label{fgh}
\xymatrix{
\op3(-a) \ar@{>->}[r]^-{\sigma} \ar[d]^{f}_{\simeq} & E 
\ar@{>>}[r]^-{\tau} \ar[d]^{g}_{\simeq} & \op3(a) 
\ar[d]^{h}_{\simeq}\\
\op3(-a) \ar@{>->}[r]^-{\sigma'} & E'\ar@{>>}[r]^-{\tau'} & \op3(a).}
\end{equation}
\end{Lema}
\begin{proof}
(i) Since $a\ge2$, there is the canonical subbundle morphism 
$i:V_{k}\otimes \mathcal{O}_{\p3}(-1)\to\mathcal{O}_{\p3}(-a) 
\oplus V_{k}\otimes \mathcal{O}_{\p3}(-1)$ which, together with 
the morphisms $\alpha$ and $\beta$ from the monad 
\eqref{Monad2}, yields a subbundle morphism $\alpha_1:=\alpha
\circ i:V_{k}\otimes \mathcal{O}_{\p3}(-1)\to V_{2k+4}\otimes 
\mathcal{O}_{\p3}$ and an epimorphism $\beta_1:=i^{\vee}\circ
\beta:V_{2k+4}\otimes\mathcal{O}_{\p3}\to V'_{k}\otimes\mathcal{
O}_{\p3}(1)$. We thus obtain a new monad of type 
\eqref{general monad} with $r=4$:
\begin{equation}\label{Monad 2A}
0 \to V_{k}\otimes\op3(-1) \xrightarrow{\alpha_1}V_{2k+4}\otimes
\op3\xrightarrow{\beta_1}V'_k\otimes\op3(1)\to 0
\end{equation}
the cohomology bundle 
\begin{equation}\label{E is 4-instanton}
E=\frac{\ker(\beta_1)}{\im(\alpha_1)}
\end{equation}
of which is a rank-4 instanton, according to a remark after
\eqref{general monad}. The monads \eqref{Monad2} and 
\eqref{Monad 2A} fit in a commutative diagram with exact columns
\begin{equation}\label{i}
\xymatrix{
V_k\otimes\op3(-1)\ar@{>->}[d]^-i\ar@{>->}[dr]^-{\alpha_1} & & \op3(a)
\ar@{>->}[d] \\ 
V_k\otimes\op3(-1)\oplus\op3(-a)\ar@{>>}[d]\ar@{>->}[r]^-{\alpha} & 
V_{2k+4}\otimes\op3\ar@{>>}[r]^-{\beta}\ar@{>>}[dr]^-{\beta_1} & 
V'_k\otimes\op3(1)\oplus\op3(a)\ar@{>>}[d]^-{i^\vee} \\
\op3(-a) & & V'_k\otimes\op3(1).} 
\end{equation}
Now a standard diagram chasing with diagram \eqref{i} using 
\eqref{E is 4-instanton} and the relation $\cale=\frac{\ker(\beta
)}{\im(\alpha)}$ yields a subbundle morphism $\op3(-a)
\xrightarrow{\sigma} E$ and an epimorphism $E\xrightarrow{\tau} 
\op3(a)$ fitting in the monad \eqref{Monad3} with the cohomology 
bundle $\cale$.

(ii) Again, since $a\ge2$, it follows immediately from 
\eqref{Monad 2A} and \eqref{E is 4-instanton} that $\Hom(\op3(a),E')=
\Hom(E,\op3(-a))=\Ext^1(\op3(a),\op3(-a))=\Ext^{1}(E,\op3(-a))=\Ext^{1}
(\op3(a),E')=\Ext^{2}(\op3(a),\op3(-a))=0$ for the rank-4 instanton 
bundles $E$ and $E'$ of charge $k$. The statement (ii) now follows 
from \cite[Lemma 4.1.3]{okonek}.
\end{proof}

\begin{Lema}\label{equivalence2}
Given a monad \eqref{Monad3} with $E$ being a rank 4 instanton 
bundle of charge $k$, there is a monad of the form 
\eqref{Monad2} whose cohomology coincides with the cohomology 
of the above monad.
\end{Lema}
\begin{proof}
This is a diagram chasing. Namely, by \eqref{general monad}, $E$ 
is the cohomology of a monad of the form
\begin{equation} \label{kinstanton}
0\to V_k\otimes\op3(-1)\xrightarrow{\alpha_1}V_{2k+4}\otimes\op3
\xrightarrow{\beta_1}V'_k\otimes\op3(1)\to 0.
\end{equation} 
This monad can be splitted to the exact triples of bundles
\begin{equation}\label{tri1}
0\to E\to\mathrm{coker}(\alpha_1)\to V'_k\otimes\op3(1)\to 0,
\end{equation}
\begin{equation}\label{tri2}
0\to V_k\otimes\op3(-1)\xrightarrow{\alpha_1}V_{2k+4}\otimes 
\op3\xrightarrow{\varepsilon}\mathrm{coker}(\alpha_1)\to0.
\end{equation}
Respectively, the monad \eqref{Monad3} splits into the exact
triples
\begin{equation}\label{split1}
0\to\ker(\tau)\to E\xrightarrow{\tau}\op3(a)\to0,\ \ \ \ \ \ 
0\to\op3(-a)\to\ker(\tau)\xrightarrow{\delta}\cale\to0,
\end{equation}
where $\cale$ is the cohomology bundle of the monad 
\eqref{Monad3}. The triple \eqref{tri1} and the first triple 
\eqref{split1}, together with the vanishing of 
$\Ext^1(V'_k\otimes\op3(1),\op3(a))$, yields by push-out the 
exact triple $0\to\ker(\tau)\to\mathrm{coker}(\alpha_1)
\xrightarrow{\gamma}V'_k\otimes\op3(1)\oplus\op3(a)\to0$ 
which, together with \eqref{tri2}, yields a commutative 
diagram in which we set $K:=\ker(\gamma\circ\varepsilon)$: 
$$ 
\xymatrix{
V_k\otimes\op3(-1)\ar@{>->}[r]\ar@{=}[d]& K\ar@{>>}[r]\ar@{>->}[d] & \ker(\tau)\ar@{>->}[d] \\
V_k\otimes\op3(-1) \ar@{>->}[r] & V_{2k+4}\otimes\op3\ar@{>>}[r]
^-{\varepsilon}\ar@{>>}[d]^-{\gamma\circ\varepsilon} & \mathrm{coker}(\alpha_1)\ar@{>>}[d]^-{\gamma}\\
& V'_k\otimes\op3(1)\oplus\op3(a)\ar@{=}[r] & 
V'_k\otimes\op3(1)\oplus\op3(a).} 
$$
Similarly, the upper horizontal triple of this diagram, together
with the second triple \eqref{split1}, yield the exact triple
$0\to V_k\otimes\op3(-1)\oplus\op3(-a)\to K\to\cale\to0$ which,
being combined with the middle vertical triple in this diagram, 
yields the monad \eqref{Monad2} with the cohomology bundle $\cale$.
\end{proof}

Next, we argue that the instanton bundle $E$ obtained 
in Lemma \ref{equivalence1} comes with a natural symplectic 
structure.

\begin{Lema}\label{sympl}
If $E$ is a rank 4 instanton bundle of charge $k$ that 
fits in a monad of the form (\ref{Monad3}), such that its 
cohomology sheaf $\cale$ is a vector bundle, then $E$ admits 
a symplectic structure, and $\tau$ is determined by $\sigma$.
\end{Lema}
\begin{proof}
Since $\cale$ is a rank $2$ vector bundle with $c_1(\cale)=0
$, there is a (unique up to scaling) symplectic isomorphism 
$\varphi:\cale\xrightarrow{\simeq}\cale^{\vee}$. Now, 
repeating the proof of Lemma \ref{equivalence1}(ii) for 
$\cale'=\cale^{\vee}$, we obtain an isomorphism of monads:
$$
\xymatrix{
\op3(-a) \ar@{>->}[r]^-{\sigma}\ar[d]^{g}_-{\simeq}&  E \ar@{>>}[r]
^-{\tau}\ar[d]^{\varphi}_-{\simeq}& \mathcal{O}_{\mathbb{P}^3}(a) 
\ar[d]^-{h}_-{\simeq} \\
\op3(-a) \ar@{>->}[r]^-{\tau^{\vee}}&  E^{\vee} 
\ar@{>>}[r]^-{\sigma^{\vee}} & \mathcal{O}_{\mathbb{P}^3}(a)}
$$
such that $\varphi^{\vee} = - \varphi$, so 
$(E,\varphi)$ is a symplectic instanton bundle, and 
$\tau=\sigma^\vee\circ\varphi$. 
\end{proof}

Putting Lemmas \ref{equivalence1}, \ref{equivalence2} and 
\ref{sympl} together, we obtain the following statement.

\begin{Prop}\label{bijection} 
A rank 2 bundle $\mathcal{E}$ belongs to $\mathcal{G}(a,k)$, 
i. e., $\mathcal{E}$ is the cohomology of a monad of the form
\eqref{Monad2} if and only if it is also the cohomology 
$\mathcal{E}=\mathcal{H}^0(A_{E,\varphi,\sigma})$ of a monad 
of the form:
\begin{equation}\label{mon M}
A_{E,\varphi,\sigma}:\ \ \ \ \ \ 
0 \to \mathcal{O}_{\mathbb{P}^3}(-a) \xrightarrow{\sigma} 
E \xrightarrow{\sigma^\vee\circ\varphi} 
\mathcal{O}_{\mathbb{P}^3}(a) \to 0,
\end{equation}
where $(E,\varphi)$ is a rank 4 symplectic instanton bundle 
of charge $k$.
\end{Prop}

\section{Set $\mathcal{G}(a,1)$ and related families of 
sheaves}\label{G(a,1)}

We introduce a piece of notation which we will use below.
Denote by $\cali(k)$ the set of isomorphism classes of 
symplectic rank 4 instanton bundles with $c_2=k$. As before, 
let $V_k$ and $V_{2k+4}$ be the fixed vector spaces of 
dimensions $k$ and $2k+4$, respectively, and let $(\wedge^2V
_{2k+4}^{\vee})^0$ be an open subset of the vector space 
$\wedge^2V_{2k+4}^{\vee}$ consisting of nondegenerate 
symplectic forms on $V_{2k+4}$. Next, for a given morphism
$\tilde{\alpha}: V_k\otimes\op3(-1)\to V_{2k+4}\otimes \op3$
we denote by $a$ the homomorphism $V_k\otimes U_4\to 
V_{2k+4}$ corresponding to the morphism $\tilde{\alpha}$ 
under the isomorphism $\Hom(V_k\otimes\op3(-1),V_{2k+4}
\otimes\op3)\cong W:=\Hom(V_k\otimes U_4,V_{2k+4})$, where 
$U_4:=H^0(\op3(1))^{\vee}$. We will call $\tilde{\alpha}$ 
the \textit{morphism associated to} $a\in W$.

Recall the description of symplectic rank 4 instantons 
$(E, \varphi)$ in terms of symplectic monads 
\eqref{sympl monad} below. Namely, for a given point
$$
m=(a,q)\in W\times(\wedge^2V_{2k+4}^{\vee})^0
$$
consider the monad \eqref{kinstanton} in which $\tilde{\alpha
}$ the morphism associated to the homomorphism $a$, and the
morphism $\tilde{\beta}$ is such that $\tilde{\beta}=\tilde{
\alpha}^t(q)$, where $\tilde{\alpha}^t(q)$ is the 
composition 
$V_{2k+4}\otimes \op3\xrightarrow{q\otimes\mathrm{id}_{\op3}}
V_{2k+4}^{\vee}\otimes\op3\xrightarrow{\tilde{\alpha}^{\vee}}
V_k^{\vee}\otimes\op3(1)$:
\begin{equation}\label{sympl monad}
A_{m}:\ \ \ 0\to V_k\otimes \op3(-1)
\xrightarrow{\tilde{\alpha}} V_{2k+4}\otimes\op3\xrightarrow{
\tilde{\alpha}^t(q)} V_k^{\vee}\otimes\op3(1)\to0.
\end{equation} 
We call $A_m$ a \textit{symplectic monad}. We also will 
denote by $\mathcal{H}^0(A_m)$ the cohomology bundle of the 
monad $A_m$.

Consider the set 
$\mathcal{M}(k)$ of symplectic monads 
\eqref{sympl monad}:
\begin{equation}\label{space M(k)}
\mathcal{M}(k)=\{(a,q)\in W\times(\wedge^2V_{2k+4}^{\vee})^0
\ |\ (a,q)\ \textrm{satisfies the conditions (i)-(ii)} \}
\end{equation}
where:\\
(i) the morphism $\tilde{\alpha}$ associated to $a$ 
is a subbundle morphism,\\
(ii) the composition $\tilde{\alpha}^t(q)\circ\tilde{\alpha}
$ is the zero morphism.\\
Since $W$ is a vector space, and the condition (i), resp., 
(ii) is an open, resp., closed condition on the point $a
\in W$, it follows that $\mathcal{M}(k)$ has a natural 
structure of a locally closed subscheme of the affine space 
$W\times\wedge^2V_{2k+4}^{\vee}$. 

\vspace{5mm}
From now on we will restrict to the case $k=1$. Set
$\widetilde{M}:=\mathcal{M}(1)$. Note that the 
condition (i) of the definition of $\mathcal{M}(k)$ is empty 
in the case $k=1$, since in this case the the vanishig of 
$\wedge^2(V_1^{\vee}\otimes\op3(1))$ clearly implies 
$\alpha^t(q)\circ\alpha=0$. Hence, $\widetilde{M}$ 
is a nonempty open (hence dense) subset of the affine space 
$W\times\wedge^2V_6^{\vee}$, where $W=\Hom(V_1\otimes 
U_4,V_6)\simeq\mathbf{k}^{24}$. In particular, 
$\widetilde{M}$ is irreducible and
\begin{equation}\label{dim M}
\dim\widetilde{M}=\dim W+\dim\wedge^2V_6^{\vee}=45.
\end{equation}

\begin{Prop}\label{boundedness}
Any rank 4 instanton of charge 1 appears as a cohomology 
bundle of a symplectic monad 
\begin{equation}\label{sympl monad k=1}
A_{m}:\ \ \ 0\to \op3(-1)\xrightarrow{\tilde{\alpha}} V_6
\otimes\op3\xrightarrow{\tilde{\alpha}^t(q)}\op3(1)\to0.
\end{equation} 
for some $m\in\widetilde{M}$.
\end{Prop}
\begin{proof}
Let $E$ be a rank 4 instanton of charge 1. According to 
Proposition \eqref{sympl str charge 1}, $E$ admits a 
symplectic structure $\varphi:E\xrightarrow{\sim}E^{\vee}$.
It then known from \cite[Section 3]{Bruzzo2012} that, under 
the condition $h^0(E)=h^1(-2)=0$ on a symplectic bundle $E$, 
this bundle is a cohomology of a symplectic monad from 
$\widetilde{M}$. However, the proof given therein, 
works without changes under the slightly weaker conditions 
\eqref{inst vanish} used in the Definition \ref{Def 
instanton}.
\end{proof}

On $\widetilde{\mathbf{M}}=\p3\times\widetilde{M}$ there is 
the universal symplectic monad
$\boldsymbol{\mathcal{A}}_{\widetilde{\mathbf{M}}}:\ \ 
0\to\calo_{\widetilde{\mathbf{M}}}(-1,0)
\xrightarrow{\boldsymbol{\alpha}}V_6\otimes
\calo_{\widetilde{\mathbf{M}}}\xrightarrow{\boldsymbol{\alpha
}^t}\calo_{\widetilde{\mathbf{M}}}(1,0)\to0
$
with the cohomology sheaf 
$\widetilde{\mathbf{E}}=\ker\boldsymbol{\alpha}^t/
\mathrm{im}\boldsymbol{\alpha}$.
Here $\boldsymbol{\alpha}^t=\boldsymbol{\alpha}^{\vee}\circ
\mathbf{q}_{\widetilde{\mathbf{M}}}$ and
$\mathbf{q}_{\widetilde{\mathbf{M}}}:\ V_6\otimes\calo_{
\widetilde{\mathbf{M}}}\xrightarrow{\sim}V_6^{\vee}\otimes
\calo_{\widetilde{\mathbf{M}}}$
is the tautological symplectic structure on
$V_6\otimes\calo_{\widetilde{\mathbf{M}}}$.
From now on we fix an isomorphism of the monad
$\boldsymbol{\mathcal{A}}_{\widetilde{\mathbf{M}}}$ with its 
dual monad $\boldsymbol{\mathcal{A}}_{\widetilde{
\mathbf{M}}}^{\vee}$ by the following diagram:
$$
\xymatrix{
\boldsymbol{\mathcal{A}}_{\widetilde{\mathbf{M}}}:\ \ 
\calo_{\widetilde{\mathbf{M}}}(-1,0)\ar@{>->}[r]^-{\boldsymbol{
\alpha}}\ar[d]_{\mathrm{-id}}^-{\simeq} &  
V_6\otimes\calo_{\widetilde{\mathbf{M}}}
\ar@{>>}[r]^-{\boldsymbol{\alpha}^t} 
\ar[d]_{\mathbf{q}_{\widetilde{\mathbf{M}}}}^-{\simeq}& 
\calo_{\widetilde{\mathbf{M}}}(1,0)
\ar[d]_-{\mathrm{id}}^-{\simeq} \\
\boldsymbol{\mathcal{A}}_{\widetilde{\mathbf{M}}}^{\vee}:\ \  
\calo_{\widetilde{\mathbf{M}}}(-1,0)
\ar@{>->}[r]^-{(\boldsymbol{\alpha}^t)^{\vee}} & V_6^{\vee}\otimes
\calo_{\widetilde{\mathbf{M}}}\ar@{>>}[r]^-{\boldsymbol{\alpha}
^{\vee}} & \calo_{\widetilde{\mathbf{M}}}(1,0).}
$$
This isomorphism induces the symplectic structure 
\begin{equation}\label{bf E}
\boldsymbol{\varphi}_{\widetilde{\mathbf{M}}}:
\widetilde{\mathbf{E}}\xrightarrow{\simeq}
\widetilde{\mathbf{E}}^{\vee},\ \ \ \ \text{and}\ \ \ \ 
E_{m}=\widetilde{\mathbf{E}}|_{\p3\times\{m\}},\ \ \ \ 
\varphi_{m}=\boldsymbol{\varphi}_{\widetilde{\mathbf{M}}}
|_{\p3\times\{m\}}:E_{m}\xrightarrow{\sim}E_{m}^{\vee},\ \ \ \ 
m\in\widetilde{M},
\end{equation}
i. e. $(E_m,\varphi_m)$ is a symplectic rank 4 instanton of charge 1. 
Note that, by the universality of the space $\widetilde{M}$, for any 
symplectic rank 4 instanton $(E,\varphi)$, there exists a unique 
point $m\in\widetilde{M}$ such that $(E,\varphi)=(E_m,\varphi_m)
$, where $E_m$ and $\varphi_m$ are given by \eqref{bf E}.
It follows from \eqref{h0E=2} and the Base Change that the 
$\calo_{\widetilde{M}}$-sheaf $\widetilde{\mathbf{U}}:=p_{
\widetilde{M}*}\widetilde{\mathbf{E}}$ is a rank 2 locally free 
sheaf and there is an exact triple on $\widetilde{\mathbf{M}}$, 
where $\mathbf{ev}$ is the canonical morphism:
\begin{equation}\label{global triple}
0\to \widetilde{\mathbf{U}}_{\widetilde{\mathbf{M}}}
\xrightarrow{\mathbf{ev}}\widetilde{\mathbf{E}}\to
\widetilde{\mathbf{N}}\to0,\ \ \ \ \ \ 
\widetilde{\mathbf{N}}:=\mathrm{coker}(\mathbf{ev}),
\end{equation} 
and, for any $m\in\widetilde{M}$, the restriction of this triple 
onto $\p3\times\{m\}$ coincides with the triple \eqref{sum} for 
$E=E_{m}$. We thus have a map $\Psi:\widetilde{M}\to\mathbb{P}^5
=P(\wedge^2V_4^{\vee})$,\ $m\mapsto[\widetilde{\mathbf{N
}}|_{\p3\times\{m\}}]$. The map $\Psi$ has the following explicit 
description. Given a point $m=(a,q)\in\widetilde{M}$, consider a
homomorphism $f(a,q):V_4\xrightarrow{a}V_6\xrightarrow{q}V_6
^{\vee}\xrightarrow{a^{\vee}}V_4^{\vee}$. It is clearly skew-
symmetric: $f(a,q)\in\wedge^2V_4^{\vee}$. An easy diagram
chasing with the display of the monad $\boldsymbol{\mathcal{
A}}_{\widetilde{\mathbf{M}}}|_{\p3\times\{m\}}$ (i. e., 
equivalently, of the monad \eqref{sympl monad k=1}) using 
\eqref{global triple} shows that
\begin{equation}\label{Psi=kf}
\Psi(m)=\langle f(a,q)\rangle\in P(\wedge^2V_4
^{\vee}),
\end{equation}
so that $\Psi$ is a well-defined morphism. By the 
universality of the monad $\boldsymbol{\mathcal{A}}
_{\widetilde{\mathbf{M}}}$, $\Psi$ is surjective.

We next consider the set
\begin{equation*}\label{cal M}
M:=\{m\in\widetilde{M}\ |\ \widetilde{\mathbf{N}}|
_{\p3\times\{m\}}\ \textrm{is locally 
free}\}.
\end{equation*}\label{def cal M}
From the definition of $M$ it follows that it is a nonempty open 
subset of $\widetilde{M}$, hence it is irreducible since 
$\widetilde{M}$ is irreducible. Denote
\begin{equation}\label{bf cal E}
\mathbf{E}:=\widetilde{\mathbf{E}}
_{\mathbf{M}},\ \ \ \ \ \ 
\boldsymbol{\varphi}_{\mathbf{M}}:=(\boldsymbol{\varphi}
_{\widetilde{\mathbf{M}}})_{\mathbf{M}}:
\mathbf{E}\xrightarrow{\simeq}\mathbf{E}^{\vee},\ \ \ \ \ 
\mathbf{U}:=\widetilde{\mathbf{U}}_{M},\ \ \ \ \ \ 
\mathbf{N}:=\widetilde{\mathbf{N}}_{\mathbf{M}},
\end{equation}
where $\boldsymbol{\varphi}_{\widetilde{\mathbf{M}}}$ is the
symplectic structure \eqref{bf E}.
Note that, by Lemma \ref{directsum}, for any $m\in M$ the 
triple \eqref{global triple} restricted onto $\p3\times\{m\}$ 
splits:
\begin{equation}\label{loc split}
E_{m}\simeq\op3^{\oplus2}\oplus N_{m},\ \ \ \ \ \  m\in M,
\end{equation} 
where $N_{m}$ is a null correlation bundle. We now show that 
these splittings globalize to the splitting of the triple 
$0\to\mathbf{U}\to\mathbf{E}\to\mathbf{N}\to0$ obtained from  
\eqref{global triple} by restriction onto $\mathbf{M}$:
\begin{equation}\label{glob split}
\mathbf{E}=\mathbf{U}\oplus\mathbf{N}.
\end{equation} 
Indeed, the last triple considered as an extension is given by
the element in $\Ext^1(\mathbf{N},\mathbf{U})$. By
\eqref{Homs=0}, \eqref{Exts=0} and the Base Change \cite[Thm. 
1.4]{L}, the sheaves $\lext^i_{p_M}(\mathbf{N},\mathbf{U}),\ 
i=0,1$, vanish, and the exact sequence relating global and 
relative Ext \cite[(1)]{L} yields $\Ext^1(\mathbf{N},\mathbf{U}
)=0$.

Now, for $a\ge2$ and any $m\in M$, the triple 
\eqref{sum} twisted by $\op3(a)$, in which we set $E=
E_{m}$, yields: 
\begin{equation}\label{h0(E(a))}
h^0(E_{m}(a))=4\binom{a+3}{3}-a-2,\ \ \ \ \ 
\ h^i(E_{m}(a))=0,\ \ i>0.
\end{equation}
Formulas \eqref{bf E}, \eqref{h0(E(a))} and the 
Base Change show  that the sheaf
\begin{equation}\label{F}
F=p_{M*}(\mathbf{E}(a,0))
\end{equation}
is a locally free $\calo_{M}$-sheaf of rank $r=h^0(E_{m}(a))$. 
Consider the scheme $T=\mathbf{P}(F^{\vee})$. By the above, $T$ 
is set-theoretically described as
\begin{equation}\label{M(a,k)}
T=\{(m,\langle\sigma\rangle)\ |\ 
m\in M,\ 0\ne\sigma\in H^0(E_{m}(a))\},
\end{equation}
and the natural projection $\rho:T\to M,\ (m,\langle\sigma
\rangle)\mapsto m$ is a locally trivial $\mathbb{P}^{r-1}
$-bundle. Note that, since $M$ is an open subset of the 
affine space $W$, it follows that $T$ is a variety, and 
from \eqref{dim M} and \eqref{h0(E(a))} we have
\begin{equation}\label{dims}
\dim T=h^0(E_{m}(a))-1+\dim
M=4\binom{a+3}{3}-a+42. 
\end{equation}
On $T$ and $\mathbf{M}$ we have canonical morphisms $F_T^{\vee}
\stackrel{\mathrm{ev}}{\twoheadrightarrow}L$ and $F_{\mathbf{M}}
\xrightarrow{\mathrm{can}}\mathbf{E}(a,0)$, respectively, where 
$L=\calo_{\mathbf{P}(F^{\vee})}(1)$ is the Grothendieck sheaf. 
Consider the composition of morphisms
\begin{equation}\label{bold sigma}
\boldsymbol{\sigma}:\op3\boxtimes 
L^{\vee}\xrightarrow{\mathrm{ev}_{\mathbf{T}}^{\vee}}
F_{\mathbf{T}}\xrightarrow{\mathrm{can}_{\mathbf{T}}}
\mathbf{E}_{\mathbf{T}}(a,0).
\end{equation}
By definition, for any point $(m,\mathbf{k}\sigma)\in T$ the 
restriction $\boldsymbol{\sigma}|_{\p3\times\{(m,\mathbf{k}\sigma)
\}}$ coincides, up to a twist by $\op3(-a)$, with the morphism 
$\sigma:\op3(-a)\to E_{m}$. In view of \eqref{loc split} we may 
represent $\sigma$ as $\sigma=(\sigma_1,\sigma_2),\ \  
\sigma_1\in H^0(\op3^{\oplus2}(a)),\ \ \sigma_2\in H^0(N_{m}(a))$.
For the pair $\sigma=(\sigma_1,\sigma_2)\ne(0,0)$ we will adopt 
in the sequel, together with the notation $\langle\sigma\rangle$, 
the following equivalent notation:
\begin{equation}\label{new notation}
[\sigma_1:\sigma_2]:=\langle\sigma\rangle=\{(\lambda\sigma_1,
\lambda\sigma_2)|\lambda\in\mathbf{k}^{\times}\},
\end{equation}
and also understand $[\sigma_1:\sigma_2]$ as a point of the
projective space $P(H^0(\op3^{\oplus2}(a))\oplus H^0(N_m(a)))$. 
Under this notation, define an open subset $S$ of $T$ as
\begin{equation}\label{def M(a)}
\begin{split}
& S:=\{(m,[\sigma_1:\sigma_2])\in T\ |\ (i)\ 
\sigma=(\sigma_1,\sigma_2):\op3(-a)
\to E_{m}\simeq\op3^{\oplus2}\oplus N_{m}\ \\
& \textrm{\ \ \ \ \ \ \ \ \ \ is a subbundle morphism and}\ 
(ii)\ \sigma_1,\sigma_2\ne0\}.
\end{split}
\end{equation}
The subset $S$ is clearly open in $T$. Moreover, it is 
nonempty. Indeed, for any point $m\in M$, $E_m$ decomposes as 
in \eqref{loc split}. Take any $a\ge2$. Since the direct 
summand $N_{m}$ is a null correlation bundle, it follows 
quickly from the triple \eqref{nul corr} for $N=N_{m}$, 
twisted by $\op3(a)$, that $N_{m}(a)$ is generated by global 
sections. From this it follows easily (cf. 
\cite[Proof of Prop. 1.4] 
{Hart1})  that a general section $\sigma_2\in H^0(N_{m}(a))
$ has 1-dimensional zero-locus $(\sigma_2)_0$. Next, since a 
general section $\sigma_1\in H^0(\op3^{\oplus2}(a))$ has for 
its zero locus  a complete intersection curve 
$(\sigma_1)_0=D_1
\cap D_2$ for two surfaces $D_1$, $D_2$ of degree $a$, it 
follows that for general $D_1$ and $D_2$ we have $(\sigma_1)_0
\cap(\sigma_2)_0=\emptyset$. Hence, the section $\sigma=(
\sigma_1,\sigma_2)\in H^0(E_m(a))$ has no zeroes and 
therefore defines a subbundle morphism $\sigma:\op3(-a)\to
E_m$.

It follows that $S$ is irreducible and dense in $T$ since 
$T$ is irreducible. 
The morphism $\boldsymbol{\sigma}_{\mathbf{S}}$ is included 
in the monad $\boldsymbol{\mathcal{A}}:=(\boldsymbol{
\mathcal{A}}_{\widetilde{\mathbf{M}}})_{\mathbf{S}}$ on 
$\mathbf{S}$:
\begin{equation}\label{bf A}
\boldsymbol{\mathcal{A}}:\ 0\to\op3(-a)\boxtimes 
L^{\vee}\xrightarrow{\boldsymbol{\sigma}_{\mathbf{S}}}
\mathbf{E}_{\mathbf{S}}\xrightarrow{\boldsymbol{
\sigma}_{\mathbf{S}}^t}\op3(a)\boxtimes L\to0
\end{equation}
where $\boldsymbol{\sigma}_{\mathbf{S}}^t$ is the composition
$\mathbf{E}_{\mathbf{S}}\xrightarrow{
\boldsymbol{\varphi}_{\mathbf{S}}}\mathbf{E}
_{\mathbf{S}}^{\vee}\xrightarrow{\boldsymbol{\sigma}_{
\mathbf{S}}^{\vee}}\op3(a)\boxtimes L$.
By construction, for any point $(m,\langle\sigma\rangle)\in
S$, the restriction of the monad $\boldsymbol{\mathcal{A}}$ 
onto $\p3\times\{(m,\langle\sigma\rangle)\}$ is 
isomorphic to the monad $A_{E_m,\varphi_m,\sigma}$ in 
\eqref{mon M}. Hence,
\begin{equation}\label{H0=H0}
\mathcal{H}^0(\boldsymbol{\mathcal{A}})|_{\p3\times\{(m,\langle
\sigma\rangle)\}}=\mathcal{H}^0(A_{E_{m},\varphi_{m},\sigma}),\ \ 
\ \ \ \ \ (m,\langle\sigma\rangle)\in S.
\end{equation}
In \eqref{tilde F,T}-\eqref{H0=H0 new} below we will extend 
the constructions \eqref{F}-\eqref{M(a,k)}, 
\eqref{def M(a)}-\eqref{H0=H0} of the data $F$, $T$, $S$, 
$\boldsymbol{\mathcal{A}}$ and $\mathcal{H}^0
(\boldsymbol{\mathcal{A}})$ over $M$ to the
constructions of the corresponding data $\widetilde{F}$, 
$\widetilde{T}$, $\widetilde{S}$, 
$\widetilde{\boldsymbol{\mathcal{A}}}$,
$\mathcal{H}^0(\widetilde{\boldsymbol{\mathcal{A}}})$ 
over $\widetilde{M}$. As a consequence, it will follow:
\begin{equation}\label{F,bold Ma}
F=\widetilde{F}_{\mathbf{M}},\ \ \ \ \ \ 
T=M\times_{\widetilde{M}}\widetilde{T},\ \ \ \ \ \ 
\xymatrix{S\ \ar@{^{(}->}[rr]^{\textrm{open\ dense}} & &\ 
\widetilde{S},}\ \ \ \ \ \ 
\boldsymbol{\mathcal{A}}=
(\widetilde{\boldsymbol{\mathcal{A}}})_{\mathbf{S}},
\ \ \ \ \ \ \mathcal{H}^0(\boldsymbol{\mathcal{A}})=
(\mathcal{H}^0(\widetilde{\boldsymbol{\mathcal{A}}}))
_{\mathbf{S}}.
\end{equation}

For this, we first set
\begin{equation}\label{tilde F,T}
\widetilde{F}:=p_{\widetilde{M}*}(\widetilde{\mathbf{E}}(a,0)),\ 
\ \ \ \ \ \widetilde{T}:=\mathbf{P}(\widetilde{F}^{\vee}),
\end{equation}
and remark that formulas \eqref{h0(E(a))} are still true for any 
$m\in\widetilde{M}$, so that the sheaf $\widetilde{F}$
is a locally free $\calo_{\widetilde{M}}$-sheaf of rank $r=h^0(
E_m(a))$ given by \eqref{h0(E(a))}, and the scheme $\widetilde{T}:
=\mathbf{P}(\widetilde{F}^{\vee})$ is set-theoretically described 
as $\widetilde{T}=\{(m,\langle\sigma\rangle)\ |\ m\in\widetilde{
\mathcal{M}},\ 0\ne\sigma\in H^0(E_{m}(a))\}$. The natural 
projection $\widetilde{\rho}:\widetilde{T}\to\widetilde{M},\ 
(m,\langle\sigma\rangle)\mapsto m$ is a locally trivial 
$\mathbb{P}^{r-1}$-bundle, so that, since $\widetilde{M}$ is an 
open subset of the affine space $W$, it follows that 
$\widetilde{T}$ is an irreducible variety of dimension
\begin{equation}
\dim\widetilde{T}=h^0(E_{m}(a))-1+\dim\widetilde{M}=4\binom{a
+3}{3}-a+42. 
\end{equation}
Here, in accordance with \eqref{dims}, $\widetilde{T}$ and 
$T$ have the same dimension. Next, we have an open subset 
$\widetilde{S}$ of $\widetilde{T}$ defined as $\widetilde{S}:=
\{(m,\langle\sigma\rangle)\in\widetilde{T}\ |\ \sigma:\op3(-a)\to
E_{m}\ \textrm{is a subbundle morphism.}\}$ Since the condition 
(ii) in \eqref{def M(a)} is open, comparing the definition of 
$\widetilde{S}$ with \eqref{def M(a)} we obtain that $S$ is an 
open subset of $T\cap \widetilde{S}$, where the intersection is 
taken in $\widetilde{T}$. Since $S$ is nonempty and $\widetilde{T}
$ is irreducible, the inclusion $\xymatrix{S\ \ar@{^{(}->}[rr]^{
\textrm{open\ dense}} & &\ \widetilde{S}}$ in \eqref{F,bold Ma}
follows  and, moreover, $\widetilde{\rho}_S:S\to M$ coincides 
with the projection $\rho$.

Next, we have the extension of the universal monad 
\eqref{bf A} from $\mathbf{S}$ to $\widetilde{\mathbf{S}}$:
$\widetilde{\boldsymbol{\mathcal{A}}}:\ 0\to\op3(-a)\boxtimes 
L^{\vee}\xrightarrow{\boldsymbol{\sigma}}\widetilde{\mathbf{E}}
_{\widetilde{\mathbf{S}}}\xrightarrow{\boldsymbol{\sigma}^t}\op3
(a)\boxtimes L\to0$, satisfying the relation similar to \eqref{H0=H0}:\ \ 
\begin{equation}\label{H0=H0 new}
\mathcal{H}^0(\widetilde{\boldsymbol{\mathcal{A}}})
|_{\p3\times\{(m,\langle\sigma\rangle)\}}=\mathcal{H}^0(A_{E_{m},
\varphi_{m},\sigma}),\ \ \ \ \ \ 
(m,\langle\sigma\rangle)\in\widetilde{S}.
\end{equation}
%
%
Whence, the relations \eqref{F,bold Ma} follow from 
\eqref{bf cal E}, \eqref{tilde F,T} and the Base Change.  

Consider the modular morphisms
\begin{equation}\label{Phi S,Phi tilde S}
\Phi_{S}:S\to\mathcal{B}(a^2+1),\ \ \ 
\ \Phi_{\widetilde{S}}:\widetilde{S}
\to\mathcal{B}(a^2+1),  
\end{equation}
defined by the families of sheaves $\mathcal{H}^0(\boldsymbol{
\mathcal{A}})$ and $\mathcal{H}^0(\widetilde{\boldsymbol{
\mathcal{A}}})$, respectively. The relations \eqref{F,bold Ma}, 
\eqref{H0=H0 new}, and Proposition \ref{bijection} together with 
the irreducibility of $\widetilde{S}$ yield
\begin{Prop}\label{descrn of G(a,1)}
(i) For $a\ge2$, the set $\mathcal{G}(a,1)$ of isomorphism 
classes of cohomology sheaves of monads \eqref{Monad2} for 
$k=1$ is the image of the modular morphism 
$$
\Phi_{\widetilde{S}}:\ \widetilde{S}
\to\mathcal{B}(a^2+1),\ \ \ \ \ \ (m,\langle\sigma\rangle)
\mapsto[\mathcal{H}^0(\widetilde{\boldsymbol{\mathcal{A}}})
|_{\p3\times\{(m,\langle\sigma\rangle)\}}],
$$
defined by the family $\mathcal{H}^0(\widetilde{\boldsymbol{
\mathcal{A}}})$ of sheaves over $\widetilde{S}$. Its closure 
$\overline{\mathcal{G}(a,1)}$ in $\mathcal{B}(a^2+1)$ is an 
irreducible scheme. \\
(ii) The set $\mathcal{G}(a,1)_0:=\Phi_{S}(S)$ is dense in 
$\overline{\mathcal{G}(a,1)}$.
\end{Prop}

In the remaining part of this section we will construct a new
family of monads $\mathbf{A}_{\mathbf{Y}}$ on $\p3$,  with 
base $Y$ and cohomology sheaves belonging to 
$\mathcal{G}(a,1)$, for which the related modular morphism
\begin{equation*}
\Phi_Y:\ Y\to\mathcal{B}(a^2+1),\ \ \ \ \ \ 
y\mapsto[\mathcal{H}^0(\mathbf{A}_{\mathbf{Y}})|_{\p3\times\{y\}}]\end{equation*}
has $\mathcal{G}(a,1)_0$ as its image (see Proposition 
\ref{Prop 15} below). This family will be used in the next 
Section to prove one of the main results of the paper - the 
rationality of $\overline{\mathcal{G}(a,1)}$. 

To construct the variety $Y$, consider the moduli space of 
$B:=\mathcal{B}(1)$ of locally free null correlation bundles 
on $\p3$. This is well known to be isomorphic to $\mathbb{P}
^5\smallsetminus G(2,4)$, where $G(2,4)$ is the Pl\"ucker 
hyperquadric (see, e.g., \cite[Thm. 4.3.4]{okonek}). 
Moreover, on $\mathbf{B}=\p3\times B$ there is the universal 
family $\boldsymbol{\mathcal{N}}$ of null correlation 
bundles. Consider the vector bundle $\boldsymbol{\mathcal{E}}=
V_2\otimes\calo_{\mathbf{B}}\oplus\boldsymbol{\mathcal{N}}$
and denote $E_b=\boldsymbol{\mathcal{E}}|_{\p3\times\{b\}}$, 
$N_b=\boldsymbol{\mathcal{N}}|_{\p3\times\{b\}},\ b\in B$,
so that
\begin{equation}\label{loc split new}
E_b=V_2\otimes\calo_{\p3}\oplus N_b,\ \ \ \ \ \ b\in B.
\end{equation} 
By linear algebra, there are canonical isomorphisms
$\boldsymbol{\varphi}_{(1)}:V_2\otimes\calo_{\mathbf{B}}
\xrightarrow{\simeq}V_2^{\vee}\otimes\wedge^2V_2\otimes\calo_{
\mathbf{B}}$ and $\boldsymbol{\varphi}_{(2)}:\boldsymbol{
\mathcal{N}}\xrightarrow{\simeq}\boldsymbol{\mathcal{N}}^{\vee}\otimes
\wedge^2\boldsymbol{\mathcal{N}}$. The sheaf $\boldsymbol{
\mathcal{N}}$ fits in the exact triple $0\to\op3\boxtimes\calo
_B(-1)\to\Omega_{\p3}^1(1)\boxtimes\calo_B\to\boldsymbol{
\mathcal{N}}\to0$ globalizing \eqref{nul corr}, so that $\wedge
^2\boldsymbol{\mathcal{N}}\simeq\op3\boxtimes\calo_B(1)$. 
(Here we set $\calo_B(\pm1):=\calo_{G(2,4)}(\pm1)|_B$.)
Consider the varieties $B_1:=\mathbf{V}(\wedge^2V_2^{\vee}
\otimes\calo_B)\smallsetminus\{0-\text{section}\}\xrightarrow{
\pi_1}B$ and $B_2:=\mathbf{V}(\calo_B(-1))\smallsetminus\{0-
\text{section}\}\xrightarrow{\pi_2}B$. Note that the pullback 
of a line bundle onto its total space with the 0-section 
removed trivializes this bundle, we obtain $\pi_2^*\calo_B(1)
\simeq\calo_{B_2}$, hence $(\wedge^2\boldsymbol{\mathcal{N}})
_{\mathbf{B}_2}\simeq\calo_{\mathbf{B}_2}$. Similarly, 
$(\wedge^2V_2\otimes\calo_{\mathbf{B}})_{\mathbf{B}_1}\simeq
\calo_{\mathbf{B}_1}$. Thus,  we obtain the symplectic
structures
\begin{equation*}\label{new sympl str}
\boldsymbol{\varphi}_{\mathbf{B}_1}:=(\boldsymbol{\varphi}_{(1)
})_{\mathbf{B}_1}:V_2\otimes\calo_{\mathbf{B}_1}\xrightarrow{
\simeq}V_2^{\vee}\otimes\calo_{\mathbf{B}_1},\ \ \ \ \ \ \varphi_{
\mathbf{B}_2}:=(\boldsymbol{\varphi}_{(2)})_{\mathbf{B}_2}:
\boldsymbol{\mathcal{N}}_{\mathbf{B}_2}\xrightarrow{\simeq}
\boldsymbol{\mathcal{N}}_{\mathbf{B}_2}^{\vee}.
\end{equation*}
Consider the variety 
$\widetilde{B}:=B_1\times_BB_2$. 
On $\widetilde{\mathbf{B}}$ we obtain from $\boldsymbol{
\mathcal{E}}$ a vector bundle $\boldsymbol{\mathcal{E}}_{
\widetilde{\mathbf{B}}}$ with the symplectic structure $\boldsymbol{\varphi}_{\widetilde{
\mathbf{B}}}$, where
\begin{equation}\label{E cal B}
\boldsymbol{\mathcal{E}}_{\widetilde{\mathbf{B}}}=
V_2\otimes\calo_{\widetilde{\mathbf{B}}}
\oplus\boldsymbol{\mathcal{N}}_{\widetilde{\mathbf{B}}},
\ \ \ \ \ 
\boldsymbol{\varphi}_{\widetilde{\mathbf{B}}}=\boldsymbol{
\varphi}_1\oplus\boldsymbol{\varphi}_2:\ \boldsymbol{
\mathcal{E}}_{\widetilde{\mathbf{B}}}\to\boldsymbol{
\mathcal{E}}^{\vee}_{\widetilde{\mathbf{B}}},
\end{equation}
and
$\boldsymbol{\varphi}_1:=(\varphi_{\mathbf{B}_1})
_{\widetilde{\mathbf{B}}}: 
V_2\otimes\calo_{\widetilde{\mathbf{B}}}\xrightarrow{
\simeq}V_2^{\vee}\otimes\calo_{\widetilde{\mathbf{B}}},
\ 
\boldsymbol{\varphi}_2:=(\varphi_{\mathbf{B}_2})
_{\widetilde{\mathbf{B}}}:
\boldsymbol{\mathcal{N}}_{\widetilde{\mathbf{B}}}\xrightarrow
{\simeq}\boldsymbol{\mathcal{N}}_{\widetilde{\mathbf{B}}}
^{\vee}$. 
By the above, we have the following description of the 
varieties $B_1$, $B_2$ and $\widetilde{B}$:
\begin{equation}\label{descrn cal B1,B2,B}
\begin{split}
& B_1=\{(b,\varphi_1)~|~b\in B,~\varphi_1:V_2\otimes\op3
\xrightarrow{\simeq}V_2^{\vee}\otimes\op3\ \text{is a symplectic structure}\},\\
& B_2=\{(b,\varphi_2)~|~b\in B,~\varphi_2:N_b\xrightarrow
{\simeq}N_b^{\vee}\ \textrm{is a symplectic structure}\},\\
& \widetilde{B}=\{(b,\varphi_1,\varphi_2)~|~(b,\varphi_i)\in
B_i,\ i=1,2\}.
\end{split}
\end{equation}
The following constructions (see \eqref{def Ya}-\eqref{H0=H0 for 
Y}) are parallel to the constructions 
\eqref{def M(a)}-\eqref{H0=H0}. Twisting the equality \eqref{loc 
split new} by $\op3(a)$, we obtain as in \eqref{h0(E(a))}: 
$h^0(E_b(a))=4\binom{a+3}{3}-a-2,\ h^i(E_b(a))=0,\ i>0$. Thus, as 
in \eqref{F}, the sheaf $F_B=p_{B*}(\boldsymbol{\mathcal{E}}(a,0)
)$ is a locally free $\calo_B$-sheaf of rank $r=h^0(E_b(a))$. 
Consider the variety $\mathcal{T}:=\mathbf{P}(F_B^{\vee})$. 
Similarly to \eqref{M(a,k)} we have
\begin{equation}\label{cal bf Ya}
\mathcal{T}=\{(b,\langle\sigma\rangle)\ |\ 
b\in B,\ 0\ne\sigma\in H^0(E_b(a))\}.
\end{equation}
For any point $(b,\langle\sigma\rangle)\in\mathcal{T}$ in view 
of \eqref{loc split new} we may represent 
$\sigma$ as a pair $\sigma=(\sigma_1,\sigma_2),\ \sigma_1\in 
H^0(V_2\otimes\op3(a)),\ \sigma_2\in H^0(N_b(a))$. 
Thus, using the notation \eqref{new notation} we can rewrite
\eqref{cal bf Ya} as $\mathcal{T}=\{(b,[\sigma_1:\sigma_2])\ |\ 
b\in B,\ [\sigma_1:\sigma_2]\in P(H^0(E_b(a)))\}$. On the other 
hand, representing $\sigma$ as a morphism $\sigma:\op3(-a)\to E_b
$, we see that, when $(b,\langle\sigma\rangle)$ runs through 
$\mathcal{T}$, the morphisms $\sigma$, as in \eqref{bold sigma}, 
globalize to a morphism $\boldsymbol{\sigma}_{\boldsymbol{
\mathcal{T}}}:\ \op3(-a)\boxtimes L_{\mathcal{T}}^{\vee}\to
\boldsymbol{\mathcal{E}}_{\boldsymbol{\mathcal{T}}}$ on 
$\boldsymbol{\mathcal{T}}$, where $L_{\mathcal{T}}$ is the 
Grothendieck sheaf $\calo_{\mathcal{T}/B}(1)$. Next, similar to 
\eqref{def M(a)}, we define an open subset $\mathcal{S}$ of 
$\mathcal{T}$ as
\begin{equation}\label{def Ya}
\mathcal{S}:=\{(b,[\sigma_1:\sigma_2])\in\mathcal{T}\ |\ 
(i)\ (\sigma_1,\sigma_2):\op3(-a)\to E_m\ \textrm{is a subbundle morphism and}\ (ii)\ \sigma_1,\sigma_2\ne0\}.
\end{equation}
Note that $\mathcal{S}$ is a nonempty set. (The proof mimics that 
of nonemptiness of the subset $M$ of $T$ given in paragraph after 
\eqref{def M(a)}.) By the Base Change, the sheaf $F_{\widetilde{B}
}=p_{\widetilde{B}*}(\boldsymbol{\mathcal{E}}_{\widetilde{\mathbf{
B}}}(a,0))$ is isomorphic to the sheaf $(F_B)_{\widetilde{B}}$. 
Therefore, from the definition of $\mathcal{T}$ it follows that 
the variety $\widetilde{Y}:=\mathbf{P}(F_{\widetilde{B}}^{\vee})$ 
is isomorphic to $\widetilde{B}\times_B\mathcal{T}$:
\begin{equation}\label{def bold Ba}
\widetilde{Y}\simeq\widetilde{B}\times_B\mathcal{T}.
\end{equation}
Thus by \eqref{descrn cal B1,B2,B} and \eqref{cal bf Ya} we have
$\widetilde{Y}=\{(b,\varphi_1,\varphi_2,[\sigma_1:\sigma_2])\ |\ 
(b,\varphi_1,\varphi_2)\in \widetilde{B},\ [\sigma_1:\sigma_2]\in 
P(H^0(E_b(a)))\}$, and the natural projection $\widetilde{Y}\to 
\widetilde{B},\ (\beta,\langle\sigma\rangle)\mapsto\beta$ is a 
locally trivial $\mathbb{P}^{r-1}$-bundle. We now use \eqref{def 
bold Ba} and the open subset $\mathcal{S}$ of $\mathcal{T}$ to 
define an open subset $Y$ of $\widetilde{Y}$ as
\begin{equation}\label{def Y}
Y:=\widetilde{B}\times_B\mathcal{S}.
\end{equation}
Here, $Y$ is a nonempty open in $\widetilde{Y}$ since $\mathcal{S}
$ is nonempty. It follows that $Y$ is irreducible and dense in 
$\widetilde{Y}$ since $\widetilde{Y}$ is irreducible. In addition,
using \eqref{def Ya} and the above description of $\widetilde{Y}$ 
we obtain:
\begin{equation}\label{descrn Y}
Y=\{(b,\varphi_1,\varphi_2,[\sigma_1:\sigma_2])\in\widetilde{Y}\ 
|(i)\ (\sigma_1,\sigma_2):\op3(-a)\to E_m\ \textrm{is a subbundle morphism and}\ (ii)\ \sigma_1,\sigma_2\ne0\}.
\end{equation}
The morphism $\boldsymbol{\sigma}_{\mathbf{Y}}:=(\boldsymbol{
\sigma}_{\boldsymbol{\mathcal{T}}})_{\mathbf{Y}}$ is included in 
the universal monad on $\mathbf{Y}$:
\begin{equation}\label{AY}
\mathbf{A}_{\mathbf{Y}}:\ 0\to\op3(-a)\boxtimes L_Y^{\vee}
\xrightarrow{\boldsymbol{\sigma}_{\mathbf{Y}}}\boldsymbol{
\mathcal{E}}_{\mathbf{Y}}\xrightarrow{\boldsymbol{\sigma}
_{\mathbf{Y}}^t}\op3(a)\boxtimes L_Y\to0,
\end{equation}
where $L_Y=(L_{\mathcal{T}})_Y$ and $\boldsymbol{
\sigma}_Y^t$ is the composition 
$\boldsymbol{\mathcal{E}}_{\mathbf{Y}}
\xrightarrow{\boldsymbol{\varphi}_{\mathbf{Y}}}
\boldsymbol{\mathcal{E}}_{\mathbf{Y}}^{\vee}
\xrightarrow{\boldsymbol{\sigma}_{\mathbf{Y}}^{\vee}}\op3(a)
\boxtimes L_{Y}$. By construction, for any 
point $(\beta,\langle\sigma\rangle)\in Y$, $\beta=(b,\varphi_1,
\varphi_2)$, the restriction of the monad $\mathbf{A}
_{\mathbf{Y}}$ onto $\p3\times\{(\beta,\langle\sigma\rangle)\}$ 
is isomorphic to the monad $A_{E_b,\varphi_1\oplus
\varphi_2,\sigma}$ in \eqref{mon M}. Hence,
\begin{equation}\label{H0=H0 for Y}
\mathcal{H}^0(\mathbf{A}_{\mathbf{Y}})|_{\p3\times\{(\beta,
\langle\sigma\rangle)\}}=\mathcal{H}^0(A_{E_b,\varphi_1\oplus
\varphi_2,\sigma}),\ \ \ \ \ \ (\beta,\langle\sigma\rangle)\in Y,
\ \ \ \ \ \beta=(b,\varphi_1,\varphi_2).
\end{equation}

Now consider the rank 2 the vector bundle $\mathbf{U}$ on $M$ 
defined in \eqref{bf cal E} and its associated principal frame 
bundle 
\begin{equation*}
I:=\mathbf{Isom}(V_2\otimes\calo_{M},\mathbf{U})
\xrightarrow{\xi}M
\end{equation*}
together with the tautological isomorphism $V_2\otimes\calo_{I}
\xrightarrow{\sim}\mathbf{U}_I$.
Using this isomorphism and applying to \eqref{glob split} the 
functor $\boldsymbol{\xi}^*$ we obtain an isomorphism
\begin{equation}\label{EX=dir sum}
\mathbf{E}_{\mathbf{I}}\cong V_2\otimes\calo
_{\mathbf{I}}\oplus\mathbf{N}_{\mathbf{I}}.
\end{equation}
Besides, by \eqref{bf cal E}, we have a symplectic structure
$\boldsymbol{\varphi}_{\mathbf{I}}:=(\boldsymbol{\varphi}
_{\mathbf{M}})_{\mathbf{I}}:\mathbf{E}_{\mathbf{I}}\xrightarrow{
\simeq}\mathbf{E}_{\mathbf{I}}^{\vee}$ on $\mathbf{E}_{\mathbf{
I}}$. This symplectic structure in view of \eqref{EX=dir sum} 
splits into a direct sum of two symplectic structures
\begin{equation}\label{phi X=dir sum}
\boldsymbol{\varphi}_{\mathbf{I}}=\boldsymbol{\varphi}
_{\mathbf{I},1}\oplus\boldsymbol{\varphi}_{\mathbf{I},2},
\ \ \ \ \ \ \ 
\boldsymbol{\varphi}_{\mathbf{I},1}:V_2\otimes\calo
_{\mathbf{I}}\xrightarrow{\simeq}V_2^{\vee}\otimes\calo
_{\mathbf{I}}, \ \ \ \ \ \ 
\boldsymbol{\varphi}_{\mathbf{I},2}:\mathbf{N}_{\mathbf{I}}
\xrightarrow{\simeq}\mathbf{N}_{\mathbf{I}}^{\vee}.
\end{equation}
Remark that, by the defscription of the morphism $\Psi$ given 
in \eqref{Psi=kf}, we have $\Psi(M)=B$. Now, comparing \eqref{E 
cal B}-\eqref{descrn cal B1,B2,B} with 
\eqref{EX=dir sum}-\eqref{phi X=dir sum}, we obtain a morphism
\begin{equation}\label{Gamma}
\Gamma:I\to\widetilde{B},\ \ \ \ x\mapsto(b,\varphi_1,
\varphi_2),\ \ \ \ \ b=\Psi(\xi(x)),\ \ \ \ \ \varphi_i=
\boldsymbol{\varphi}_{\mathbf{I},i}|_{\p3\times\{x\}},\ i=1,
2,
\end{equation}
such that
\begin{equation}\label{E cal X}
\mathbf{E}_{\mathbf{I}}\cong(\boldsymbol{\mathcal{E}}
_{\widetilde{\mathbf{B}}})_{\mathbf{I}},\ \ \ \ \ \ 
\boldsymbol{\varphi}_{\mathbf{I}}\cong(\boldsymbol{\varphi}_{
\widetilde{\mathbf{B}}})_{\mathbf{I}},
\end{equation}
and these isomorphisms are compatible with the direct sum
decompositions \eqref{EX=dir sum}, \eqref{phi X=dir sum} and
\eqref{E cal B}. From \eqref{Gamma} and 
the surjectivity of $\Psi$ it follows that $\Gamma$ is also 
surjective. Set
\begin{equation}\label{def X etc}
X:=I\times_M S,\ \ \ \ \ \ \ \ \ \ \
Y\xleftarrow{\Gamma_Y}X\xrightarrow{\xi_S}S,\ \ \ \ \ \ \ \ 
\ \ \  F_I:=p_{I*}(\mathbf{E}_{\mathbf{I}}(a,0)).
\end{equation}
From \eqref{F}, \eqref{E cal X}, the isomorphism $F_{\widetilde{B}
}\simeq(F_B)_{\widetilde{B}}$ and the Base Change we obtain
$F_I\simeq (F_{\widetilde{B}})_I$, so that, in view of 
\eqref{def bold Ba} and the equality $T=\mathbf{P}(F^{\vee})$, 
the variety $\widetilde{X}:=\mathbf{P}(F_X^{\vee})$ satisfies the 
isomorphisms
\begin{equation}\label{two fibered prod}
I\times_{M}T\simeq\widetilde{X}\simeq I\times_{\widetilde{B}}
\widetilde{Y}.
\end{equation}
The definition of $X$ (see \eqref{def X etc}) and the left 
isomorphism \eqref{two fibered prod} imply that there exists an 
open embedding $X\hookrightarrow\widetilde{X}$ such that $X=
\widetilde{X}\times_{T}S$. Therefore, comparing the descriptions 
\eqref{descrn Y} and \eqref{def M(a)} of $Y$ and $S$ and using 
the right isomorphism \eqref{two fibered prod}, we obtain:
\begin{equation}\label{X is fiber product}
X\simeq I\times_{\widetilde{B}}Y.
\end{equation} 
This together with \eqref{E cal X} implies that $\mathbf{E}_{
\mathbf{X}}\cong(\boldsymbol{\mathcal{E}}_{\mathbf{Y}})_{\mathbf{
X}}$. Moreover, since $X=I\times_MS$, we have
\begin{equation}\label{AX AY}
\boldsymbol{\mathcal{A}}_{\mathbf{X}}\cong(\mathbf{A}_{\mathbf{Y}
})_{\mathbf{X}},
\end{equation}
where the monads $\boldsymbol{\mathcal{A}}$ and 
$\mathbf{A}_{\mathbf{Y}}$ were defined in \eqref{bf A} and 
\eqref{AY}, respectively. Consider the modular morphisms
\begin{equation}\label{PhiX,PhiY} 
\Phi_X:X\to\mathcal{B}(a^2+1),\ \ \ \ \ \ \ \ 
\Phi_Y:Y\to\mathcal{B}(a^2+1),
\end{equation}
defined by the (families of) sheaves $\mathcal{H}^0(
\boldsymbol{\mathcal{A}}_{\mathbf{X}})$, $\mathcal{H}^0(
\mathbf{A}_{\mathbf{Y}})$, respectively. From \eqref{AX AY}, 
\eqref{X is fiber product} and \eqref{def X etc} it follows that 
$\Phi_X$ factors through $\Gamma_Y$ and through $\xi_S$ as:
$\Phi_X=\Phi_Y\circ\Gamma_Y=\Phi_{S}\circ\xi_S.$
Here, $\Phi_{S}:S\to\calb(a^2+1)$ is the modular morphism
\eqref{Phi S,Phi tilde S}, $\xi_S$ in \eqref{def X etc} is 
surjective by the surjectivity of $\xi$, and $\Gamma_Y$ is 
surjective as $\Gamma$ is surjective. Hence,
\begin{equation}
\mathcal{G}(a,1)_0=\Phi_{S}(S)=
\Phi_Y(Y).
\end{equation} 
On the other hand, by Proposition \ref{descrn of G(a,1)}, 
$\mathcal{G}(a,1)_0$ is dense in $\overline{\mathcal{G}(a,1)}$.
We thus obtain
\begin{Prop}\label{Prop 15}
Let $\Phi_Y:Y\to\mathcal{B}(a^2+1)$ be the modular morphism
defined by the family of sheaves $\mathcal{H}^0(\mathbf{A}
_{\mathbf{Y}})$, where $\mathbf{A}_{\mathbf{Y}}$ is the 
monad \eqref{AY}. Then $\Phi_Y(Y)$ is dense in $\overline{
\mathcal{G}(a,1)}$.
\end{Prop}

\vspace{2mm}

\section{Series of rational irreducible components of the
moduli spaces $\calb(a^2+1)$}\label{Rationality}

\vspace{2mm}

Consider the variety $Y$ defined in \eqref{def Y}. We first will 
relate to $Y$ a new variety $\mathcal{P}_a$, together with a 
natural projection $\pi:\ Y\to \mathcal{P}_a$. In this section we 
will relate the morphism $\pi$ to the modular morphism $\Phi_Y:Y
\to\mathcal{B}(a^2+1)$ (for the precise formulation see Theorem 
\ref{Thm 18}). For this, take any point $y\in Y$. By 
\eqref{descrn Y}, $y$ is a collection of data
$$
y=(b,\varphi_1,\varphi_2,[\sigma_1:\sigma_2]),
$$
where (i) $b\in B$, (ii) $\varphi_1:V_2\otimes\op3\xrightarrow{
\simeq}V_2^{\vee}\otimes\op3$ and $\varphi_2:N_b\xrightarrow{
\simeq}N_b^{\vee}$ are symplectic isomorphisms:
\begin{equation}\label{phi1}
\varphi_1\in H^0(\wedge^2(V_2\otimes\op3)^{\vee})
\smallsetminus\{0\}=\wedge^2V_2^{\vee}\smallsetminus\{0\}
\cong\mathbf{k}^{\times},\ \ \ \ \ \ 
\varphi_2\in H^0(\wedge^2N_b^{\vee})\smallsetminus\{0\}=
H^0(\op3)\smallsetminus\{0\}\cong\mathbf{k}^{\times},
\end{equation}
(iii) $\sigma_1$ and $\sigma_2$ are:
\begin{equation}\label{sigma1}
0\ne\sigma_1\in H^0(V_2\otimes\op3(a))=\Hom(V_2^{\vee},W_a),
\ \ \ \ \ \ W_a:=H^0(\op3(a)),\ \ \ \ \ \ \ \ 
0\ne\sigma_2\in H^0(N_b(a)),
\end{equation}
(iv) $\sigma=(\sigma_1,\sigma_2):\op3(-a)
\to V_2\otimes\op3\oplus N_b$ is a
subbundle morphism.
In $\Hom(V_2^{\vee},W_a)$ consider an open subset
$\Hom^{\mathrm{in}}(V_2^{\vee},W_a):=\{\sigma_1\in \Hom(V_2^{
\vee},W_a)\ |\ \sigma_1:V_2^{\vee}\to W_a\ \textrm{is a 
monomorphism}\}$. One can easily see (use the argument in 
paragraph after \eqref{def M(a)}) that 
\begin{equation*}
\Hom^{\mathrm{in}}(V_2^{\vee},W_a)=\{\sigma_1\in \Hom(V_2^{
\vee},W_a)\ |\ \dim (\sigma_1)_0=1\}.
\end{equation*}
Besides, note that the group $GL(V_2)$ naturally acts on 
$\Hom^{\mathrm{in}}(V_2^{\vee},W_a)$ via its action on 
$V_2^{\vee}$, and we have an isomorphism
\begin{equation}\label{Gr}
\Hom^{\mathrm{in}}(V_2^{\vee},W_a)/GL(V_2)\xrightarrow{\simeq
}Gr(2,W_a)
\end{equation}
and the factorization morphism 
\begin{equation}\label{tau 1}
\tau_1:\Hom^{\mathrm{in}}(V_2^{\vee},W_a)\to Gr(2,W_a),\ \ \ 
\sigma_1\mapsto\mathrm{im}(\sigma_1:V_2^{\vee}\hookrightarrow
W_a).
\end{equation}
Next, as it was mentioned in Section \ref{G(a,1)} (see paragraph 
after \eqref{def M(a)}), the set $H^0(N_b(a))^*:=\{\sigma_2\in 
H^0(N_b(a))\ |\ \dim (\sigma_2)_0$
$=1\}$ is open dense in $H^0(N_b(a)
)$. Besides, it is clearly invariant under the action of the 
group $\mathrm{Aut}(N_b(a))=\mathbf{k}^{\times}$. (Recall that 
the null correlation bundle $N_b$ is stable and therefore simple, 
i. e., $\mathrm{End}(N_b(a))=\mathbf{k}
\cdot\mathrm{id}$.) Hence,
\begin{equation}\label{PH0Nb(a)}
P(H^0(N_b(a)))^*=H^0(N_b(a))^*/\mathrm{Aut}(N_b(a))\overset{
\mathrm{open}}{\hookrightarrow}P(H^0(N_b(a)))\simeq\mathbb{P}^r,
\end{equation}
where $r=2\binom{a+3}{3}-a-3$, and we have the factorization 
morphism
\begin{equation}\label{tau 2}
\tau_2:H^0(N_b(a))^*\to P(H^0(N_b(a)))^*,\ \ \ \ \ 
\sigma_2\mapsto\langle\sigma_2\rangle.
\end{equation}
Now the above condition (iv) imposed on $(\sigma_1,\sigma_2)$
can be rewritten in the form:
\begin{equation}\label{def Hba}
(\sigma_1,\sigma_2)\in H_{b,a}:=\{(\sigma_1,\sigma_2)\in
\Hom^{\mathrm{in}}(V_2^{\vee},W_a)\times H^0(N_b(a))^*\ |\  
(\sigma_1)_0\cap(\sigma_2)_0=\emptyset\}.
\end{equation}
Clearly, $H_{b,a}$ is a dense open subset of $\Hom^{\mathrm{
in}}(V_2^{\vee},W_a)\times H^0(N_b(a))^*$. This subset is
invariant under the action of the group $\mathbf{k}^{\times}$
by homotheties. Therefore, denoting 
$P(H_{b,a}):=H_{b,a}/\mathbf{k}^{\times}$
and using \eqref{tau 1} and \eqref{tau 2}, we obtain the 
factorization morphism 
\begin{equation}\label{tau}
\tau:\ P(H_{b,a})\to Gr(2,W_a)\times\mathbb{P}(H^0(N
_b(a)))^*,\ \ \ \ 
[\sigma_1:\sigma_2]\mapsto(\tau_1(\sigma_1),
\tau_2(\sigma_2)).
\end{equation}
To globalize the above pointwise (w.r.t. $b\in B$) constructions 
over $B$, set $\mathcal{K}:=p_{B*}(\boldsymbol{\mathcal{N}}(a,0))
$. The variety $\mathbf{P}(\mathcal{K}^{\vee})$ has the 
description $\mathbf{P}(\mathcal{K}^{\vee})=\{(b,\langle\sigma_2
\rangle)\ |\ b\in B,\ \langle\sigma_2\rangle\in P(H^0(N_b(a)))\}$.
Consider its dense open subset
$$
\Pi_a:=\{(b,\langle\sigma_2\rangle)\in\mathbf{P}(\mathcal{K}
^{\vee})~|~\langle\sigma_2\rangle\in\mathbb{P}(H^0(N_b(a)))^*\}
$$
and set
\begin{equation}\label{Ga}
\mathcal{G}_a:=Gr(2,W_a)\times\Pi_a,\ \ \ \ \ \ 
\mathcal{G}_a=\{(b,V,\langle\sigma_2\rangle)~|~V\in Gr(2,W_a),
(b,\langle\sigma_2\rangle)\in\Pi_a\}.
\end{equation}
By construction, $\mathcal{G}_a$ is a rational variety. Next, 
remark that, comparing the definitions \eqref{def Ya} and 
\eqref{def Hba} of $\mathcal{S}$ and $H_{b,a}$, we obtain
$$
\mathcal{S}=\{(b,[\sigma_1:\sigma_2])~|~b\in B,~[\sigma_1:
\sigma_2]\in\mathbb{P}(H_{b,a})\}.
$$
Thus, by \eqref{tau}, we have a well-defined morphism
\begin{equation}\label{tau global}
\tau:\ \mathcal{S}\to\mathcal{G}_a,\ (b,[\sigma_1:\sigma_2])
\mapsto(b,\tau_1(\sigma_1),\tau_2(\sigma_2)).
\end{equation}
Consider the group 
$\tilde{G}=GL(V_2)\times\mathbf{k}^{\times}
$, its normal subgroup $G'=\{(\rho\cdot\mathrm{id}_{V_2},\rho)
~|~\rho\in\mathbf{k}^{\times}\}$, and let 
\begin{equation}\label{group G}
G=\tilde{G}/G'
\end{equation}
be the factor group. We will use the following notation for 
elements of $G$:\ \ $[g_1:\lambda]:=(g_1,\lambda)G'=\{(\rho 
g_1,\rho\lambda)~|~\rho\in\mathbf{k}^{\times}\},\ \ (g_1,\lambda)
\in\tilde{G}$. The group $G$ naturally acts on 
$\mathcal{S}$ as:
\begin{equation}\label{G-action on Ya}
a_{\mathcal{S}}:\ \mathcal{S}\times G\to\mathcal{S},\ \ \ 
\ ((b,[\sigma_1:\sigma_2]),[g_1:\lambda])\mapsto
(b,[g_1\circ\sigma_1:\lambda\sigma_2]),
\end{equation}
and formulas \eqref{Gr}-\eqref{tau global} show that
$\mathcal{G}_a=\mathcal{S}/G$
and the morphism $\tau:\ \mathcal{S}\to\mathcal{G}_a$ in 
\eqref{tau global} is the quotient morphism for this action
and it is a principal $G$-bundle. Therefore in view of 
\eqref{h0(E(a))} we have:
\begin{equation}\label{dim Ga}
\dim\mathcal{G}_a=\dim P(H_{b,a})+\dim B-\dim G=
4\binom{a+3}{3}-a-2.
\end{equation}
The principal $G$-bundle $\mathcal{S}\xrightarrow{\tau}\calg_a$ 
by construction is locally trivial, hence there exists an open 
dense subset $U$ of $\calg_a$ and a section $U\overset{s}
{\hookrightarrow}\mathcal{S}$ of the projection $\tau|_{\tau^{-1}
(U)}:\tau^{-1}(U)\to U$:
\begin{equation}\label{cartesian 2}
\xymatrix{
& \mathcal{S}\ar[d]^{\tau}\\
U \ar@{^{(}->}[ur]^-{s}\ar@{^{(}->}[r]^{\textrm{open}} & 
\mathcal{G}_a.}
\end{equation} 
Here $U$ is rational since $\mathcal{G}_a$ is rational as it
was mentioned above.

Now consider the variety
$\mathbf{P}(\wedge^2(V_2\otimes\calo_{\mathbf{B}})\oplus
\wedge^2\boldsymbol{\mathcal{N}})$ together with the
embeddings
$$
\mathbf{P}(\wedge^2(V_2\otimes\calo_{\mathbf{B}}))
\hookrightarrow\mathbf{P}(\wedge^2(V_2\otimes\calo_{\mathbf{B
}})\oplus\wedge^2\boldsymbol{\mathcal{N}})
\hookleftarrow\mathbf{P}(\wedge^2\boldsymbol{\mathcal{N}})
$$
and denote $P\widetilde{B}:=\mathbf{P}(\wedge^2(V_2\otimes\calo
_{\mathbf{B}})\oplus\wedge^2\boldsymbol{\mathcal{N}})
\smallsetminus\{\mathbf{P}(\wedge^2(V_2\otimes\calo_{\mathbf{
B}}))\sqcup\mathbf{P}(\wedge^2\boldsymbol{\mathcal{N}})\}$.
By construction, the natural projection $P\widetilde{B}\to B$ is 
a locally trivial fibration with fiber
\begin{equation}\label{fiber T}
\mathbf{F}\simeq\mathbb{P}^1\smallsetminus\{\textrm{2 
points}\}.
\end{equation}
Using the description \eqref{descrn cal B1,B2,B} of the
varieties $B_1$, $B_2$ and the notation \eqref{new notation} 
in which we put $\varphi_1,\varphi_2$ in place of $\sigma_1,
\sigma_2$, we obtain $\mathbb{P}\widetilde{B}=\{(b,[\varphi_1:
\varphi_2])~|~(b,\varphi_i)\in B_i,~i=1,2\}$. Remark that the 
group $\mathbf{k}^{\times}$ naturally acts on
$\widetilde{B}$ as
\begin{equation}\label{k*-action}
\widetilde{B}\times\mathbf{k}^{\times}\to\widetilde{B},\ \ \ 
\ \ ((b,\varphi_1,\varphi_2),\lambda)\mapsto(b,\lambda
\varphi_1,\lambda\varphi_2),
\end{equation}
(here we use the description \eqref{descrn cal B1,B2,B} of 
$\widetilde{B}$), so that 
\begin{equation}\label{quot by k*}
P\widetilde{B}=\widetilde{B}/\mathbf{k}^{\times},
\end{equation}
and we have the factorization morphism
\begin{equation}\label{piB}
\pi_{\widetilde{B}}:\ 
\widetilde{B}\to P\widetilde{B},\ \ \ 
\  (b,\varphi_1,\varphi_2)\mapsto(b,[\varphi_1:\varphi_2]).
\end{equation}
Consider the varieties $PY:=P\widetilde{B}\times_B\mathcal{S}=\{
(b,[\varphi_1:\varphi_2],[\sigma_1:\sigma_2])~|~(b,[\varphi_1:
\varphi_2])\in P\widetilde{B},~(b,[\sigma_1:\sigma_2])\in\mathcal{
S}\}$ and $\calp_a:=P\widetilde{B}\times_B\mathcal{G}_a=\{(b,
[\varphi_1:\varphi_2],V,\langle\sigma_2\rangle)~|~(b,[\varphi_1:
\varphi_2])\in P\widetilde{B},~(b,V,\langle\sigma_2\rangle)\in
\mathcal{G}_a\}$, where $\mathcal{G}_a$ was defined in 
\eqref{Ga}. From \eqref{dim Ga} and \eqref{fiber T} we have
\begin{equation}\label{dim Pa}
\dim\calp_a=\dim\mathcal{G}_a+\dim 
\mathbf{F}=4\binom{a+3}{3}-a-1.
\end{equation}
Note that the local triviality of the fibration $P\mathcal{B}\to B$ implies that the natural projection
\begin{equation}\label{prY}
pr_{\mathcal{Y}}:\ PY\to\mathcal{S}
\end{equation}
is a locally trivial fibration with fiber $\mathbf{F}$ given 
in \eqref{fiber T}.

The morphism $\pi_{\widetilde{B}}$ in \eqref{piB} induces 
the morphism
\begin{equation}\label{piY}
\pi_Y:\ Y\to PY,\ (b,\varphi_1,\varphi_2,[\sigma_1:
\sigma_2])\mapsto(b,[\varphi_1:\varphi_2],[\sigma_1:\sigma_2]
),
\end{equation}
and from \eqref{k*-action}-\eqref{piB} it follows that 
$\pi_Y$ is a factorization morphism of the following 
$\mathbf{k}^{\times}$-action on $Y$:
\begin{equation}\label{k*-action on Y}
a_Y:\ Y\times\mathbf{k}^{\times}\to Y,\ \ \ \ \ \ 
((b,\varphi_1,\varphi_2,[\sigma_1:\sigma_2]),\lambda)\mapsto
(b,\lambda\varphi_1,\lambda\varphi_2,[\sigma_1:\sigma_2]).
\end{equation}
Respectively, the morphism $\tau:Y_a\to\mathcal{G}_a$ defined
in \eqref{tau global} induces a morphism
\begin{equation}\label{tauY}
\tau_Y:\ PY\to\mathcal{P}_a,\ (b,[\varphi_1:\varphi
_2],[\sigma_1:\sigma_2])\mapsto (b,[\varphi_1:\varphi
_2],\tau_1(\sigma_1),\tau_2(\sigma_2)).
\end{equation}
Define the morphism $\pi:Y\to\mathcal{P}_a$ as the composition
\begin{equation}\label{quotient morphism pi}
\pi=\tau_Y\circ\pi_Y:\ Y\to\mathcal{P}_a,\ \ \ \ \ 
(b,\varphi_1,\varphi_2,[\sigma_1:\sigma_2])\mapsto
(b,[\varphi_1:\varphi_2],\tau_1(\sigma_1),\tau_2(\sigma_2)).
\end{equation}
We will now proceed to the study of the fibers of the morphism $\pi$.
\begin{Def}\label{equiv on Y}
Introduce on $Y$ the following equivalence relation:
\begin{equation}\label{equiv reln1}
y=(b,\varphi_1,\varphi_2,[\sigma_1:\sigma_2])\ \sim\ 
(\tilde{b},\tilde{\varphi}_1,\tilde{\varphi}_2,
[\tilde{\sigma}_1:\tilde{\sigma}_2])=\tilde{y}
\end{equation}
if there exists an isomorphism of symplectic monads
$A_y$
and $A_{\tilde{y}}$,
i. e., a commutative diagram with rows $A_y$ and $A_{\tilde{y}}$:
\begin{equation}\label{isom of monads}
\xymatrix{
A_y:~~\op3(-a)\ar@{>->}[r]^-{(\sigma_1,\sigma_2)}\ar[d]
_{h_-}^-{\simeq} & V_2\otimes\op3\oplus N_b
\ar@{>>}[rr]^-{(\sigma_1^{\vee}\circ\varphi_1,\sigma
_2^{\vee}\circ\varphi_2)}\ar[d]_{(g_1,g_2)}^-{\simeq} & & 
\op3(a)\ar[d]_-{h_+}^-{\simeq} \\
A_{\tilde{y}}:~~\op3(-a)\ar@{>->}[r]^-{(\tilde{\sigma}_1,
\tilde{\sigma}_2)} & V_2\otimes\op3\oplus 
N_{\tilde{b}}\ar@{>>}[rr]
^-{(\tilde{\sigma}_1^{\vee}\circ\tilde{\varphi}_1,\tilde{
\sigma}_2^{\vee}\circ\tilde{\varphi}_2)} & & \op3(a).}
\end{equation}
We denote by $[y]=[b,\varphi_1,\varphi_2,[\sigma_1:\sigma_2]]$ 
the equivalence class of a point $y=(b,\varphi_1,\varphi_2,
[\sigma_1:\sigma_2])\in Y$ under this equivalence relation.
\end{Def}
Note that, in diagram \eqref{isom of monads}, one has
\begin{equation}\label{g1 in GL}
g_1\in\mathrm{Isom}(V_2\otimes\op3,V_2\otimes\op3)\cong  
GL(V_2);
\end{equation}
and $g_2\in\mathrm{Isom}(N_b,N_{\tilde{b}})$ which
in view of the stability of $N_b$ implies that
\begin{equation}\label{g2 in k^x}
b=\tilde{b}, \ \ \ \ \ \ \ g_2=\lambda\cdot\mathrm{id}_{N_b},
\ \ \ \ \ \ \lambda\in\mathbf{k}^{\times};
\end{equation}
besides, the isomorphisms $h_-,\ h_+$ are multiplications
by some constants $\mu,\ \nu\in\mathbf{k}^{\times}$,
respectively:
\begin{equation}\label{h+-}
h_-=\mu\cdot\mathrm{id}_{\op3(-a)},\ \ \ \ \ \ 
h_+=\nu\cdot\mathrm{id}_{\op3(a)}.
\end{equation}
Furthermore, in view of \eqref{phi1}, \eqref{g1 in GL}, 
\eqref{g2 in k^x} and the symplecticity of $\varphi_1,\ \varphi_2
$, we have in \eqref{isom of monads}
\begin{equation}\label{lambda i}
\tilde{\varphi}_1=\lambda_1\varphi_1,\ \ \ \  
\tilde{\varphi}_2=\lambda_2\varphi_2, \ \ \ \  
\lambda_1,\lambda_2\in\mathbf{k}^{\times},\ \ \ \ 
g_1^{\vee}\circ\varphi_1\circ g_1=\det(g_1)\varphi_1,\ \ \ \  
g_2^{\vee}\circ\varphi_2\circ g_2=\lambda^2\varphi_2.
\end{equation}
The leftmost square of diagram \eqref{isom of monads} together with \eqref{h+-} impliies:
\begin{equation}\label{left reln}
\tilde{\sigma}_1=\frac{1}{\mu}g_1\circ\sigma_1,\ \ \ \ \ \ 
\tilde{\sigma}_2=\frac{\lambda}{\mu}\sigma_2.
\end{equation}
Respectively, the rightmost square of diagram
\eqref{isom of monads} yields $\nu\sigma_1^{\vee}\circ\varphi_1
=\tilde{\sigma}_1^{\vee}\circ\tilde{\varphi}_1\circ g_1,\ \ 
\nu\sigma_2^{\vee}\circ\varphi_2=\lambda\tilde{\sigma}_2^{\vee}
\circ\tilde{\varphi}_2$. Substituting 
\eqref{h+-}-\eqref{left reln} into the last equalities we obtain 
the relations $\nu=\frac{\lambda_1\det(g_1)}{\mu}$ and $\nu=
\frac{\lambda_2\lambda^2}{\mu}$. Whence $\lambda_1\det(g_1)=
\lambda_2\lambda^2$. This relation shows that the $G$-action 
\eqref{G-action on Ya} on $\mathcal{S}$ lifts to the following 
$G$-action on $PY$:
\begin{equation}\label{G-action on PY}
a_{PY}:\ PY\times G\to PY,\ \
((b,[\varphi_1:\varphi_2],[\sigma_1:\sigma_2]),[g_1:\lambda]
)\mapsto(b,[\frac{\varphi_1}{\det(g_1)}:\frac{\varphi_2}
{\lambda^2}],[g_1\circ\sigma_1:\lambda\sigma_2]).
\end{equation}
Thus, $\calp_a=PY/G$ and the morphism
\begin{equation}\label{tauY again}
\tau_Y:\ PY\to\mathcal{P}_a
\end{equation}
in \eqref{tauY} is the quotient morphism for this action and it 
is a locally trivial principal $G$-bundle. We therefore have a 
commutative diagram 
\begin{equation*}\label{cartesian 3}
\xymatrix{
PY\ar[r]^-{\tau_Y}\ar[d]^{pr_\mathcal{Y}}
 & \mathcal{P}_a\ar[d]^{pr_{\mathcal{G}}}\\
\mathcal{S} \ar[r]^-{\tau} & \mathcal{G}_a,}
\end{equation*} 
where $pr_{\mathcal{G}}$ is a natural projection. Since by
\eqref{prY} the morphism $pr_{\mathcal{Y}}:\ PY\to
\mathcal{S}$ is a locally trivial fibration with fibre 
$\mathbf{F}$, the open subset $U$ of $\calg_a$ and the 
section $U\overset{s}{\hookrightarrow}\mathcal{S}$
in the diagram \eqref{cartesian 2}, after possible shrinking 
$U$, can be lifted to an open section $\mathbf{F}\times 
U\overset{\tilde{s}}{\hookrightarrow}PY$ 
of the projection $\tau_Y:PY\to\mathcal{P}_a$:
\begin{equation*}\label{cartesian 4}
\xymatrix{
& PY\ar[d]^{\tau_Y}\\
\mathbf{F}\times U \ar@{^{(}->}[ur]^-{\tilde{s}}\ar@{^{(}->}
[r]^-{\textrm{open}} & \mathcal{P}_a.}
\end{equation*} 

Since $\mathbf{F}$ is rational by \eqref{fiber T} and $U$ is 
rational, it follows that 
\begin{equation}\label{Pa rational}
\mathcal{P}_a\ \ \textrm{is rational}.
\end{equation}

Next, from \eqref{piY}-\eqref{k*-action on Y},
\eqref{G-action on PY} and \eqref{tauY again} it follows that
the morphism $\pi:Y\to\mathcal{P}_a$ in \eqref{quotient 
morphism pi} is the quotient morphism of the following 
action of the group $\overline{G}:=G\times\mathbf{k}^{\times}$ on 
$Y$, where $G=\widetilde{G}/G'$ was defined in \eqref{group G}:
\begin{equation}\label{overlineG-action}
a_Y:\ Y\times\overline{G}\to Y, \ \ \ \ \ 
((b,\varphi_1,\varphi_2,[\sigma_1:\sigma_2]),([g_1:\lambda],\mu))
\mapsto(b,\frac{\mu\varphi_1}{\det(g_1)},
\frac{\mu\varphi_2}{\lambda^2},[g_1\circ\sigma_1:\lambda
\sigma_2]), \ \ \ \ \ \ \overline{G}=G\times\mathbf{k}^{\times}.
\end{equation}
Moreover,
\begin{equation}\label{pi is princ Gbdl}
\pi:Y\to\mathcal{P}_a=Y/\overline{G} \ \ \ \textrm{is a 
principal}\ \overline{G}\textrm{-bundle},
\end{equation}
and computations \eqref{g1 in GL}-\eqref{G-action on 
PY} show that the equivalence class $[y]$ of any point 
$y\in Y$ is the $\overline{G}$-orbit of $y$:
\begin{equation}
[y]=a_Y(\{y\}\times \overline{G})=\pi^{-1}(\pi(y)),\ \ \ \ 
\ \ y\in Y.
\end{equation}
In other words, $\mathcal{P}_a$ is the set of equivalence 
classes of points of $Y$:
\begin{equation}\label{Pa=}
\mathcal{P}_a=\{[y]~|~y\in Y\}.
\end{equation}

Remark that, by Corollary \ref{Okoneksimpletico}, the 
equality
$[y]=[\tilde{y}]$, i. e. the isomorphism of symplectic monads 
$A_y$ and $A_{\tilde{y}}$ in \eqref{isom of monads} is 
equivalent to the isomorphism of their cohomology rank 2 
bundles as symplectic bundles $(\mathcal{H}^0(A_y),\psi_y)$ 
and $(\mathcal{H}^0(A_{\tilde{y}}),\psi_{\tilde{y}})$, i. e., 
to the commutativity of the diagram
\begin{equation}\label{comm sympl coho}
\xymatrix{
\mathcal{H}^0(A_y)\ar[r]^-{\psi_y}_-{\simeq}\ar[d]^{\simeq}
_f & \mathcal{H}^0(A_y)^{\vee}\\
\mathcal{H}^0(A_{\tilde{y}})\ar[r]^-{\psi_{\tilde{y}}}
_-{\simeq} & \mathcal{H}^0(A_{\tilde{y}})^{\vee}\ar[u]
^{f^{\vee}}_{\simeq}.}
\end{equation}
Here $\psi_y$, respectively, $\psi_{\tilde{y}}$, is a 
symplectic isomorphism induced by the symplectic isomorphism 
of the monad $A_y$ with its dual $A_y^{\vee}$, respectively, 
of $A_{\tilde{y}}$ with $A_{\tilde{y}}^{\vee}$. Thus, 
denoting by $[\mathcal{H}^0(A_y),\psi_y]$ the isomorphism 
class of the pair $(\mathcal{H}^0(A_y),\psi_y)$, we have:
\begin{equation}\label{[y]=}
[y]=[\mathcal{H}^0(A_y),\psi_y]=[\mathcal{H}^0(A_y)].
\end{equation}
This together with \eqref{pi is princ Gbdl}-\eqref{Pa=} shows
that the modular morphism
$$
\Phi_Y:\ Y\to\calb(a^2+1),\ \ \ \ y\mapsto[\mathcal{H}
^0(A_y)]
$$
factors through an injective map $\Theta:\calp_a\to\calb
(a^2+1)$, i. e.
\begin{equation}\label{PhiY=}
\Phi_Y=\Theta\circ\pi.
\end{equation}
Since $Y$ is clearly smooth, the map $\Theta$ is actually a
morphism. This outcomes from the following well known general
result. (For the convenience of the reader we give its proof 
here.)
\begin{Lema}\label{lemma1}
\textit{Let $X,\ Y,\ Z$ be quasiprojective varieties with
$Y$ smooth, and let $a:X\to Y$ and $b:X\to Z$ be morphisms
such that $a$ is surjective and $b$ is constant on the 
fibers of $a$. Then there exists  a morphism $f:Y\to Z$ such 
that $b=f\circ a$.} 
\end{Lema}
\begin{proof}
Consider the morphism $g:X\to Y\times Z,\ x\mapsto(a(x),b(x)
)$, and let $Y\xleftarrow{a'}Y\times Z\xrightarrow{b'}Z$ be 
the projections onto factors  so that $a=a'\circ g$ and 
$b=b'\circ g$. Since $b$ is constant on the fibers of $p$, 
it follows that $\tilde{a}:=a'|_{g(X)}: g(X)\to Y$ is a 
bijection. Therefore, as $Y$ is smooth, $\tilde{a}$ is an 
isomorphism (see, e.g., \cite[Ch.2, Section 4.4, Thm. 
2.16]{S}). The desired morphism $f$ is now the composition 
$f=b'\circ\tilde{a}^{-1}$.
\end{proof}

Now Proposition \ref{Prop 15} together with \eqref{dim Pa},
\eqref{Pa rational}, \eqref{pi is princ Gbdl} and 
\eqref{PhiY=} yields 

\begin{Teo}\label{Thm 18}
There exists an injective morphism $\Theta:\calp_a
\hookrightarrow\calb(a^2+1)$ such that the modular morphism
$\Phi_Y:Y\to\calb(a^2+1)$ factorizes as
\begin{equation}
\Phi_Y:\ Y\xrightarrow{\pi}\calp_a\overset{\Theta}
{\hookrightarrow}\calb(a^2+1),
\end{equation}
where $\pi:Y\to\calp_a$ is a principal $\overline{G}$-bundle with 
the group $\overline{G}$ defined in \eqref{overlineG-action}. 
The variety $\overline{\mathcal{G}(a,1)}$ containing the rational 
variety $\mathcal{G}(a,1)_0=\Theta(\calp_a)$ as a dense subset is 
rational of dimension $4\binom{a+3}{3}-a-1$. 
\end{Teo}

We next obtain the following important formula.
\begin{Lema}\label{h1end}
For every $[\mathcal{E}]\in\mathcal{G}(a,1)_0$ with $a\ge2$, 
it holds 
$$ 
h^1(\lend(\cale))=4\cdot{a+3\choose 3}-a-1 +\varepsilon(a), 
$$
where $\varepsilon(a)=1$ when $a=3$, and $\varepsilon(a)=0$ 
when $a\ne3$.
\end{Lema}

\begin{proof}
Since $\mathcal{E}$ is a self dual rank 2 bundle, we have 
$\lend(\mathcal{E})\simeq S^2\mathcal{E}\oplus\Lambda^2
\mathcal{E} = S^2\mathcal{E}\oplus\op3$, thus 
$h^{1}(\lend(\mathcal{E}))=h^1(S^2\mathcal{E})$. We will 
compute the latter.

By the definition of $\calg(a,1)_0$ (see Proposition 
\ref{descrn of G(a,1)}.(ii), \eqref{def M(a)} and 
\eqref{H0=H0}), $\cale$ is the cohomology of a complex 
$M^{\bullet}$ with terms $M^{-1}=\op3(-a),\ M^0=E\simeq\op3
^{\oplus2}\oplus N,\ M^1=\op3(a)$. Proceed to the double 
complex $M^{\bullet}\otimes M^{\bullet}$ and to its total 
complex $T^{\bullet}$. The symmetric part of  $T^{\bullet}$ is 
the monad $0\to E(-a)\to S^2E\oplus\op3\to E(a)\to 0$, whose 
cohomology sheaf is isomorphic to $S^2\cale$. Therefore this 
monad can be broken into two short exact sequences
$$
0\to K\to S^2E\oplus\op3\to E(a)\to 0 \ \ \ \ \ 
{\rm and}\ \ \ \ \ 0 \to E(-a)\to K\to S^2\cale\to 0. 
$$
Since $h^0(E(-a))=h^0(S^2\cale)=0$, it follows that $h^0(K)=0$; in addition, $h^1(E(a))=h^2(S^2E\oplus\op3)=0$ (use 
Proposition \ref{hi S2E}) implies that $h^2(K)=0$. It then 
follows in view of the splitting $E\simeq\op3^{\oplus2}\oplus 
N$ that 
\begin{equation}\label{h1s2e}
h^1(S^2\mathcal{E})=h^1(K)+h^2(E(-a))=h^1(K)+\varepsilon(a),
\ \ \ \ \ \ \ \ \varepsilon(a):=h^2(N(-a)),
\end{equation}
since $h^1(E(-a))=0$ for $a\ge2$.

To complete our calculation, consider the exact sequence
$$ 
0\to H^0(S^2E\oplus\op3) \to H^0(E(a)) \to H^1(K) \to 
H^1(S^2E\oplus\op3) \to 0. 
$$
Since $h^0(S^2E\oplus\op3)= 4$ and $h^1(S^2E\oplus\op3)= 5$ by 
Proposition \ref{hi S2E}, we conclude that
$$ 
h^1(K) = h^0(E(a)) + 1 = h^0(N(a)) +  V_{2}\otimes 
h^0(\op3(a)) + 1, 
$$
which, together with the equality in equation (\ref{h1s2e}), 
yields the desired formula.
\end{proof}

It is interesting to observe that the right hand side of the 
formula in Lemma \ref{h1end} yields the value of $h^1(\lend(
\cale))$ expected by the deformation theory when $a=2$ and 
$a=3$, respectively 37 and 77; when $a\ge4$, one can check 
that $4\cdot {a+3 \choose 3} - a - 1 > 8(a^2+1)-3$.

Noting that, in view of Theorem \ref{Thm 18}, the dimension of 
$\overline{\mathcal{G}(a,1)}$ equals $h^1(\lend(E))$ for $a=2$ 
and $a\ge4$, as calculated in Lemma \ref{h1end}, and using 
Proposition \ref{descrn of G(a,1)}, we have therefore 
completed the proof of the first main result of this paper. 

\begin{Teo}\label{Thm 20}
For $a=2$ and $a\ge4$, the rank 2 bundles given as cohomology 
of monads of the form 
$$ 
0 \to \op3(-a)\oplus\op3(-1)\to V_{6}\otimes\op3 \to 
\op3(1)\oplus\op3(a) \to 0 $$ fill out a dense subset of a
rational irreducible component of $\mathcal{B}(a^2+1)
$ of dimension
$$ 
4\cdot{{a+3}\choose 3} - a - 1  . 
$$
\end{Teo} 

In particular, for the case $a=2$, we conclude that rank 2 
bundles given as cohomology of monads of the form 
(\ref{invariant121}) yield a dense subset of an irreducible 
component of $\mathcal{B}(5)$ with expected dimension 37.

\medskip

\section{Cohomology bundle $\cale$ of the monad of type \eqref{invariant131} and the related reflexive sheaf $\calf$}\label{F related to E}

\medskip

Consider the set
$$ 
\calh=\{ [\cale]\in\calb(5) ~|~ \cale~ 
\text{is cohomology of a monad of type}\ \eqref{invariant131}\}. 
$$
It is known that $\calh\ne\emptyset$ - see \cite[Table 5.3, $c_2=5$, 
Case (2).ii)]{hartshorne1991}. Note that the set $\calh$ is a 
constructible subset of $\calb(5)$, as well as $\calg(2,1)$ (see 
Remark after Proposition \ref{Prop 7}). The aim of this and the 
subsequent sections is to prove
\begin{Teo}\label{Cal H}
The set $\calh$ satisfies the condition
$\dim(\calh\smallsetminus(\calg(2,1)\cap\calh)) \le 36$. Its
closure in $\calb(5)$ does not constitute a component of
$\calb(5)$. 
\end{Teo}
In this section we will relate the vector bundle $[\cale]\in\calh
\setminus(\calg(2,1)\cap\calh)$ to a rank 2 reflexive sheaf $\calf$ 
with Chern classes $c_1(\calf)=0$, $c_2(\calf)=2$ and 
$c_3(\calf)=2k$, $0\le k\le 6$, 
which appears as a middle cohomology of a left-exact complex 
$K^{\bullet}$ (see \eqref{complex K}) induced by the monad of type 
\eqref{invariant131} defining $\cale$. This relation will be 
established in Proposition \ref{two steps}. We will then use it in 
Section \ref{properties of F} to prove Theorem \ref{Cal H}.

\vspace{2mm}
Let $[\cale]\in\calh\setminus(\calg(2,1)\cap\calh)$ be the cohomology 
bundle of the monad of the form \eqref{invariant131}:
\begin{equation}\label{Monad M}
\begin{split}
& M^{\bullet}:\ 0\to M^{-1}\xrightarrow{\alpha}M^0\xrightarrow{\beta}
M^{1}\to0,\\
& M^{-1}=\op3(-2)\oplus V_{2}\otimes\op3(-1),\ 
M^0=\op3(-1)\oplus V_{6}\otimes\op3\oplus\op3(1),\\ 
& M^{1}=V'_{2}\otimes\op3(1)\oplus\op3(2).
\end{split}
\end{equation}
Since the bundle $V_{2}\otimes\op3(-1)$ is a uniquely 
defined subbundle of the bundle $M^{-1}$ (respectively, 
$V'_{2}\otimes\op3(1)$ is a uniquely defined quotient bundle of 
$M^{1}$), we obtain a commutative diagram in which
$\alpha_0$ and $\beta_0$ are the induced morphisms:
\begin{equation}\label{i1}
\xymatrix{
V_2\otimes\op3(-1)\ar@{>->}[d]\ar@{>->}[dr]^-{\alpha_0} & & 
\op3(2)
\ar@{>->}[d] \\  
M^{-1}\ar@{>>}[d]\ar@{>->}[r]^-{\alpha} & 
M^0\ar@{>>}[r]^-{\beta}
\ar@{>>}[dr]^-{\beta_0} & 
M^{1}\ar@{>>}[d] \\
\op3(-2) & & V'_2\otimes\op3(1).} 
\end{equation}
Here the induced monad
\begin{equation}\label{i=2 new}
0\to V_{2}\otimes\op3(-1)\xrightarrow{\alpha_0}M^0\xrightarrow{\beta_0}
V'_{2}\otimes\op3(1)\to0
\end{equation}
has the rank 4 cohomology bundle 
\begin{equation}\label{E=coho bdl}
E=\frac{\ker\beta_0}{\mathrm{im}~\alpha_0}.
\end{equation}
Mimicking now the argument with diagram \eqref{i}, we obtain 
that there exist a subbunle morphism $\sigma:\op3(-2)\to E$
and an epimorphism $\tau:E\to\op3(2)$ which yield the monad
the the cohomology bundle $\cale$:
\begin{equation}\label{monad4}
0\to\op3(-2)\overset{\sigma}{\to}E\overset{\tau}{\to}\op3(2)
\to0,\ \ \ \ \ \ \cale=\ker\tau/\mathrm{im}~\sigma.
\end{equation}
Since there is a uniquely defined (up to a scalar multiple) quotient 
morphism $M^0\twoheadrightarrow \op3(-1)$, we have a well-defined 
morphism
\begin{equation}\label{tilde alpha}
\tilde{\alpha}:\ V_2\otimes\op3(-1)\overset{\alpha_0}
{\hookrightarrow}M^0\twoheadrightarrow \op3(-1)
\end{equation}
and, dually, a well-defined morphism
\begin{equation}\label{tilde beta}
\tilde{\beta}:\ \op3(1){\hookrightarrow}M^0\overset{\beta_0}
{\twoheadrightarrow}V'_2\otimes\op3(1).
\end{equation}
Assume that both $\tilde{\alpha}$ and $\tilde{\beta}$ 
are nonzero morphisms. Then a standard diagram chasing shows
that, in the monad \eqref{i=2 new}, one can split out a direct
summand $\op3(-1)$ from $V_{2}\otimes\op3(-1)$ and $M^0$, 
respectively, split out a direct summand $\op3(1)$ from $M^0$ 
and $V'_{2}\otimes\op3(1)$, without changing its cohomology 
bundle $E$. Thus, the monad \eqref{i=2 new}
reduces to a monad
\begin{equation}\label{i=2 new split}
0\to\op3(-1)\overset{\alpha'}{\to}V_{6}\otimes\op3
\overset{\beta'}{\to}\op3(1)\to0,\ \ \ \ \ \ 
E=\frac{\ker\beta'}{\mathrm{im}~\alpha'}.
\end{equation}
Now by the remark after 
Lemma \ref{directsum}, $E$ is a rank 4 instanton bundle, so 
that, by \eqref{monad4} and Lemma \ref{equivalence2}, $\cale$
is the cohomology bundle of the monad \eqref{Monad2} for $a=2
$ and $k=1$. This means that $\cale\in\calg(2,1)\cap\calh$,
contrary to the assumption on $\cale$.  

We thus may assume that either (a) $\tilde{\alpha}=0,\ 
\tilde{\beta}\ne0$, or (b) $\tilde{\alpha}=\tilde{\beta}=0$. 
(We omit the case $\tilde{\alpha}\ne0,\ \tilde{\beta}=0$, 
since it is completely similar to the case (a).)

\vspace{3mm}
(a) Case $\tilde{\alpha}=0,\ \tilde{\beta}\ne0$. We are going
to show that this case is impossible. 

First, note that, since $\tilde{\beta}\ne0$, we may as above 
split out a direct summand $\op3(1)$ from the middle term and 
the righthand term of the monad \eqref{i=2 new}, without 
changing its cohomology bundle $E$. Thus, this monad reduces 
to a monad
\begin{equation}\label{i=2 new split}
0\to V_2\otimes\op3(-1)\overset{\alpha'}{\to}\op3(-1)\oplus 
V_6\otimes\op3\overset{\beta'}{\to}\op3(1)\to0,\ \ \ \ \ \ 
E=\frac{\ker\beta'}{\mathrm{im}~\alpha'}.
\end{equation}
Next, the condition $\tilde{\alpha}=0$ means that the 
subbundle morphism $\alpha'$ in \eqref{i=2 new} factors 
through a subbundle morphism $\alpha''$ in the commutative
diagram
\begin{equation}\label{diag (a)}
\xymatrix{
& V_2\otimes\op3(-1)\ar@{>->}[r]^-{\alpha''}\ar@{=}[d]& 
V_6\otimes\op3 \ar@{>>}[r]\ar@{>->}[d] & F_4\ar@{>->}[d]^
-{\lambda} & \\
& V_2\otimes\op3(-1)\ar@{>->}[r]^-{\alpha'} &\op3(-1)\oplus 
V_6\otimes\op3\ar@{>>}[r]\ar@{>>}[d] &F_5\ar@{>>}[d]^-{\mu}& \\
& & \op3(-1)\ar@{=}[r] & \op3(-1), & } 
\end{equation}
where $F_4:=\mathrm{coker}~\alpha''$ and $F_5:=\mathrm{coker}
~\alpha'$ are vector bundles of rank 4 and 5, respectively.
From this diagram it follows immediately that $\op3(-1)$ 
splits out as a direct summand of $F_5$:
\begin{equation}\label{F5 direct sum}
F_5\cong \op3(-1)\oplus F_4.
\end{equation}
The monad \eqref{i=2 new split} and the diagram 
\eqref{diag (a)} yield a commutative diagram
\begin{equation}\label{diag (a1)}
\xymatrix{
& F_3\ar@{>->}[r]\ar@{>->}[d]& 
F_4\ar@{>>}[r]^-{\eta\circ\lambda}\ar@{>->}[d]^-{\lambda} & 
A\ar@{>->}[d]& \\
& E\ar@{>->}[r]^-{\nu}\ar@{>>}[d]^-{\mu\circ\nu} & F_5 
\ar@{>>}[r]^-{\eta}\ar@{>>}[d]^-{\mu} & 
\op3(1)\ar@{>>}[d]& \\
 & B\ar@{>->}[r] & 
\op3(-1)\ar@{>>}[r]^-{\bar{\eta}} & C, &} 
\end{equation}
where $F_3:=\ker(\eta\circ\lambda),\ A:=F_4/F_3,\ B:=E/F_3,\ 
C:=\op3(1)/A$. Here $A\ne0$, since otherwise $C\simeq\op3(1)
$, and then $\bar{\eta}$ is not surjective, contrary to
\eqref{diag (a1)}. Hence, $C$ is a torsion sheaf, and $A$, $B$
and $F_3$ are torsion free sheaves of rank 1, 1 and 4,
respectively. Therefore, the diagram \eqref{diag (a1)} 
implies that $c_1(F_4)-c_1(E)=2c_1(\op3(1))$. On the other 
hand, in view of \eqref{F5 direct sum} we have a well-defined
injective morphism $\rho:\ E\xrightarrow{\nu}F_5\xrightarrow{
pr_2}F_4$ such that, by the Snake Lemma, $Q:=\mathrm{coker}
\rho\cong A/B$ is a torsion sheaf. In addition, by the above 
equality, $c_1(Q)=2c_1(\op3(1))\ne0$, i. e., $Q\ne0$. However,
\eqref{diag (a1)} and the Snake Lemma yield a commutative 
diagram
\begin{equation}\label{diag (a2)}
\xymatrix{
& E\ar@{>->}[r]^-{\rho}\ar@{=}[d] & F_4\ar@{>>}[r] 
\ar@{>->}[d]^-{i} & Q\ar@{>->}[d]^-{\bar{i}} & \\
& E\ar@{>->}[r]^-{\nu} &\op3(-1)\oplus F_4\ar@{>>}[r]^-{\eta}
\ar@{>>}[d] & \op3(1)\ar@{>>}[d] & \\
& & \op3(-1)\ar@{=}[r] & \op3(1)/\bar{i}(Q), & } 
\end{equation}
where $i$ is the inclusion of the direct summand and $\bar{i}$
is the induced morphism. But the torsion sheaf $Q$ is not a 
subsheaf of $\op3(1)$, and we obtain a contradiction, as 
claimed. 

Summarizing the above arguments, we see that the bundle 
$[\cale]\in\calh$ is the cohomology $\calh^0(M^{\bullet})$ of a 
monad $M^{\bullet}$ of the form \eqref{invariant131} satisfying 
the condition (a): $(\tilde{\alpha},\tilde{\beta})\ne(0,0)$, then 
$M^{\bullet}$ is reducible to a monad of the form 
\eqref{invariant121}, i. e. $[\cale]\in\calh\cap\calg(2,1)$. 
Thus, denoting
\begin{equation}\label{def H0}
\calh_0:=\{[\cale]\in\calh\ |\ \cale=\calh^0(M^{\bullet
}), \textrm{where}\ M^{\bullet}\ \textrm{satisfies the condition (b):
}\ \tilde{\alpha}=\tilde{\beta}=0\},
\end{equation}
we obtain
\begin{equation}\label{in H0}
\calh\smallsetminus(\calh\cap\calg(2,1))\subset\calh_0.
\end{equation}
We thus proceed to the study of the case $\tilde{\alpha}=\tilde{\beta}
=0$.

\vspace{3mm}
(b) Case $\tilde{\alpha}=\tilde{\beta}=0$.

First, consider the commutative diagram
\begin{equation}\label{diag (b1)}
\xymatrix{
& \op3(1)\ar@{=}[r]\ar@{>->}[d]^-{j_1}& \op3(1)\ar@{>->}[d]^-{j_0} 
& & \\
& V_6\otimes\op3\oplus\op3(1)\ar@{>->}[r]^-{i_0}\ar@{>>}[d]
^-{h_1} & M^0\ar@{>>}[r]^-{g_0}\ar@{>>}[d]^-{h_0} & \op3(1)\ar@{=}[d] 
& \\
& V_6\otimes\op3\ar@{>->}[r]^-i & V_6\otimes\op3
\oplus\op3(-1)\ar@{>>}[r]^-{g} & \op3(1)&} 
\end{equation}
and the exact triples following from \eqref{i=2 new} and
\eqref{E=coho bdl}
\begin{equation}\label{tripl 1}
0\to V_2\otimes\op3(-1)\xrightarrow{\alpha_0}M^0\xrightarrow{c_0}C_0
\to0,
\end{equation}
\begin{equation}\label{tripl 2}
0\to E\xrightarrow{d_0}C_0\xrightarrow{e_0}V'_2\otimes\op3(1)\to0,\ \ 
\ \ \ \ C_0:=\mathrm{coker}~\alpha_0,\ \ \ \ \ \ \beta_0=e_0\circ c_0.
\end{equation}
The condition $\tilde{\alpha}=0$ implies that there exists a 
subbundle morphism $0\to V_2\otimes\op3(-1)\xrightarrow{
\alpha_1}V_6\otimes\op3\oplus\op3(1)$ such that
\begin{equation}\label{j0}
\alpha_0=i_0\circ\alpha_1.
\end{equation}
Setting $C:=\mathrm{coker}(h_0\circ\alpha_0),$ $C_1:=\mathrm{
coker}~\alpha_1$, $\alpha_2:=h_1\circ\alpha_1$, $C_2:=\mathrm{
coker}~\alpha_2$, we obtain from 
\eqref{diag (b1)}-\eqref{tripl 1} and \eqref{j0} an induced 
commutative diagram
\begin{equation}\label{diag (b2)}
\xymatrix{
\op3(1)\ar@{=}[r]\ar@{>->}[d]^-{\bar{j}_1}& \op3(1)\ar@{>->}[d]
^-{\bar{j}_0} & \\
C_1\ar@{>->}[r]^-{\bar{i}_0}\ar@{>>}[d]^-{\bar{h}_1} & C_0\ar@{>>}[r]
^-{\bar{g}_0}\ar@{>>}[d]^-{\bar{h}_0} & \op3(-1)\ar@{=}[d]\\
C_2\ar@{>->}[r]^-{\bar{i}} & C\ar@{>>}[r]^-{\bar{g}} & \op3(-1)} 
\end{equation}
and an exact triple
\begin{equation}\label{tripl 3}
0\to V_2\otimes\op3(-1)\xrightarrow{\alpha_2}V_6\otimes\op3
\xrightarrow{c_2}C_2\to0.
\end{equation}
From the condition $\tilde{\beta}=0$ and diagram chasing it 
follows that there exists an injective morphism $j:\ \op3
(1)\to E$ such that $\bar{j}_0=d_0\circ j$. From this relation and 
\eqref{tripl 2}, \eqref{diag (b2)} and \eqref{tripl 3} by diagram 
chasing we obtain the folowing data:\\ 
1) an exact triple
\begin{equation}\label{triple E3}
0\to\op3(1)\xrightarrow{j}E\xrightarrow{h}E_3\to0,\ \ \ \ \ \ 
E_3:=\mathrm{coker}j,
\end{equation}
2) a commutative diagram
\begin{equation}\label{diag (b3)}
\xymatrix{
& \calf\ar@{>->}[r]\ar@{>->}[d]& E_3\ar@{>>}[r]^-{\varepsilon}
\ar@{>->}[d]^-{d} & \op3(-1)\ar@{=}[d] & \\
& C_2\ar@{>->}[r]^-{\bar{i}}\ar@{>>}[d]^-{\bar{e}} & C 
\ar@{>>}[r]^-{\bar{g}}\ar@{>>}[d]^-{e} & \op3(-1) & \\
& V'_2\otimes\op3(1)\ar@{=}[r] & V'_2\otimes\op3(1), & &} 
\end{equation}
where $d$ and $e$ are the induced morphisms, $\varepsilon:=
d\circ\bar{g}$, $\bar{e}:=e\circ\bar{i}$,\\
3) a sheaf
\begin{equation}\label{E2}
\calf:=\ker~\varepsilon,
\end{equation}
and a left-exact complex
\begin{equation}\label{complex K}
\begin{split}
& K^{\bullet}:\ 0\to K^{-1}\xrightarrow{\alpha_2}K^0\xrightarrow{
\beta_2}K^{1}\to0,\ \ \ \beta_2:=\bar{e}\circ c_2,\\
& K^{-1}=V_2\otimes\op3(-1),\ \ K^0=V_6\otimes\op3,\ \ 
K^{1}=V'_2\otimes\op3(1),
\end{split}
\end{equation}
such that
\begin{equation}\label{H^i(K)}
\calh^0(K^{\bullet})=\calf,\ \ \ \ \ \ \ \calh^1(K^{\bullet})=
\mathrm{coker}~\varepsilon.
\end{equation}

From \eqref{monad4}, \eqref{complex K} and the vanishing of
$\Hom(\op3(1),\op3(-2))$ follows the commutative diagram
\begin{equation}\label{diag c1}
\xymatrix{
&\op3(1)\ar@{>->}[d]^j\ar@{=}[r]&\op3(1)\ar@{>->}[d]^-{j'} &\\
\op3(-2) \ar@{>->}[r]^-{\sigma} \ar@{=}[d]& E\ar@{>>}[r]  
\ar@{>>}[d]^-h & \mathrm{coker}~\sigma\ar@{>>}[d] & \\
\op3(-2)\ar@{>->}[r]^-{h\circ\sigma} & E_3 \ar@{>>}[r] & \call &\\} 
\end{equation}
where $\call:=\mathrm{coker}(h\circ\sigma)$ and $j'$ is an induced 
morphism which is nonzero, hence injective, since 
$\mathrm{coker}~\sigma$ is locally free by the exact triple
$0\to\cale\to\mathrm{coker}~\sigma\to\op3(2)\to0$ following 
from \eqref{monad4}. Since $\cale$ is stable by assumption, 
so that $h^0(\cale(-1))=0$ (see \cite{okonek}), the last 
triple and \eqref{diag c1} yield a commutative diagram
\begin{equation}\label{diag c2}
\xymatrix{
& & \op3(1)\ar@{>->}[d]^{j'} \ar@{=}[r] & 
\op3(1)\ar@{>->}[d]^-{\bar{\beta}} & \\
& \cale\ar@{>->}[r]\ar@{=}[d] & \mathrm{coker}~\sigma
\ar@{>>}[r] \ar@{>>}[d] & \op3(2)\ar@{>>}[d] & \\
&\cale\ar@{>->}[r]& \call\ar@{>>}[r]&\calo_{\mathbb{P}^2}(2),&} 
\end{equation}
where $\p2=\mathrm{Supp}(\coker\bar{\beta})$ is a projective 
plane in $\p3$. Note that, in this diagram, $\bar{\beta}$ is the
composition $\bar{\beta}:\op3(1)\xrightarrow{\mathrm{can}}M^0
\xrightarrow{\beta}M^1$, and $\mathrm{im}~\bar{\beta}\hookrightarrow
\op3(2)$ since $\tilde{\beta}=0$. Thus, $\p2$ is uniquely defined by
the morphism $\beta$ in the monad $M^{\bullet}$. In a similar way,
since $\tilde{\alpha}=0$, the morphism $\alpha$ in $M^{\bullet}$ 
uniquely defines a morphism $\bar{\alpha}:\op3(-2)\to\op3(-1)$, hence
a projective plane $\mathbb{P}^2_0=\mathrm{Supp}(\coker\bar{\alpha})$. 
For these two planes we will use the notation
\begin{equation}\label{notation for P2}
\p2=\p2(M^{\bullet},\beta),\ \mathbb{P}^2_0=\p2(M^{\bullet},\alpha).
\end{equation}

Consider the lower horizontal triple in \eqref{diag c2}:
\begin{equation}\label{cal Q}
0\to\cale\xrightarrow{\theta}\call\xrightarrow{\gamma}\calo_{
\mathbb{P}^2}(2)\to0,\ \ \ \ \ \ \ \ \p2=\p2(M^{\bullet},\beta).
\end{equation}
\noindent
\begin{Lema}\label{sheaf Q}
The sheaf $\call$ in \eqref{cal Q} is a stable reflexive rank 2 
sheaf on $\p3$, $[\call]\in\calr(1,4,6)$.
\end{Lema}
\begin{proof}
First, show that the triple \eqref{cal Q} doesn't split. Indeed,
otherwise, the lower horizontal triple in \eqref{diag c1} extends to 
a commutative push-out diagram
\begin{equation}\label{diag c3}
\xymatrix{
\op3(-2)\ar@{>->}[r]\ar@{=}[d] & E_3\ar@{>>}[r] & \cale\oplus 
\calo_{\p2}(2)\\
\op3(-2)\ar@{>->}[r] & \op3(-2)\oplus\calo_{\p2}(2) 
\ar@{>>}[r] \ar[u] & \calo_{\mathbb{P}^2}(2),\ar[u]} 
\end{equation}
where the lower triple splits since $\Ext^1(\calo_{\mathbb{P}^2}(2),
\op3(-2))=0$. This yields a nonzero morphism $\delta:\ 
\calo_{\mathbb{P}^2}(2)\to E_3$ which being composed with the 
morphism $\varepsilon$ in \eqref{E2} is the zero morphism 
$\op2(2)\to\op3(-1)$. Hence, in \eqref{E2}, $\delta$ factors through 
a nonzero morphism
\begin{equation}\label{delta'}
\delta':\ \calo_{\mathbb{P}^2}(2)\to \calf.
\end{equation}
On the other hand, \eqref{complex K} and \eqref{H^i(K)} yield an exact
triple $0\to\op3(-2)\to\ker~\beta_2\to \calf\to0$ which, together with
\eqref{delta'}, extends to a push-out diagram similar to \eqref{diag 
c3}:
\begin{equation}\label{diag c4}
\xymatrix{
\op3(-2)\ar@{>->}[r]\ar@{=}[d]&\ker~\beta_2\ar@{>>}[r]& \calf\\
\op3(-2)\ar@{>->}[r] & \op3(-2)\oplus\calo_{\mathbb{P}^2}(2) \ar@{>>}
[r]\ar[u]^-{\delta''}& \calo_{\mathbb{P}^2}(2),\ar[u]^-{\delta'}} 
\end{equation}
where $\delta''|_{\calo_{\mathbb{P}^2}(2)}$ is nonzero. However, this
is impossible since $\ker~\beta_2$ by definition is torsion free as a 
subsheaf of a locally free sheaf $V_6\otimes\op3$. 

Next, since $\cale\cong\cale^{\vee}$ is locally free and $\cale 
xt^1(\calo_{\mathbb{P}^2}(2),\op3)=\calo_{\mathbb{P}^2}(-1)$, then, 
applying the functor $\cale xt^{\bullet}(-,\op3)$ to the triple 
\eqref{cal Q} we have an exact sequence
\begin{equation}\label{dual Q}
0\to\call^{\vee}\xrightarrow{\theta^{\vee}}\cale\to\calo_{\mathbb{P}^2
}(-1)\xrightarrow{\varphi}{}^1\call\to0,\ \ \ \ \ \ \ \ \ \ \ \ \ 
{}^1\call:=\cale xt^1(\call,\op3).
\end{equation}
Let $d=\dim({}^1\call)$. Consider the three possible cases: (a) 
$d=2$, (b) $d=1$, and (c) $d=0$. We will show that the cases (a) and 
(b) lead to a contradition.\\
(a) $d=2$. In this case $\dim\mathrm{Sing}~\call=2$, i. e. the torsion 
subsheaf $\mathcal{T}ors(\call)$ of $\call$ has dimension 2. This 
necessarily implies that the composition $\mathcal{T}ors(\call)
\hookrightarrow\call\overset{\gamma}{\twoheadrightarrow}\calo_{
\mathbb{P}^2}(2)$ is an isomorphism giving the splitting of the 
triple \eqref{cal Q}, contrary to the above.\\
(b) $d=1$. In this case ${}^1\call=\calo_Z(-1)$ for $Z$ a subscheme 
of $\mathbb{P}^2$, of dimension $\dim Z=1$. Hence $\ker\varphi
\hookrightarrow\calo_{\mathbb{P}^2}(-1-k)$ for some $k\ge2$. By 
\eqref{dual Q} the sheaf $\ker\varphi$ is the quotient of $\cale$, 
it follows that $h^0(\cale_{\p2}(-1-k))\ne0,\ k\ge1$, and 
so  $h^0(\cale_{\p2}(-2))\ne0$. On the other hand, 
since $\cale $ is the cohomology of (\ref{Monad M}), by 
\cite[Table 5.3, case 5(2.ii)]{hartshorne1991} it  has the 
spectrum $\mathrm{Sp}(\cale)=(-1,0,0,0,1)$, and then it follows 
that 
\begin{equation}\label{h1=0,1}
h^1(\cale(-3))=0,\ \ \ h^1(\cale(-2))=1.
\end{equation}
Thus, by the first equality in \eqref{h1=0,1}, the inequality 
$h^0(\cale_{\p2}(-2))\ne0$ contradicts to the cohomology 
sequence of the exact triple 
\begin{equation}\label{restr P2}
0\to\cale(-3)\to\cale(-2)\to\cale_{\p2}(-2)\to0
\end{equation}
as $h^0(\cale(-2))=0$ by the stability of $E$. Note also that we 
have proved here the equality
\begin{equation}\label{h0(E(-2)|P2) vanishes)}
H^0(\cale_{\p2}(-2))=0,\ \ \ \ \ \ \ \ \ \ \ \forall\ 
\p2\subset\p3.
\end{equation}
(c) $d=1$. In this case ${}^1\call=\calo_Z(-1)$ for $Z$ a 
subscheme of $\mathbb{P}^2$ of dimension $\dim Z=0$, and the 
sequence $0\to\call^{\vee}\xrightarrow{\theta^{\vee}}\cale\to
\cali_{Z,\p2}(-1)\to0$, and $Z=(s)_0$, $0\ne s\in H^0(\cale(-1)|
_{\p2})$. Since $\dim Z=0$, it follows that $\lext^1(\cali_{Z,\p2}
(-1),\op3)=\lext^1(\op2(-1),\op3)=\op2(2)$. Thus, applying the 
functor $\lext^{\bullet}(-,\op3)$ to the last triple, since  
$\cale\cong\cale^{\vee}$ we obtain an exact triple
$0\to\cale\xrightarrow{\theta^{\vee\vee}}\call^{\vee\vee}
\xrightarrow{\gamma}\calo_{\mathbb{P}^2}(2)\to0$. Comparing this 
triple with \eqref{cal Q} and taking into account that, by 
construction, the composition $\cale\xrightarrow{\theta}\call
\xrightarrow{\mathrm{can}}\call^{\vee\vee}$ coincides with
$\theta^{\vee\vee}$ we obtain that $\call\xrightarrow{\mathrm{can
}}\call^{\vee\vee}$ is an isomorphism, i. e. $\call$ is reflexive. 

Next, as $c_t(\cale)=1+5t^2$, formulas for Chern classes of 
$\call$ follow from \eqref{cal Q}. In particular, $\call^{\vee}
\cong\call(-1)$ has $c_1(\call(-1))=-1$, and since $h^0(\cale)=0
$, it follows that 
\begin{equation}\label{h0(calg(-1))=0}
h^0(\call(-1))=0.
\end{equation} 
Thus, $\call$ is stable by \cite[Lemma 3.1]
{Hartreflexive}. Lemma is proved.
\end{proof}

We now proceed to the more close study of the sheaf $\calf$. Consider 
the upper horizontal triple of the diagram \eqref{diag (b3)} which 
extends to an exact sequence:
\begin{equation}\label{exact E2}
0\to \calf\to 
E_3\xrightarrow{\varepsilon}\op3(-1)\to\calo_{\bar{Y}}(-1)\to0,
\ \ \ \ \ \ \ \ \ \ \bar{Y}\subset\p3.
\end{equation}
\begin{Lema}\label{sheaf E2}
The sheaf $\calf$ defined in \eqref{E2} is a reflexive rank 2 sheaf on 
$\p3$ fitting in an exact triple
\begin{equation}\label{E2 and Q}
0\to \calf\xrightarrow{\zeta}\call\to\cali_{\bar{Y},\mathbb{P}^2_0}(-1)
\to0,
\end{equation}
and in its dual
\begin{equation}\label{E2 and Q dual}
0\to \call(-1)\to\calf\xrightarrow{\rho}\cali_{\bar{Z},\mathbb{P}^2_0}
(2)\to0,
\end{equation}
where $\mathbb{P}^2_0=\p2(M^{\bullet},\alpha)$, $\bar{Y},\ \bar{Z}
\subset\mathbb{P}^2_0$, $\dim\bar{Y}\le0$, $\dim\bar{Z}\le0$, and 
\begin{equation}\label{lY+lZ=6}
   \ell(\bar{Y})+\ell(\bar{Z})=6. 
\end{equation}
Chern classes of $\calf$ are $c_1(\calf)=0$, $c_2(\calf)=2$, 
$0\le c_3(\calf)=2\ell(\bar{Y})\le12$.  
\end{Lema}
\begin{proof}
We first show that $\rk \calf=2$. Indeed, if $\varepsilon$
in \eqref{exact E2} is the zero morphism, then the diagram 
\eqref{diag (b3)} and the Snake Lemma yield an epimorphism 
$V'_2\otimes\op3(1)\twoheadrightarrow\op3(-1)$ which is impossible.
Hence, $\varepsilon\ne0$ and \eqref{E2} implies that $\rk \calf=2$ and,
moreover, that $\bar{Y}\subsetneqq\p3$, i. e., $\bar{Y}$ is a proper 
subscheme of $\p3$. Note also that, by \eqref{i=2 new} and 
\eqref{E=coho bdl}, $c_1(E)=0$, hence $c_1(E_3)=-1$ in view of 
\eqref{triple E3}. Thus, \eqref{exact E2} implies that 
$c_1(\calf)=c_1(\calo_{\bar{Y}}(-1))\ge0$.

Next, consider the lower exact triple in \eqref{diag c1}:
\begin{equation}\label{E2,Q}
0\to\op3(-2)\xrightarrow{h\circ\sigma}E_3\to\call\to0.
\end{equation}
If the composition $f:=\varepsilon\circ h\circ\sigma$ is zero, then
\eqref{exact E2} and \eqref{E2,Q} imply that there exist injective 
morphisms $\op3(-2)\overset{f_1}\rightarrowtail \calf$ and 
$\mathrm{coker}(f_1)\overset{f_2}\rightarrowtail\call$. Since $\rk 
\calf=2$, $c_1(\calf)\ge0$ and $\call$ is reflexive by Lemma 
\ref{sheaf Q}, it follows that $\mathrm{coker}(f_1)$ is a rank 1 
torsion free sheaf with $c_1(\mathrm{coker}(f_1))\ge2$. Thus, the
injectivity of $f_2$ shows that $h^0(\call(-2))\ne0$, contrary to the
stability of $\call$ (see Lemma \ref{sheaf Q}). It follows that 
$f\ne0$, so that \eqref{exact E2} and \eqref{E2,Q} extend to a 
comutative diagram
\begin{equation}\label{diag c5}
\xymatrix{
& \op3(-2)\ar@{=}[r]\ar@{>->}[d]^-{h\circ\sigma}& \op3(-2)\ar@{>->}[d]
^-{f} & & \\
\calf\ar@{=}[d]\ar@{>->}[r]&E_3\ar[r]^-{\varepsilon}\ar@{>>}[d]&
\op3(-1)\ar@{>>}[r]^-{\delta}\ar@{>>}[d] & \calo_{\bar{Y}}(-1)
\ar@{=}[d]\\
\calf\ar@{>->}[r] & \call\ar[r]^-{\bar{\varepsilon}} & 
\calo_{\mathbb{P}^2_0}(-1)\ar@{>>}[r]^-{\bar{\delta}}
& \calo_{\bar{Y}}(-1),}
\end{equation}
where $\mathbb{P}^2_0$ is some projective plane in $\p3$. If 
$\bar{\delta}$ is an isomorphism, then $\mathrm{coker}(\varepsilon)
\simeq\calo_{\mathbb{P}^2_0}(-1)$, so that the diagram 
\eqref{diag (b3)} and the Snake Lemma yield an epimorphism 
$V'_2\otimes\op3(1)\twoheadrightarrow\calo_{\mathbb{P}^2_0}(-1)$
which is impossible. 
Hence, $\bar{Y}\subsetneqq\mathbb{P}^2_0$, i. e., $\bar{Y}$ is a 
proper subscheme of $\mathbb{P}^2_0$, $\dim\bar{Y}\le1$, and 
\eqref{diag c5} yields an exact triple \eqref{E2 and Q}.

Show that the case $\dim\bar{Y}=1$ is impossible. Indeed, in this case 
$\bar{Y}$ contains a divisor $D\subset\mathbb{P}^2_0$ of degree 
$k\ge1$ as a subscheme, and this yields an epimorphim $\calo_{\bar{Y}}
(-1)\overset{b}{\twoheadrightarrow}\calo_D(-1)$. On the other hand, 
the middle horizontal exact sequence in \eqref{diag c5}, together with 
diagram \eqref{diag (b3)} and the Snake Lemma, yield an epimorphism 
$V'_2\otimes\op3(1)\twoheadrightarrow\calo_{\bar{Y}}(-1)$. This 
epimorphism composed with the above epimorphism $b$ gives an 
epimorphism $V'_2\otimes\op3(1)\twoheadrightarrow\calo_D(-1)$ which is 
impossible, since $h^0(\calo_D(-2))=0$, as follows from the 
cohomology of the exact triple $0\to\calo_{\mathbb{P}^2_0}(-2-k)\to
\calo_{\mathbb{P}^2_0}(-k)\to\calo_D(-2)\to0$. 

Hence, $\dim\bar{Y}\le0$ and therefore, denoting ${}^i\cali:=\cale 
xt^i(\cali_{\bar{Y},\mathbb{P}^2_0}(-1),\op3),\ i\ge1$, we obtain 
${}^1\cali=\op2(2)$, $\dim{}^2\cali\le0$, ${}^3\cali=0$. Besides, set 
${}^i\calf:=\cale xt^i(\calf,\op3)$, ${}^i\call:=\cale xt^i(\call,
\op3)$, $i\ge1$, and remark that, for the reflexive sheaf $\call$, 
$\dim{}^1\call=0$, ${}^i\call=0,\ i=2,3$ (see \cite[proof of Thm. 
2.5]{Hartreflexive}). 
Now, applying to \eqref{E2 and Q} the functor $\cale xt^{\bullet}
(-,\op3)$ and using the above relations we obtain the equalities
${}^i\calf=0,\ i=2,3$, and an exact sequence $0\to\call^{\vee}
\xrightarrow{\zeta^{\vee}}\calf^{\vee}\to\op2(2)\xrightarrow{\mu}{}^1
\call\to{}^1\calf\to{}^2\cali\to0$, wherefrom $\dim{}^1\calf\le0$ and  
$\ker\mu\simeq\cali_{Z,\mathbb{P}^2_0}(2)$ for some subscheme $Z$ of 
$\mathbb{P}^2_0$, of dimension $\dim Z\le0$. We thus obtain an exact 
triple $0\to\call^{\vee}\xrightarrow{\zeta^{\vee}}\calf^{\vee}\to
\cali_{Z,\mathbb{P}^2_0}(2)\to0$ and the relation $\cale xt^1(\cali_{Y,
\mathbb{P}^2_0}(2),\op3)=\op2(-1)$. 
Next, applying to the last triple the 
functor $\cale xt^{\bullet}(-,\op3)$ yields an exact sequence
$0\to \calf^{\vee\vee}\xrightarrow{\zeta^{\vee\vee}}\call^{\vee\vee}\to
\op2(-1)\xrightarrow{\nu}\cale xt^1(\calf^{\vee},\op3)$. By \cite[Cor. 
1.2]{Hartreflexive} $\calf^{\vee}$ is a reflexive rank 2 sheaf, hence
$\dim\cale xt^1(\calf^{\vee},\op3)\le0$ by \cite[Rem. 2.7.1]
{Hartreflexive}, and therefore $\ker\nu\simeq\cali_{W,\mathbb{P}^2_0}
(-1)$ for some subscheme $W$ of $\mathbb{P}^2_0$, of dimension $\dim 
W\le0$. Thus the last sequence leads to an exact triple $0\to 
\calf^{\vee\vee}\xrightarrow{\zeta^{\vee\vee}}\call^{\vee\vee}\to\cali
_{W,\mathbb{P}^2_0}(-1)\to0$ which together with \eqref{E2 and Q} fits 
in a commutative diagram
$$
\xymatrix{
\calf^{\vee\vee}\ar@{>->}[r]^-{\zeta^{\vee\vee}}&\call^{\vee\vee}\ar
@{>>}[r]\ar@{=}[d]^-{\mathrm{can}}&\cali_{W,\mathbb{P}^2_0}(-1)\\
\calf\ar@{>->}[r]^{\zeta}\ar@{>->}[u]^-{\mathrm{can}} & 
\call\ar@{>>}[r] & \cali_{\bar{Y},\mathbb{P}^2_0}(-1).\ar[u]^-{c}}
$$
Besides, the above stated relations ${}^i\calf=0,\ i=2,3$, $\dim{}^1
\calf\le0$ show that the sheaf $\calf$ is locally free outside the set 
of dimension $\le0$, and this shows that the sheaf $\kappa=\mathrm{
coker}(\calf\xrightarrow{\mathrm{can}}\calf^{\vee\vee})$ has dimension
$\le0$ and by the Snake Lemma $\kappa$ is a subsheaf of $\ker c$. 
However, the sheaf has no subsheaves of dimension 0. Hence, $\kappa
=0$ and $\calf\xrightarrow{\mathrm{can}}\calf^{\vee\vee}$ is an 
isomorphism, i. e. $\calf$ is reflexive. A standard computation with the
triple \eqref{E2 and Q} yields the values of Chern classes of $\calf$,
The triple \eqref{E2 and Q dual} and the equality \eqref{lY+lZ=6} are 
obtained by applying to \eqref{E2 and Q} the functor $\cale 
xt^{\bullet}(-,\op3)$ and using formulas for Chern classes of $\calf$ 
and $\call$. The inequality $0\le c_3(\calf)\le12$ follows from 
\eqref{lY+lZ=6}.
\end{proof}

\begin{Lema}\label{P20= P2}
The projective planes $\p2$ and $\mathbb{P}^2_0$ defined in 
\eqref{notation for P2} coincide.
\end{Lema}
\begin{proof}
The middle horizontal triple $0\to\cale\to\mathrm{coker}~
\sigma\to\op3(2)\to0$ in \eqref{diag c2} as an extension is defined
by a nonzero element in $\Ext^1(\cale,\op2(2))\simeq H^1(\cale(-2))$.
Since $h^1(\cale(-2))=1$ by \eqref{h1=0,1}, it follows that the 
sheaf $\mathrm{coker}~\sigma$ is defined by $\cale$ uniquely up to an
isomorphism. Since $h^0(\call(-1))=0$ as $\call$ is stable by Lemma
\ref{sheaf Q}, the twisted by $\op3(-1)$ middle vertical triple $0\to
\op3\to\mathrm{coker}\sigma(-1)\to\call(-1)\to0$ in \eqref{diag c2}
shows that $h^0(\mathrm{coker}\sigma(-1))=1$. Hence, $\call=\call(M^{
\bullet})$ 
is uniquely up to an isomorphism defined by $\mathrm{coker}\sigma$ 
(and therefore by $\cale$) as 
\begin{equation}\label{G by E}
\call(M^{\bullet})=(\mathrm{coker}\sigma(-1)/\op3)(1).\ \ \ \ \ \ 
\end{equation}
Then the lower horizontal triple in \eqref{diag c2} shows that the 
plane $\p2=\p2(M^{\bullet},\beta)$ is determined uniquely by $\cale$ as
\begin{equation}\label{P20 by E}
\p2(M^{\bullet},\beta)=\mathrm{Supp}(\call(M^{\bullet})/\cale).
\end{equation}
Next, it follows from \eqref{notation for P2} that
\begin{equation}\label{P20=...}
\p2(M^{\bullet},\alpha)=\p2(M^{\bullet\vee},\beta^{\vee}),
\ \ \ \text{resp.,}\ \ \ \p2(M^{\bullet},\beta)=\p2(M^{\bullet
\vee},\alpha^{\vee}),
\end{equation}
where $M^{\bullet\vee}:
\ 0\to (M^1)^{\vee}\xrightarrow{\beta^{\vee}}(M^0)^{\vee}\xrightarrow{
\alpha^{\vee}}(M^{-1})^{\vee}\to0$ is the monad dual to $M^{\bullet}$. 
The monad $M^{\bullet\vee}$ defines the monad dual to \eqref{monad4}: 
$0\to\op3(-2)\overset{\tau^{\vee}}{\to}E\overset{\sigma^{\vee}}{\to}
\op3(2)\to0$ with $\cale^{\vee}=\ker(\sigma^{\vee})/\mathrm{im}(\tau^{
\vee})$, and the argument dual to the above yields the formulas dual to
\eqref{G by E} and \eqref{P20 by E}:
$\p2(M^{\bullet\vee},\beta^{\vee})=\mathrm{Supp}(\call(M^{\bullet\vee})
/\cale^{\vee})$,\ $\call(M^{\bullet\vee})=(\mathrm{coker}(\tau^{\vee})
(-1)/\op3)(1)$.
Since $\cale^{\vee}\simeq\cale$ these formulas mean in view of 
\eqref{P20=...} that the plane $\mathbb{P}^2_0=\p2(M^{\bullet},\alpha)
$ is uniquely defined by $\cale$ via the same construction as above, 
hence it coincides with $\p2=\p2(M^{\bullet},\beta)$.
\end{proof}

\vspace{2mm}
Let ${\calf}\in\calr(0,2,2k)$ be the reflexive sheaf defined in 
\eqref{E2}, where $0\le k\le6$ by Lemma \ref{sheaf E2}, i. e., 
\begin{equation}\label{R(0,2)}
[\calf]\in\underset{0\le k\le6}{\bigsqcup}\calr_k,\ \ \ \ \ \ \ \ \ \ 
\ \calr_k:=\calr(0,2,2k).
\end{equation}
Formulas \eqref{def H0}, \eqref{in H0} and Lemmas \ref{sheaf Q},
\ref{sheaf E2} and \ref{P20= P2} yield
\begin{Prop}\label{two steps}
There is an inclusion
\begin{equation}\label{cup Hk}
\calh\smallsetminus(\calh\cap\calg(2,1))\subset\underset{0\le 
k\le6}{\bigsqcup}\calh_k,\ \ \ \ \ \ \ \ \ \ where
\end{equation} 
\begin{equation}\label{def Hk}
\begin{split}
& \calh_k=\{[\cale]\in\calb(5) \ |\ 
\cale\ \mathrm{is\ obtained\ from}\ \calf, \mathrm{where}\ 
[\calf]\in\calr_k,\\
& \mathrm{by\ the\ two\ subsequent\ elementary\ transformations\  
\eqref{two transfns}\ below}\},
\end{split}
\end{equation} 
\begin{equation}\label{two transfns}
\begin{split}
& 0\to \call(-1)\to\calf\xrightarrow{\rho}\cali_{\bar{Z},\p2}
(2)\to0,\ \ \ \ \ \ \ (step\ 1)\\
& 0\to\cale\to\call\xrightarrow{\gamma}\op2(2)\to0,
\ \ \ \ \ \ \ \ \ \ \ \ \ \ \ (step\ 2)
\end{split}
\end{equation}
where $\p2$ is some plane in $\p3$, $\bar{Z}\subset\p2$, 
$\dim\bar{Z}\le0$, $\ell(\bar{Z})=6-k$, and $\call$ is a stable 
reflexive sheaf from $\calr(1,4,6)$.  
\end{Prop}

\medskip

\section{Geometric properties of sheaves $\calf$ and moduli of 
cohomology bundles $\cale$ of monads (\ref{invariant131})}
\label{properties of F}

\medskip

In this section we explore in detail the geometry of the reflexive 
sheaves $\calf$ described in Lemma \ref{sheaf E2}. The main result of
this study will be the upper estimates for the dimensions of the moduli
space of sheaves $\calf$ and sheaves $\call$ obtained from $\calf$ by
the elementary transformation \eqref{E2 and Q dual}. These estimates 
are obtained in Propositions \ref{Prop F unstable} and \ref{Prop F 
stable}  below. This will eventually lead to the proof of Theorem 
\ref{Cal H}.

Denote
$$
\calr_k^u:=\{[\calf]\in\calr_k\ |\ \calf\ \textrm{is unstable}\},
\ \ \ \ \ \ 
\calr_k^s:=\{[\calf]\in\calr_k\ |\ \calf\ \textrm{is stable}\},
$$
$$
\calh_k^u:=\{[\cale]\in\calh_k\ |\ \cale\ \textrm{is obtained from}\ 
\calf\ \textrm{in}\ \eqref{def Hk}, \textrm{where}\ [\calf]\in\calr_k^u
\},
$$
$$
\calh_k^s:=\{[\cale]\in\calh_k\ |\ \cale\ \textrm{is obtained from}\ 
\calf\ \textrm{in}\ \eqref{def Hk}, \textrm{where}\ [\calf]\in\calr_k^s
\},
$$
where $0\le k\le6$. Thus, $\calr_k=\calr_k^u\sqcup\calr_k^s$ and 
\eqref{in H0} and \eqref{cup Hk} yield:
\begin{equation}\label{cup Hk us}
\calh\smallsetminus(\calh\cap\calg(2,1))\subset\underset{0\le 
k\le6}{\bigsqcup}(\calh_k^u\sqcup\calh_k^s).
\end{equation} 
The estimate for the dimension of $\calh\smallsetminus(\calh\cap\calg(
2,1))$ will eventually follow from the computations of dimensions of
$\calh_k^u$ and $\calh_k^s$ which we will give below. For this, we 
start with an explicit description of the spaces $\calr_k^u$ and 
$\calr_k^s$.
\begin{Prop}\label{Prop F unstable}
(i) $\calr_k^u\ne\emptyset$ only for $0\le k\le3$, and any sheaf 
$\calf$ from $\calr_k^u$ fits in an exact triple
\begin{equation}\label{tripl for F}
0\to\op3\xrightarrow{s}\calf\xrightarrow{u}\cali_{C,\p3}\to0,
\end{equation}
where $C=\mathrm{Sing}(\calf/\op3)$ is a l.c.i. curve of degree 2 in 
$\p3$, $\chi(\calo_C)=4-\frac{1}{2}c_3(\calf)=4-k$.\\
(ii) If $C$ is reduced, then either $c_3(\calf)=4$ and $C$ is a 
disjoint union $l_1\sqcup l_2$ of two projective lines in $\p3$, or
$c_3(\calf)=6$, then $C$ is a plane conic in $\p3$. \\
(iii) If $C$ is nonreduced then $C$ is the scheme structure 
of multiplicity two on a projective line $l$ in $\p3$ defined by an 
exact sequence
\begin{equation}\label{Ferrand k}
0\to\cali_{C,\p3}\to\cali_{l,\p3}\to\calo_l(m)\to0,\ \ \ \ \ \ \ 
-1\le m=2-k\le2.
\end{equation}
(iv) The moduli spaces $\calr_k^u$ are varieties of dimensions
\begin{equation}\label{dim Rk}
\dim\calr_0^u=\dim\calr_3^u=14,\ \ \ \ \dim\calr_1^u=\dim\calr_2^u=13,
\end{equation}
and they are fine.
\end{Prop}
\begin{proof}
(i)-(iii). By Lemma \ref{sheaf E2}, we have $c_1(\calf)=0$, $c_2(\calf)
=2$. Since $\calf$ is unstable, it follows from \cite[Lemma 3.1]
{Hartreflexive} that $H^0(\calf)\ne0$. Besides, from
\eqref{h0(calg(-1))=0} and the triple \eqref{E2 and Q} twisted by 
$\op3(-1)$ we obtain $H^0(\calf(-1))=0$. Take a section $0\ne s\in 
H^0(\calf)$ and define a subscheme $C$ in $\p3$ by the ideal sheaf 
$\cali_{C,\p3}=\mathrm{im}(u:\calf\xrightarrow[\simeq]{\mathrm{can}}
\calf^{\vee}\xrightarrow{s^{\vee}}\op3)$. (The canonical isomorphism 
$\mathrm{can}:\ \calf\xrightarrow{\simeq}\calf^{\vee}$ follows since 
$c_1(\calf)=0.$) From the equality $H^0(\calf(-1))=0$, by 
\cite[Thm 4.1]{Hartreflexive} we obtain that: \\
(a) $C$ is a Cohen-Macaulay curve in $\p3$ satisfying the triple 
\eqref{tripl for F}, so that $\deg C=c_2(\calf)=2$, and\\
(b) the triple \eqref{tripl for F} is exact, and the equality $\chi(
\calo_C)=4-k$ follows from this triple and \cite[
Thm. 2.3]{Hartreflexive}; moreover, \eqref{tripl for F} defines an 
extension
\begin{equation}\label{extension1}
\xi\in\Ext^1(\cali_{C,\p3},\op3)\simeq H^0(\cale xt^1(\cali_{C,\p3},
\op3))\simeq H^0(\cale xt^2(\calo_C,\op3)).
\end{equation} 
(Here we use standard isomorphisms relatig global $\Ext$-groups and
$\cale xt$-sheaves - see \cite[Sec. 4]{Hartreflexive}.) If $C$ is a 
reduced curve then, since $\deg C=2$, $C$ is either a disjoint union 
$l_1\sqcup l_2$ of lines, or a conic. If $C$ is nonreduced, then $C$ 
is the scheme structure of multiplicity two on a projective line $l$ 
(in the sense of \cite[Definition on p. 58]{Ch}). Moreover, since $C$ 
is Cohen-Macaulay, the sheaf $\cali_{l,\p3}/\cali_{C,\p3}$ has no 
0-dimensional torsion. Hence, by \cite[Claim on p. 59]{Ch}, the exact
triple \eqref{Ferrand k} follows and, moreover, $C$ is a locally 
complete intersection. The triples \eqref{Ferrand k} and 
\eqref{tripl for F} yield the equality $m=2-\frac{1}{2}c_3(\calf)$=2-k.
Furthermore, \eqref{Ferrand k} and the isomorphism
\begin{equation}\label{conormal}
\cali_{l,\p3}|_l\simeq N_{l/\p3}^{\vee}\simeq\calo_l(-1)^{\oplus2}
\end{equation}
imply $m\ge-1$. Besides, $2-m=k=\frac{1}{2}c_3(\calf)\ge0$, as $\calf$ 
is reflexive. Thus, $-1\le m\le2$ and therefore $0\le k\le3$.\\ 
(iv) Consider the varieties $\calc_k=\{C\ |\ C\ \textrm{is a l.c.i. 
curve of degree 2 in}\ \p3, \chi(\calo_C)=4-k\}$, $0\le k\le3$. From 
(i)-(iii) and \cite[Remark 1.3]{Ch} it follows that $\calc_k$ are
rational varieties of dimensions
\begin{equation}\label{dim Ck}
\dim\calc_0=11,\ \ \ \ \dim\calc_1=9,\ \ \ \ \dim\calc_2=\dim\calc_3=8.
\end{equation} 
Note that \eqref{Ferrand k} yields an exact triple $0\to\calo_l(2-k)\to
\calo_C\to\calo_l\to0$,\ $k=\frac{1}{2}c_3(\calf)$. 
Applying to it the functor $\cale xt^2(\calo_l,
\op3)$ and using the relations $\cale xt^2(\calo_l,\op3)\simeq\det(N
_{l/\p3})\simeq\calo_l(2)$, $\cale xt^i(\calo_l,\op3)=0,\ i=1,3,$ (see 
\cite[pp. 49-50]{okonek}) we obtain an exact triple $0\to\calo_l(2)\to
\cale xt^2(\calo_C,\op3)\xrightarrow{\epsilon}\calo_l(k)\to0$ which, 
together with \eqref{extension1}, yields
\begin{equation}\label{dim Ext1}
\begin{split}
&\dim\Hom(\cali_{C,\p3},\op3)=1,\ \Ext^{\ge2}(\cali_{C,\p3},\op3)=0,\\
& \dim\Ext^1(\cali_{C,\p3},\op3)=h^0(\cale xt^2(\calo_C,\op3)=k+4,
\end{split}
\end{equation}
Now, by (i)-(iii), for $0\le k\le3$, the spaces $\calr_k^u$ are 
described as:
$\calr_k^u=\{([\calf],\langle\xi\rangle)\ |\ [\calf]\in\calr_k,\ 
\calf\ \textrm{fits in}\ \eqref{tripl for F},\ \langle\xi\rangle\in
\mathbb{P}(\Ext^1(\cali_{C,\p3},\op3))\}=\{(C,\langle\xi\rangle)\ |\ 
C\in\calc_k,\ \langle\xi\rangle\in\mathbb{P}(\Ext^1(\cali_{C,\p3},\op3
))\}$.
This, together with \eqref{dim Ext1}, shows that $\calr_k$ is a 
projective fibration with fibre $\mathbb{P}^{k+3}$ over $\calc_k$,
and \eqref{dim Ck} yields \eqref{dim Rk}. Note that there exist 
universal flat families $\Gamma\subset\boldsymbol{\calc}_k$ of curves 
$C$, and in view of \eqref{dim Ext1} and \cite[Thm. 1.4]{L} the 
sheaves $\lext^i_{p_{\calc_k}}(\cali_{\Gamma,\boldsymbol{\calc}_k},
\calo_{\boldsymbol{\calc}_k})$ commute with the base change. Hence, by
\cite[Prop. 4.2]{L} there exist universal sheaves $\boldsymbol{\calf}$ 
on $\boldsymbol{\calr}_k^u$, i. e., $\calr_k^u$ are fine moduli spaces.
\end{proof}

\begin{Prop}\label{Prop F stable}
Suppose that $[\calf]\in\calr_k^s$. Then the following statements 
hold.\\ 
(i) $\calr_k^s\ne\emptyset$ only for $0\le k\le2$.\\
(ii) $\dim\calr_k^s=13,\ k=0,1,2$.\\
(iii) For $0\le k\le2$ and any $[\calf]\in\calr_k^s$, 
$\dim\Ext^1(\calf,\calf)=13,\ \ 
\Ext^2(\calf,\calf)=0$.\\
(iv) For any $\p2\subset\p3$, $h^0(\calf_{\p2}(2))=10$, $ h^1(\calf
_{\p2}(2))=0$.
\end{Prop}
\begin{proof}
Statements (i)-(iii) are proved in \cite[Sec. 2]{Ch}. The equalities in
statement
(iv) follow from the exact triple $0\to\calf(1)\to\calf(2)\to\calf
_{\p2}(2)\to0$ and \cite[Tables 2.8.1 and 2.12.2]{Ch} for $k=2,4$ and, 
respectively, from \cite[\S9]{Hart1} for $k=0$.
\end{proof}

We next proceed to a detailed description of the relation between the 
spaces $\calh_k^u$ and $\calr_k^u$ for $0\le k\le3$ and the spaces 
$\calh_k^s$ and $\calr_k^s$ for $0\le k\le2$ given by steps 1 and 2 of 
formula \eqref{two transfns}. Denote
$$
{}^2S_{\calf}^u:=\{\p2\in\check{\mathbb{P}}^3\ |\ \dim(C\cap\p2)=1,\ 
C\subset\p2,\ \calo_C=\op3/(\calf/\op3)\},\ \ [\calf]\in\calr_3^u,\ \ 
\ \ \
$$
$$
{}^2S_{\calf}^u:=\emptyset,\ \ [\calf]\in\calr_k^u,\ \ \ k\le2,
\ \ \ \ \ \ \ \ \ \ \ \ \ \ \ \ \ \ \ \ \ \ \ \ \ \ \ \ \ \ 
\ \ \ \ \ \ \ \ \ \ \ \ \ \ \ \ \ \ \ \ \ \ \ \ \ \ \ \ \ \ \ \ \ \ \ \ \ \ \ 
$$
$$
{}^1S_{\calf}^u:=\{\p2\in\check{\mathbb{P}}^3\ |\ \dim(C\cap\p2)=1,\ 
\calo_C=\op3/(\calf/\op3),\ C\not\subset\p2\},\ \ 
[\calf]\in\calr_k^u,\ \ \ \ \
$$
$$
{}^2S_{\calf}^s={}^1S_{\calf}^s:=\emptyset,\ \ \ \text{if}\ 
[\calf]\in\calr_k^s,
\ \ \ \ \ \ \ \ \ \ \ \ \ \ \ \ \ \ \ \ \ \ \ \ \ \ \ \ \ \ 
\ \ \ \ \ \ \ \ \ \ \ \ \ \ \ \ \ \ \ \ \ \ \ \ \ \ \ \ \ \ \ \ \ \ \ \ \ \ \ 
$$
$$
{}^0S_{\calf}^u:=\{\p2\in\check{\mathbb{P}}^3\smallsetminus({}^1S
_{\calf}^u\cup{}^2S_{\calf}^u)\ |\ \mathrm{Sing}\calf\cap\p2
\ne\emptyset\},\ \ [\calf]\in\calr_k^u,\ \ \ \ \ \ \ \ \ \ \ \ \ 
\ \ \ \ \ \ \ \ \ \ \ \ \ \  
$$
$$
{}^0S_{\calf}^s:=\{\p2\in\check{\mathbb{P}}^3\smallsetminus({}^1S
_{\calf}^s\cup{}^2S_{\calf}^s)\ |\ \mathrm{Sing}\calf\cap\p2\ne
\emptyset\},\ \ [\calf]\in\calr_k^s,\ \ \
\ \ \ \ \ \ \ \ \ \ \ \ \ \ \ \ \ \ \ \ \ \ \ \ 
$$
$$
{}^{-1}S_{\calf}^u:=\check{\mathbb{P}}^3\smallsetminus({}^0
S_{\calf}^u
\cup{}^1S_{\calf}^u\cup{}^2S_{\calf}^u),\ \ \ \ 
\ \ \ \ \ \ \ \ \ \ \ \ \ \ \ \ \ \ \ \ \ \ \ \ \ \ \ \ \ \ \
\ \ \ \ \ \ \ \ \ \ \ \ \ \ \ \ \ \ \ \ \ \ \ \ \ \ \ \ \ \ \ \ 
$$
$$
{}^{-1}S_{\calf}^s:=\check{\mathbb{P}}^3\smallsetminus({}^0
S_{\calf}^s\cup{}^1S_{\calf}^s\cup{}^2S_{\calf}^s),\ \ \ \ 
\ \ \ \ \ \ \ \ \ \ \ \ \ \ \ \ \ \ \ \ \ \ \ \ \ \ \ \ \ \ \
\ \ \ \ \ \ \ \ \ \ \ \ \ \ \ \ \ \ \ \ \ \ \ \ \ \ \ \ \ \ \ \ 
$$
$$
\cald_k^u:=\calr_k^u\times\check{\mathbb{P}}^3,\ \ \ \ \ 
{}^i\cald_k^u:=\{([\calf],\p2)\in\cald_k^u\ |\ \p2\in{}^iS_{
\calf}^u\},\ \ \ \ \ -1\le i\le2,\ \ \ \ \ \ \ \ \ \ \ \ \ \ \ 
$$
$$
\cald_k^s:=\calr_k^s\times\check{\mathbb{P}}^3,\ \ \ \ \ 
{}^i\cald_k^s:=\{([\calf],\p2)\in\cald_k^s\ |\ \p2\in{}^i
S_{\calf}^s\},\ \ \ \ \ -1\le i\le2,\ \ \ \ \ \ \ \ \ \ \ \ \ \ \
$$
$$
\cald_k:=\cald_k^u\sqcup\cald_k^s=\underset{-1\le i\le2}{\sqcup}
{}^i\cald_k,\ \ \mathrm{where}\ \ {}^i\cald_k:={}^i\cald_k^u
\sqcup{}^i\cald_k^s,\ \ \  -1\le i\le2. \ \ \ \ \ \ \ \ \ \ \ 
\ \ \ 
$$
Clearly, ${}^i\cald_k^u$ (respectively, ${}^i\cald_k^s$) are 
locally closed in $\cald_k^u$ (respectively, $\cald_k^s$) and
\begin{equation}\label{Duk,Dsk}
\cald_k^u=\underset{-1\le i\le2}{\bigsqcup}{}^i\cald_k^u,\ \ \ \ 
\ \ \ \cald_k^s=\underset{-1\le i\le2}{\bigsqcup}{}^i\cald_k^s,\ 
\ \  \ \ \ \ \ \ \ \ \ \ \  \ \ \ \ \ \ \ \ \ \ \  \ \ \ \ \ \ \ 
\ \ \ \ \ \ \ 
\end{equation}
\begin{equation}\label{dim Diuk,Disk}
\dim{}^i\cald_k^u\le\dim\calr_k^u+2-i,\ \ \ 
\dim{}^i\cald_k^s\le\dim\calr_k^s+2-i,\ \ \ -1\le i\le2. 
\end{equation}
Next, denote
\begin{equation}\label{Def Pi}
\begin{split}
& \Ph(\calf,\p2):=\{\langle\rho\rangle\in\mathbb{P}(\Hom(\calf,
\op2(2)))|\mathrm{im}(\rho:\calf\to\op2(2))=\cali_{Z,\p2}(2)\\ 
& \mathrm{for\ a\ subscheme}\ Z\subset\p2,\ \dim Z\le0,\ 
\ell(Z)=6-k\},\ \ \  ([\calf],\p2)\in\cald_k, 
\end{split}\ \ \
\end{equation}

\vspace{1mm}
\begin{equation}\label{Def K,Q}
\begin{split}
& K_{([\calf],\p2)}:=\{\langle\rho\rangle\in\Ph(\calf,\p2)\ |\ 
[\call_{\rho}:=(\ker\rho)(1)]\in\calr(1,4,6)\ \mathrm{is\ stable}\},\\
& ([\calf],\p2)\in\cald_k,\calq_k:=\{([\calf],\p2,\langle\rho\rangle)\ 
|\ ([\calf],\p2)\in\cald_k,\ \langle\rho\rangle\in K_{([\calf],
\p2)}\},\\
& \calq_k\xrightarrow{p_{1k}}\cald_k\ \text{is the forgetful map},\ \ 
p_{1k}^{-1}([\calf],\p2)=K_{([\calf],\p2)},
\end{split}
\end{equation}

\vspace{1mm}
\begin{equation}\label{Def Sigma,T}
\begin{split}
& \Sigma_{([\calf],\p2,\langle\rho\rangle)}:=\{\langle\gamma\rangle\in
\mathbb{P}(\Hom(\call_{\rho},\op2(2)))\ |\ \gamma:\ \call_{\rho}\to
\op2(2)\ \text{is an} \\ 
& \text{epimorphism and}\ \ker\gamma\ \text{is locally free}\},\ 
\ \ \ \ \ ([\calf],\p2,\langle\rho\rangle)\in\calq_k,\\  
& T_k:=\{([\calf],\p2,\langle\rho\rangle,\langle\gamma\rangle)\ |\ 
([\calf],\p2,\langle\rho\rangle)\in\calq_k,\ \langle\gamma\rangle\in
\Sigma_{([\calf],\p2,\langle\rho\rangle)}\},\\
& T_k\xrightarrow{p_{2k}}\calq_k\ \text{is the forgetful map},\ \ 
p_{2k}^{-1}([\calf],\p2,\langle\rho\rangle)=\Sigma_{([\calf],\p2,
\langle\rho\rangle)},
\end{split}
\end{equation}

\vspace{1mm}
\begin{equation}\label{Def Qk,Tk}
\begin{split}
& \calq_k^u:=p_{1k}^{-1}(\cald_k^u\cap p_{1k}(\calq_k)),\ \ \  
T_k^u:=p_{2k}^{-1}(\calq_k^u\cap p_{2k}(T_k)),\\
& \calq_k^s:=p_{1k}^{-1}(\cald_k^s\cap p_{1k}(\calq_k)),\ \ \  
T_k^s:=p_{2k}^{-1}(\calq_k^s\cap p_{2k}(T_k)).
\end{split}
\end{equation}
(Here $0\le k\le3$ and  $0\le k\le2$ in unstable case and stable case,
respectively.) Since $\cald_k=\cald_k^u\sqcup\cald_k^s$, it follows 
that
\begin{equation}\label{Qk,Tk}
\calq_k=\calq_k^u\sqcup\calq_k^s,\ \ \ \ \ \ T_k=T_k^u\sqcup T_k^s,\ \ \ \ \ \ \ \ 0\le k\le3.
\end{equation}
(For consistency, in \eqref{Qk,Tk} and below we set $\calq_3^s=T_3^s=
\emptyset$.)
Since the stability of the sheaf $\call_{\rho}$ is an open property in 
flat families \cite[Prop. 2.3.1]{HL} it follows that 
\begin{equation}\label{open subset K}
p_{1k}^{-1}([\calf],\p2)=K_{([\calf],\p2)}
\xymatrix{\ar@{^{(}->}[r]^{\textrm{open}} &}
\Ph(\calf,\p2),\ \ \ 
\ \ \ \ \ \ \ \ ([\calf],\p2)\in\cald_k.
\end{equation}
Take any point $([\calf],\p2,\langle\rho\rangle,\langle\gamma\rangle)$ 
Since by definition $[\call_{\rho}]\in\calr(1,4,6)$ is stable and 
$\cale=\ker\gamma$ is a vector bundle, we obtain from the second 
triple \eqref{two transfns} that $[\cale]\in\calr(0,5,0)$ is also 
stable, i. e. $[\cale]\in\calb(5)$. Thus, we obtain a natural map
\begin{equation}\label{map f_k}
f_k:\ T_k\to\calb(5),\ \ \  ([\calf],\p2,\langle\rho\rangle,\langle
\gamma\rangle)\mapsto[\ker\gamma]
\end{equation}
and by Proposition \ref{two steps}, $\calh_k^u\subset f_k(T_k^u),\ 
\calh_k^s\subset f_k(T_k^s)$. This, together with \eqref{cup Hk us} and
and the second formula \eqref{Qk,Tk} yields
\begin{equation}\label{H in}
\calh\smallsetminus(\calh\cap\calg(2,1))\subset\underset{0\le 
k\le3}{\bigsqcup}f_k(T_k).
\end{equation}
It will follow from computations below that $T_k$ are disjoint unions 
of schemes and $f_k$ are morphisms for each of these schemes and all 
admissible values of $k$.

\begin{Lema}\label{aux}
Let $([\calf],\p2)\in\cald_k$ and $\Ph(\calf,\p2)\ne\emptyset$. Then:\\
(i) there is an open embedding $j:\ \Ph(\calf,\p2)\xymatrix{\ar@
{^{(}->}[r]^{\mathrm{open}} &}\mathbb{P}(H^0((\calf_{\p2})^{\vee\vee}(2
)))$ and for any $\langle\rho\rangle\in\Ph(\calf,\p2)$ there exists a 
subscheme $W(\rho)$ of $\p2$, $\dim W(\rho)\le0$, and an exact triple
\begin{equation}\label{triple W}
0\to\calf_{\p2}(2)\xrightarrow{\mathrm{can}}(\calf_{\p2})^{\vee\vee}
(2)\to\calo_{W(\rho)}\to0,\ \ \ \ \ \ \ \ \ \ell_{W(\rho)}=k;
\end{equation}
(ii) if $\Sigma_{([\calf],\p2,\langle\rho\rangle)}\ne\emptyset$ for 
$([\calf],\p2,\langle\rho\rangle)\in\calq_k$, then there is an open 
embedding $\Sigma_{([\calf],\p2,\langle\rho\rangle)}\xymatrix{
\ar@{^{(}->}[r]^{\mathrm{open}}&}\mathbb{P}(\Hom(\call,\op2(2)))\simeq
\mathbb{P}^{10}$;\\
(iii) if $([\calf],\p2)\in{}^{-1}\cald_k$, then 
$k=0$, $h^0(\calf_{\p2}^{\vee\vee}(2))=h^0(\calf_{\p2}(2))=10$
;\\
(iv) if ${}^0\cald_k\ne\emptyset$ and $([\calf],\p2)\in{}^0\cald_k$, 
then $1\le k\le2$, $h^0(\calf_{\p2}(2))=10$
and 
\begin{equation}\label{h0 for k=1,2,3}
h^0((\calf_{\p2})^{\vee\vee}(2))=10+k 
;
\end{equation}
(v) if ${}^1\cald_k\cup{}^2\cald_k\ne\emptyset$ and $([\calf],\p2)\in
{}^1\cald_k\cup{}^2\cald_k$, then the equalities 
\eqref{h0 for k=1,2,3} hold for $k=1,2,3$ and
\begin{equation}\label{h0 for k=0}
h^0((\calf_{\p2})^{\vee\vee}(2))=h^0((\calf_{\p2})(2))=11,\ \ \  
\text{if}\ \ \ k=0.\\
\end{equation} 
\end{Lema}
\begin{proof}
(i) Take any $\langle\rho\rangle\in\Ph(\calf,\p2)$. By the 
definition of $\Ph(\calf,\p2)$ we may consider $\rho$ as a composition 
$\rho:\calf\xrightarrow{\otimes\op2}\calf_{\p2}\xrightarrow{\bar{\rho}}
\cali_{Z,\p2}(2)$ with $\dim Z\le0$, $\ell_Z=6-k$. As $\calf$ is 
reflexive, $\calf_{\p2}$ has no torsion as a $\op2$-sheaf 
\cite[\S1]{Hartreflexive}. Hence, $\ker\bar{\rho}$ is a rank 1 torsion 
free $\op2$-sheaf. Since $c_1(\calf_{\p2})=0$, $c_2(\calf_{\p2})=2$, 
$\ker\bar{\rho}\simeq\cali_{W,\p2}(-2)$, where $\dim W\le0$, and there 
is an exact triple
\begin{equation}\label{tr with W}
0\to\cali_{W,\p2}\xrightarrow{\theta_{\rho}}\calf_{\p2}(2)\xrightarrow{
\bar{\rho}}\cali_{Z,\p2}(4)\to0.\ \ \ \ \ 
\end{equation}
The monomorphism $\theta=\theta_{\rho}$ in this triple extends to a
commutative square
\begin{equation}\label{diag 01}
\xymatrix{
\cali_{W,\p2}\ar@{>->}[r]^-{\theta_{\rho}}\ar@{>->}[d]_-{\mathrm{can}} 
& \calf_{\p2}(2)\ar@{>->}[d]^-{\mathrm{can}} &\\
\op2\ar@{>->}[r]^-{\theta_{\rho}^{\vee\vee}} & 
(\calf_{\p2})^{\vee\vee}(2),}
\end{equation} 
and we obtain a morphism $j:\Ph(\calf,\p2)\to\mathbb{P}(H^0((\calf_{\p2
})^{\vee\vee}(2))),\ \langle\rho\rangle\mapsto\theta_{\rho}^{\vee\vee}
$. To construct the inverse to $j$ morphism $\psi:(\mathrm{im}j)\to\Ph(
\calf,\p2)$, take any $(\tilde{\theta}:\op2\to(\calf_{\p2})^{\vee\vee}
(2))\in\mathrm{im}j$. The morphism $\theta:\cali_{W,\p2}\to\calf_{\p2}
(2)$ such that $\tilde{\theta}=\theta^{\vee\vee}$ is recovered 
from $\tilde{\theta}$ as $\tilde{\theta}|_{\cali_{W,\p2}}$, where 
$\cali_{W,\p2}=\mathrm{can}^{-1}(\tilde{\theta}(\op2)\cap\mathrm{can}
(\calf_{\p2}(2)))$. Then $\bar{\theta}$ defines via $\theta$ a 
morphism $\bar{\rho}$ as the quotient morphism $\calf_{\p2}(2)\to
\mathrm{\coker}\theta\simeq\cali_{Z,\p2}(4)$, and we set $\psi(\langle
\bar{\theta}\rangle):=\langle\bar{\rho}\circ(-\otimes\op2)\rangle$. 
The openness of $j$ follows from the openness of the condition
$\rho:\calf\to\op2(2)$ to be surjective.

Next, remark that, in \eqref{diag 01}, the $\op2$-sheaf $\mathrm{
\coker}\theta\simeq\cali_{Z,\p2}(4)$ has no torsion, hence there is no
nonzero morphism $\calo_W=\op2/\cali_{W,\p2}\to\mathrm{\coker}\theta$,
since $\dim W\le0$. Thus, \eqref{diag 01} and the Snake Lemma yield an 
exact triple \eqref{triple W} with $W(\rho)=W$.

\vspace{2mm}
(ii) The injection $\Sigma_{([\calf],\p2,\langle\rho\rangle)}
\hookrightarrow\mathbb{P}(\Hom(\call_{\rho},\op2(2)))\simeq\mathbb{P}
^{10}$ is an open embedding since the condition that $\gamma:\call:=
\call_{\rho}\to\op2(2)$ is an epimorphism and $\ker\gamma$ is 
locally free is open on $\langle\gamma\rangle\in\mathbb{P}(\Hom(\call,
\op2(2)))$. We thus have to show that $\dim\Hom(\call,\op2(2))=11$. 
Consider the epimorphism $\bar{\gamma}=\gamma|_{\p2}:\call_{\p2}
\twoheadrightarrow\op2(2)$. Since by definition $[\call]\in\calr(1,4,6)
$, it follows that $\ker\bar{\gamma}\simeq\cali_{Y,\p2}(-1)$ for some 
subscheme $Y$ of $\p2$, $\dim Y=0$, $\ell_Y=6$. This yields an exact 
triple $0\to\cali_{Y,\p2}(-1)\to\call_{\p2}\xrightarrow{\bar{\gamma}}
\op2(2)\to0$. Applying to it the functor $\lext^{\bullet}_{\op2}(-,\op2
(2))$ we obtain an exact triple $0\to\op2\to\lhom(\call_{\p2},\op2(2))
\to\op2(3)\to0$ which implies $\dim\Hom(\call,\op2(2))=\dim\Hom(\call
_{\p2},\op2(2))=h^0(\lhom(\call_{\p2},\op2(2)))=11$.

\vspace{2mm}
(iii) Since $([\calf],\p2)\in{}^{-1}\cald_k$, 
$(\calf_{\p2})^{\vee\vee}\simeq\calf_{\p2}$ is locally free 
$\op2$-sheaf, and \eqref{triple W} implies $k=0$. Now, if $\calf$ is unstable, then applying to \eqref{tripl for F} the functor $-\otimes
\op2(2)$ we have an exact triple
\begin{equation}\label{otimes op2(2)}
0\to\op2\to\calf_{\p2}\to\cali_{Y,\p2}(2)\to0,\ \ \ \ \  \dim Y=0,\ 
\ \ \ \ \ \ell_Y=\deg C=2, 
\end{equation}
and this triple yields the desired values of $h^i((\calf_{\p2})^{\vee
\vee}(2))=h^i(\calf_{\p2}(2))$. If $\calf$ is stable, then these 
values are given by Proposition \ref{Prop F stable}.(iv).

\vspace{2mm}
(iv) Since $([\calf],\p2)\in{}^0\cald_k\cup{}^2\cald_k\ne\emptyset$, 
the morphism $\mathrm{can}$ in \eqref{triple W} is not an isomorphism, 
hence $k=\ell_{W(\rho)}\ge1$. On the other hand, $k\le3$ by 
Propositions \ref{Prop F stable}(i) and \ref{Prop F unstable}(i). As 
above, if $\calf$ is unstable, the triple \eqref{otimes op2(2)} is 
true, which yields the equalities $h^0(\calf_{\p2}(2))=10$, $h^1(\calf
_{\p2}(2))=0$. Respectively, if $\calf$ is stable, these equalities 
follow from Proposition \ref{Prop F stable}.(iv). Whence, by 
\eqref{triple W}, we have \eqref{h0 for k=1,2,3}. 

We only have to show that, in case $\calf$ is unstable, $k\le2$. 
By the definition of the sets ${}^i\cald_k^u,\ i=0,1$, the condition 
$([\calf],\p2)\in{}^0\cald_k^u$ implies that $\p2\not\in{}^1S_{\calf}
^u$. This means that the exact triple \eqref{otimes op2(2)} is true, 
with $\dim Y=0,\ \ell_Y=\deg C=2$. Dualizing this $\op2$-triple we 
easily obtain an inequality $h^0(\lext^1(\calf_{\p2},\op2))\le h^0(
\lext^2(\calo_Y,\op2))=\ell_Y=2$ and an exact tripe $0\to\op2\to(\calf
_{\p2})^{\vee}\to\cali_{Z,\p2}(2)\to0$ for some scheme $Z\subset\p2$ 
with $\dim Z\le0$, $\ell_Z=2-h^0(\lext^1(\calf_{\p2},\op2))$. This 
triple, together with the triple \eqref{otimes op2(2)} and the 
isomorphism $(\calf_{\p2})^{\vee}\simeq(\calf_{\p2})^{\vee\vee}$, 
yields an exact triple $0\to\calf_{\p2}(2)\xrightarrow{\mathrm{can}}(
\calf_{\p2})^{\vee\vee}(2)\to K\to0$, where $K$ is an artinian sheaf 
of length $h^0(K)=h^0(\lext^1(\calf_{\p2},\op2))\le2$. Comparing this 
triple with \eqref{triple W} we obtain $K\simeq\calo_{W(\rho)}$ and
$k=\ell_{W(\rho)}\le2$.

\vspace{2mm}
(v) From the condition $([\calf],\p2)\in{}^1\cald_k\ne\emptyset$ and 
Proposition \ref{Prop F unstable} it follows that 
$\p2\cap\mathrm{Sing}\calf=l$ is a
line, if $k=0,1,2$; respectively, $\p2\cap\mathrm{Sing}\calf=C$ 
is a conic, if $k=3$. Thus, applying to the triples \eqref{tripl for 
F} and \eqref{Ferrand k} the functor $-\otimes\op2(2)$ and using the 
resolution $0\to\op3(-1)\to\op3\to\op2\to0$, we obtain the following 
exact triples, where $\dim W=0$ and if $k=0$ or 1, then $W\subset l$:
\begin{equation}\label{I restr}
\begin{split}
& 0\to\op2(2)\to\calf_{\p2}(2)\to\cali_{C,\p3}(2)|_{\p2}\to0,\\
& 0\to\calo_l(3-k)\to\cali_{C,\p3}(2)|_{\p2}\to\cali_{W,\p2}(1)
\to0,\ \ell_W=3-k,\ k\le2,\\
& 0\to\calo_l\to\cali_{C,\p3}(2)|_{\p2}\to\op2(1)\to0,\ k=3,\ \ 
C\not\subset\p2,\ \ \ \p2\cap C=l,\\
& 0\to\calo_C(1)\to\cali_{C,\p3}(2)|_{\p2}\to\op2\to0,\ k=3,\ \ \ 
C\subset\p2.
\end{split}
\end{equation}
Since $\calf$ is locally free for $k=0$, $h^i((\calf_{\p2})
^{\vee\vee}(2))=h^i(\calf_{\p2}(2))$, from \eqref{I restr} we obtain 
\eqref{h0 for k=0}. Respectively, for $k=1,2,3$, \eqref{I restr} and
\eqref{triple W} imply \eqref{h0 for k=1,2,3}.
\end{proof}

For $0\le k\le3$, let $B\subset\p3\times\check{\mathbb{P}}^3$ be the 
graph of incidence,$\calo_B(2)=\op3(2)\boxtimes\calo_{\check{\mathbb{P
}}^3}|_B$, and let $pr_0:\boldsymbol{\cald}_k^u\to\cald_k^u$, $pr_1:
\boldsymbol{\cald}_k^u\to\boldsymbol{\calr}_k^u$, $\boldsymbol{\cald}
_k^u\to\p3\times\check{\mathbb{P}}^3$ be the projections. For each 
$m\ge0$ consider the set 
\begin{equation}\label{Def Yu}
Y=Y_{k,m}^u:=\{([\calf],\p2)\in\cald_k^u\ |\ \dim\Hom(\calf,\op2(2))
=m\},\ \ \ \ m\ge0,
\end{equation}
and set $\mathbf{Y}=pr_0^{-1}(Y)$, $q_i=pr_i|_{\mathbf{Y}}$, $i=0,1,2
$, $L=\lext_{q_0}(q_1^*\boldsymbol{\calf},q_2^*\calo_B(2))$, where
$\boldsymbol{\calf}$ is the universal sheaf on $\boldsymbol{\calr}_k^u
$ 
which exists by Proposition \ref{Prop F unstable}.(iv), $\caly=
\mathbf{Y}\times_Y\mathbf{P}(L^{\vee})$, and let $\mathbf{P}(L^{\vee})
\xleftarrow{\lambda}\caly\xrightarrow{\pi}\mathbf{Y}$ and $\caly
\xrightarrow{\mu}\mathbf{P}(L^{\vee})\xrightarrow{\nu}Y$ be the 
projections. By \cite [Satz 3]{BPS}, $Y=Y_{k,m}^u$ is locally closed 
in $\cald_k^u$ and the sheaf $L$ is a rank $m$ locally free sheaf on 
$Y$ which commutes with the base change, i. e., for $y=([\calf],\p2)\in 
Y$, one has $L|_y=\Hom(\calf,\op2(2))$. On $\caly$ there is a 
universal morphism $\boldsymbol{\rho}:\ (q_1\circ\pi)^*\boldsymbol{
\calf}\to(q_2\circ\pi)^*\calo_B(2)\otimes\mu^*\calo_{\mathbf{P}(L^{
\vee})}(1)$. Consider the set
\begin{equation}\label{Def Xukm}
X=X_{k,m}^u:=\{([\calf],\p2)\in Y_{k,m}^u\ |\ K_{([\calf],\p2)}\ne
\emptyset\}.
\end{equation}
From this definition it follows that the sheaf $\mathrm{im}
\boldsymbol{\rho}$ is flat over $\mathbf{P}(L^{\vee})$ at any point 
$x\in\nu^{-1}(X)$. This implies that $X$ is an open (possibly, empty) 
subset of $Y$, hence it is locally closed in $\cald_k^u$. Therefore, 
since in view of Proposition \ref{Prop F unstable}.(iv) $\cald_k^u$ 
are varieties, the set $\Phi_k^u=\{m\in\mathbb{Z}_{\ge0}\ |\ X_{k,m}^u
\ne\emptyset\}$ is finite. By the definitions 
\eqref{Def Pi}, \eqref{Def K,Q}, \eqref{Def Qk,Tk} and 
\eqref{Def Xukm} we have
\begin{equation}\label{Def Xku}
X_k^u:=\underset{m\in\Phi_k^u}{\bigsqcup}X_{k,m}^u=p_{1k}(\calq_k^u),
\ \ \ \ \ \ \ \ \ \ \calq_k^u=p_{1k}^{-1}(X_k^u).
\end{equation}
Denoting ${}^iX_k^u={}^i\cald_k^u\cap X_k^u$, ${}^i\calq_k^u=p_{1k}
^{-1}({}^iX_k^u)$, $-1\le i\le2$, we find from the first equality 
\eqref{Duk,Dsk} that
\begin{equation}\label{Def Quk}
\calq_k^u=\underset{-1\le i\le2}{\bigsqcup}{}^i\calq_k^u.
\end{equation}
The inclusion \eqref{open subset K} and Lemma \ref{aux}.(i) yield that
the projection $p_{1k}:\ {}^i\calq_k^u\to{}^iX_k^u$ decomposes as 
\begin{equation}\label{descrn of Qiuk}
p_{1k}:\ {}^i\calq_k^u\xymatrix{\ar@{^{(}->}[r]^{\mathrm{open}} &}
{}^i\tilde{\calq}_k^u\xrightarrow{\tilde{p}_{1k}}{}^iX_k^u,\ \ \ \ \ \ 
\ \ \ \ -1\le i\le2,\ \ \ 0\le k\le3,
\end{equation} 
where ${}^i\tilde{\calq}_k^u\xrightarrow{\tilde{p}_{1k}}{}^iX_k^u$ is 
the projective fibration with fibre $\mathbb{P}(H^0((\calf_{\p2})^{\vee
\vee}(2)))$ over an arbitrary point $([\calf],\p2)\in{}^iX_k^u$. Here
by \eqref{Def Xku} each ${}^iX_k^u$ is a disjoint union of schemes,
This shows that each ${}^i\calq_k^u$ is a disjoint union of schemes.
Since ${}^iX_k^u{}\subset\cald_k^u$, it follows from
\eqref{dim Diuk,Disk} and Lemma \ref{aux}.(iii)-(v) that
\begin{equation}\label{dim Qiuk}
\dim{}^i\calq_k^u\le\dim{}^i\calr_k^u+11-i+k,\ \ \ \ \ \ \ \ \ \ -1\le 
i\le2,\ \ \ 0\le k\le3.
\end{equation}
Thus, in view of \eqref{dim Rk}, we obtain $\dim{}^i\calq_k^u\le26$ for
all possible $i,k$, hence \eqref{Def Quk} yields
\begin{equation}\label{dim Quk}
\dim\calq_k^u\le26,\ \ \ \ \ \ 0\le k\le3.
\end{equation}

To obtain a similar estimate for dimensions of $\calq_k^s$, we define 
similarly to \eqref{Def Yu} the locally closed subsets $Y_{k,m}^s:=
\{([\calf],\p2)\in\cald_k^s\ |\ \dim\Hom(\calf,\op2(2))=m\},\ m\ge0,$ 
of $\cald_k^s$. Next, note that apriori there is no universal sheaf 
$\boldsymbol{\calf}$ on $\boldsymbol{\calr}_k^s$. However, by 
Proposition \ref{Prop F stable}, $\Ext^2(\calf,\calf)=0$ for any 
$[\calf]\in\calr_k^s$, $0\le k\le2$. This means that the deformation 
theory for $\calr_k^s$ is unobstructed, so there exists an open cover 
$\calr_k^s=\underset{j\in J}{\bigcup}\calu_j$ and universal sheaves 
$\boldsymbol{\calf}_j$ on $\boldsymbol{\calu}_j$ (see, e. g., 
\cite[Appendice A1-A2]{BrH}, \cite[Ch. 6]{F}). The existence of these
local universal sheaves is enough to show that the sets $X_{k,m}^s$ 
defined similarly to \eqref{Def Xukm} as $X_{k,m}^s:=\{([\calf],
\p2)\in Y_{k,m}^s\ |\ K_{([\calf],\p2)}\ne\emptyset\}$, are locally 
closed subsets of $\cald_k^s$. We then have, similarly to
\eqref{Def Xku}-\eqref{Def Quk}, a finite dijoint unions of schemes
$X_k^s:=\underset{m\in\Phi_k^s}{\bigsqcup}X_{k,m}^s=p_{1k}(\calq_k^s)$
and relations $\calq_k^s=p_{1k}^{-1}(X_k^s)$. Denoting  ${}^iX_k^s=
{}^i\cald_k^s\cap X_k^s$, ${}^i\calq_k^s=p_{1k}^{-1}({}^iX_k^u)$, 
$-1\le i\le2$, and mimicking the argument in 
\eqref{Def Quk}-\eqref{dim Qiuk} with $u$ substituted by $s$, we 
obtain that ${}^i\calq_k^s$,
respectively, $\calq_k^s$ are disjoint unions of schemes satisfying 
the inequalities $\dim{}^i\calq_k^s\le\dim{}^i\calr_k^s+11-i+k$, 
$-1\le i\le2$, $0\le k\le2$. These formulas and Proposition \ref{Prop 
F stable}.(ii) imply the inequalities $\dim\calq_k^s\le26,\ 0\le 
k\le2$, which, together with \eqref{dim Quk} and \eqref{Qk,Tk} yield
\begin{equation}\label{dim Qk}
\dim\calq_k\le26,\ \ \ \ \ \ 0\le k\le3.
\end{equation}
(Remind that, as in \eqref{Qk,Tk}, we set $\calq_3^s=T_3^s=\emptyset$.)
Now from Lemma \ref{aux}.(ii) and the last formula in the display 
\eqref{Def Sigma,T},
similarly to \eqref{descrn of Qiuk}, we obtain that $T_k$ are disjoint
unions of schemes and projections $p_{2k}:\ T_k\to\calq_k$ are 
morphisms which decompose as $p_{2k}:\ T_k\xymatrix{\ar@{^{(}->}[r]^{
\mathrm{open}} &}\tilde{T}_k\xrightarrow{\tilde{p}_{2k}}\calq_k$, 
where $\tilde{T}_k\xrightarrow{\tilde{p}_{2k}}{}^i\calq_k$ is the 
projective fibration with fibre $\mathbb{P}(H^0((\calf_{\p2})^{\vee
\vee}(2)))\simeq\mathbb{P}^{10}$ over an arbitrary point $([\calf],\p2,
\langle\rho\rangle)\in\calq_k$, $0\le k\le3$. This together with 
\eqref{dim Qk} yields
\begin{equation}\label{dim Tk}
\dim T_k\le36,\ \ \ \ \ \ 0\le k\le3.
\end{equation}

\vspace{2mm}
\textit{Proof of Theorem \ref{Cal H}.} It is clear from the above that
the maps $f_k:\ T_k\to\calb(5)$ defined in \eqref{map f_k} are 
morphisms. The inequality $\dim(\calh\smallsetminus(\calg(2,1)
\cap\calh))\le36$ now follows from \eqref{H in} and 
\eqref{dim Tk}. However, by \cite[Remark 3.4.1]{Hartreflexive},
any irreducible component of $\calb(5)$ has dimension at least
37. Hence, Theorem \ref{Cal H} follows.
~\hfill$\Box$

\section{Components of $\mathcal{B}(5)$}\label{B(5)}

We finally have at hand all the results needed to complete 
the proof of our second main result, namely, the 
characterization of the irreducible components of 
$\mathcal{B}(5)$ given by Main Theorem \ref{MT2}. This 
entire section  will be devoted to this goal. 

\textit{Proof of Main Theorem \ref{MT2}.}
The first ingredient of the proof is the fact, proved by 
Hartshorne and Rao, that every bundle in $\mathcal{B}(5)$ is 
cohomology of one of the monads 
\eqref{invariant011}-\eqref{invariant131}, cf. \cite[Table 5.3, 
case 5.(1)-(4)]{hartshorne1991}.

Recall that for each stable rank 2 bundle $E$ on $\p3$ with 
vanishing first Chern class, the number 
$\alpha(E):=h^1(E(-2))~\mathrm{mod}~2$ is called the 
Atiyah--Rees $\alpha$-invariant of $E$, see \cite[Definition 
on p. 237]{Hart1}. Hartshorne showed \cite[Corollary 
2.4]{Hart1} that this number is invariant on the connected 
components of the moduli space of stable vector bundles on 
$\PP$. One can easily check that the cohomologies of monads of 
the form (\ref{invariant011}) and (\ref{invariant021}) have 
$\alpha$-invariant equal to 0, while the cohomologies of the 
other three types of monads have $\alpha$-invariant equal to 
1. 

Rao showed in \cite{Rao1987} that the family of bundles 
obtained as cohomology of monads of the form 
(\ref{invariant021}) is irreducible, of dimension 36, and it 
lies in a unique component of $\calb(5)$. Since instanton 
bundles of charge 5, i. e. the cohomologies of monads of the 
form (\ref{invariant011}), yield an irreducible family of 
dimension 37, it follows that the set
\begin{equation}\label{inst comp}
\cali := \{ [E]\in\calb(5) ~|~ \alpha(E)=0 \}
\end{equation}
forms a single irreducible component of $\mathcal{B}(5)$, of 
dimension 37, whose generic point corresponds to an 
instanton bundle. In addition, every $[E]\in\cali$ satisfies 
$\h^1(\lend(E))=37$; this was originaly proved by Katsylo 
and Ottaviani for instanton bundles \cite{KO}, and by Rao
for the cohomologies of monads of the form 
(\ref{invariant021}) \cite[Section 3]{Rao1987}. Therefore, 
we also conclude that $\cali$ is nonsingular. This completes 
the proof of the first part of the Main Theorem.

Our next step is to analyse those bundles with Atiyah--Rees 
invariant equal to 1.

Hartshorne proved in \cite[Theorem 9.9]{Hartreflexive} that 
the family $\calk$ of stable rank 2 bundles $E$ with $c_1(E)=0$ 
and $c_2(E)=5$ whose spectrum is $(-2,-1,0,1,2)$ is an 
irreducible, nonsigular family of dimension 40, and from the
definition of spectrum one has
\begin{equation}\label{h1=3}
h^1(\cale(-2))=3,\ \ \ \ \ \ \ \ \ [\cale]\in\calk.
\end{equation}
The bundles from $\calk$ are precisely those given as 
cohomologies of monads of the form (\ref{invariant111}), cf. 
\cite[Table 5.3, case 5.(4)]{hartshorne1991}, which is a 
particular case of a class of monads studied by Ein in 
\cite{Ein}. It is shown in \cite{Ein} that the closure 
$\overline{\calk}$ of $\calk$ in $\calb(5)$ is an irreducible 
component of $\mathcal{B}(5)$ of dimension 40.
  
We proved in Main Theorem \ref{MT1}, case $a=2$, that the 
bundles arising as cohomology of monads of the form 
(\ref{invariant121}) form a dense subset $\calg(2,1)$ of a 
rational irreducible component of dimension 37. Consider the set
$\calh$ of bundles arising as cohomology of monads of the form 
(\ref{invariant131}). Since the bundles from $\calg(2,1)\cup\calh
$ have the spectrum $(-1,0,0,0,1)$ by \cite[Table 5.3, case 
5.(2)]{hartshorne1991}, by we have (cf. \eqref{h1=0,1})
\begin{equation}\label{h1=1}
h^1(\cale(-2))=1,\ \ \ \ \ \ \ \ \ \ [\cale]\in\calg(2,1)\cup
\calh,
\end{equation}
so that $\alpha(\cale)=1$, and therefore, in view of \eqref{inst 
comp}, $\calh\cap\cali=\emptyset$. Since, by Theorem 
\ref{Cal H}, $\calh$ does not constitute a component in $\calb(5)
$, it then follows from the above that $\calh\subset\overline{
\calg(2,1)}\cup\overline{\calk}$.
\begin{Prop}
$\mathcal{H}\subset\overline{\mathcal{G}(2,1)}$ and 
$\overline{\calk}=\calk$.
\end{Prop}
\begin{proof}
We only have to show that $(\calg(2,1)\cup\calh)\cap\overline{
\calk}=\emptyset$. Suppose by contradiction that there exists a 
vector bundle $[\cale]\in(\calg(2,1)\cup\calh)\cap\overline{
\calk}$. By \eqref{h1=3} and the inferior semi-continuity of the 
dimension of the cohomology groups of coherent sheaves, one has 
that $h^1(\cale(-2))\geq 3$, contrary to \eqref{h1=1}.
\end{proof}
This last proposition finally concludes the proof of Main 
Theorem \ref{MT2}. ~\hfill$\Box$

We summarize all the information in the Main Theorem \ref{MT2}, 
and the discrete invariants of stable rank 2 bundles with 
$c_1=0$ and $c_2=5$ in the following table. 

\begin{table}[h]
\centering
\caption{Irreducible components of $\calb(5)$} 
\begin{tabular}{|c|c|c|c|c|}
\hline
\textbf{Component}                                           
& \textbf{Dimension}  & \textbf{Monads} & 
\textbf{Spectra}              & 
$\mathbf{\alpha}$-\textbf{invariant} \\ \hline
\multirow{2}{*}{\textbf{Instanton}}                          
& \multirow{2}{*}{37} & (\ref{invariant011}) & 
(0,0,0,0,0)                   & \multirow{2}{*}{0}   \\ 
\cline{3-4}
                                                             
&                     & (\ref{invariant021}) & 
(-1,-1,0,1,1)                 &                      \\ 
\hline
\textbf{Ein}                                                 
& 40                  & (\ref{invariant111}) & 
(-2,-1,0,1,2)                 & 1                    \\ 
\hline
\multirow{2}{*}{\textbf{\begin{tabular}[c]{@{}c@{}}Modified\\
Instanton\end{tabular}}} & \multirow{2}{*}{37} & 
(\ref{invariant121}) & \multirow{2}{*}{(-1,0,0,0,1)} & 
\multirow{2}{*}{1}   \\ \cline{3-3}
                                                             
&                     & (\ref{invariant131})   
&                               &                      \\ 
\hline
\end{tabular}
\end{table}

\textbf{Remark.} Inspired by the techniques introduced in the present 
paper, the authors of \cite{Tikhomirov3} construct another infinite 
series of irreducible components of $\mathcal{B}(0,n)$ whose special 
point corresponds to a bundle obtained as the cohomology of a monad 
similar to the one in display (24), just substituting a direct sum of 
two rank 2 instantons bundles for the rank 4 instanton bundle of 
charge 1 in middle term.

\end{document}